\documentclass[12pt,reqno]{amsart}
  
 \usepackage{comment}
\usepackage{latexsym}
\usepackage{graphicx}
\usepackage{amssymb}
 \usepackage[show]{ed}
\usepackage{color}

\renewcommand{\Re}{{\operatorname{Re}\,}}
\renewcommand{\Im}{{\operatorname{Im}\,}}
\newcommand{\Op}{\operatorname{Op}}
\newcommand{\sgn}{\operatorname{sgn}}

\renewcommand{\epsilon}{\varepsilon}

\newcommand{\R}{{\mathbb R}}
\newcommand{\C}{{\mathbb C}}

\newcommand{\Z}{{\mathbb Z}}

\newcommand{\lan}{\left\langle}
\newcommand{\ran}{\right\rangle}
\newcommand{\mc}[1]{\mathcal{#1}}
\newcommand{\e}{\epsilon}

\newcommand{\re}{\mathbb{R}}
\newcommand{\Cc}{C_c^\infty}
\newcommand{\Id}{\operatorname{Id}}

\newcommand{\dbar}{\bar\partial}

\newcommand{\supp}{{\operatorname{supp\,}}}
\renewcommand{\phi}{\varphi}

\newcommand{\ep}{\varepsilon}

\newtheorem{theo}{{\sc Theorem}}

\newtheorem{cor}{{\sc Corollary}}[section]

\newtheorem{lem}[cor]{{\sc Lemma}}

\newtheorem{prop}[cor]{{\sc Proposition}}
\numberwithin{equation}{section}

\newenvironment{rem}[1][]{\refstepcounter{cor}{\medskip\noindent{\it Remark~\thecor :\/#1} }}{\medskip}

\usepackage{hyperref}

\newcommand{\red}[1]{{#1}}

\title[Eigenfunction bounds]{Pointwise bounds for Steklov eigenfunctions}

\author{Jeffrey Galkowski}
\address{Department of Mathematics, Stanford University, Stanford, CA, USA}
\email{jeffrey.galkowski@stanford.edu }
\author{John A. Toth}
\address{Department of Mathematics and Statistics, McGill University, Montr\'eal, QC, Canada}
\email{jtoth@math.mcgill.ca} 
\date{}

\begin{document}

\begin{abstract}  
%Let $(\Omega,g)$ be a compact, real-analytic Riemannian manifold with real-analytic boundary $\partial \Omega.$ The harmonic extensions of the boundary Dirchlet-to-Neumann eigenfunctions are called {\em Steklov eigenfunctions}.  In Theorem \ref{mainthm} we show that asymptotically the Steklov eigenfuntions decay exponentially in terms of the Dirichlet-to-Neumann eigenvalues and give a sharp rate of decay to first order at the boundary. The proof uses the Poisson representation for the Steklov eigenfunctions combined with sharp $h$-microlocal concentration estimates (Theorems \ref{mainthm2} and \ref{mainthm3}) for the boundary Dirichlet-to-Neumann eigenfunctions near the cosphere bundle $S^*\partial \Omega.$

 Let $(\Omega,g)$ be a compact, real-analytic Riemannian manifold with real-analytic boundary $\partial \Omega.$ The harmonic extensions of the boundary Dirchlet-to-Neumann eigenfunctions are called {\em Steklov eigenfunctions}.  We show that the Steklov eigenfuntions decay exponentially into the interior in terms of the Dirichlet-to-Neumann eigenvalues and give a sharp rate of decay to first order at the boundary. The proof uses the Poisson representation for the Steklov eigenfunctions combined with sharp $h$-microlocal concentration estimates for the boundary Dirichlet-to-Neumann eigenfunctions near the cosphere bundle $S^*\partial \Omega.$ These estimates follow from sharp estimates on the concentration of the FBI transforms of solutions to analytic pseudodifferential equations $Pu=0$ near the characteristic set $\{\sigma(P)=0\}$. 
\end{abstract}

\maketitle
\section{Introduction} 

Let $(\Omega,g)$ be an $n$-dimensional, compact $C^{\infty}$ Riemannian manifold  with boundary $M$ and corresponding unit exterior normal $\nu$. By some abuse of notation, we also let $\nu$ denote a smooth vector field extension and $\gamma_{M}: C^0(\Omega) \to C^0(M)$ be the boundary restriction map.   Let ${\mathcal D}: C^{\infty}(M) \to C^{\infty}(M)$  be the associated Dirichlet-to-Neumann (DtN) operator defined by
\begin{equation} \label{DtN}
{\mathcal D} f :=  \gamma_{M} \partial_{\nu} u
\end{equation} 
where $u$ solves the Dirichlet problem

\begin{align} \label{dirichlet}
\Delta_{g} u(x) &= 0, \,\,\, x \in \Omega, \nonumber \\ 
u(q) &= f(q), \,\,\, q \in M. \end{align}\

%We consider the Poisson problem

%$$ \Delta_g u_{h}(x) = 0, \,\, x \in \Omega$$
%$$ \gamma_{\partial \Omega} u_h = \phi_h$$

The operator ${\mathcal D}$ is an ellptic, first order, self-adjoint pseudodifferential operator (see for example \cite[Section 7.11]{Tayl2}) with an $L^2$-normalized basis of eigenfunctions $\phi_{j}; j=1,2,....$ It is useful here to work in the semiclasscial setting from the outset. Choosing $h^{-1} \in \text{Spec} \, {\mathcal D},$ the corresponding eigenfunction $\phi_h$ then satisfies the semiclassical eigenfunction equation
$$ h {\mathcal D}\phi_h = \phi_h.$$
The harmonic extension, $u_{h} \in C^{\infty}(\Omega),$ of a DtN eigenfunction $\phi_h$ is called a {\em Steklov eigenfunction.}

There has been a substantial amount of recent work devoted to the study of the asymptotic behaviour of the DtN eigenvalues and both DtN and Steklov eigenfunctions, including the asymptotics of eigenfunction nodal sets (see for example \cite{BLin,GP,GPPS,HL,PST,Sh,SWZ, Ze,Zh, Zhu} and references therein). 
%A central feature of the latter involves estimating sup bounds of the Steklov eigenfunctions, a problem that we adress in Theorem \ref{mainthm}. 

For large eigenvalues, Steklov eigenfunctions possess both high oscillation inherited from the boundary DtN eigenfunctions and very sharp decay into the interior of $\Omega.$   As a consequence, even though Steklov eigenfunctions decay rapidly, the oscillation implies, in particular, that the nodal sets have intricate structure. It has been conjectured \cite{GP} that the analogue of Yau's conjecture \cite{Yau,Yau2}
 for nodal  volumes holds in the Steklov case. This was recently proved for real-analytic Riemann surfaces in \cite{PST}.
% In the $C^{\infty}$-case, an integration by parts argument \cite{HL} shows that $|u_{h}(x)| = O_N(h^N)$ for $x \in int(\Omega)$ and  $d(x,M) > \ep_0 >0.$
  
  The question of decay of Steklov eigenfunctions into the interior of $M$ when $(M,g)$ is real analytic was first raised by Hislop--Lutzer \cite{HL} where they conjecture that the Steklov eigenfunctions decay into the interior as $e^{-d(x,\partial\Omega)/h}.$  In the special case where dim $\Omega =2$ exponential decay with respect to $d(x,\partial\Omega)$ was indeed proved in \cite{PST}  and the eigenfunction decay is  a key feature in  their main results on nodal length. However, the analysis in \cite{PST} relies heavily on the assumption that the dimension equals two.
 
 %to prove that certain quasimodes are exponentially accurate from which bounds on the size of nodal sets were obtained. 

In Theorem \ref{mainthm}, we  prove an exponential decay result for  Steklov eigenfunctions for general real-analytic metrics in arbitrary dimension that is sharp to first order at the boundary, $\partial \Omega = M$,  and we bound the quadratic error in the decay rate  in terms of boundary curvature. In particular, we prove the conjecture due to Hislop--Lutzer \cite{HL}. 

One can heuristically view such exponential decay estimates for the Steklov eigenfunctions as describing the `tunnelling' of the boundary DtN eigenfunctions on $M$ into the interior of the manifold $\Omega.$ In the Schr\"odinger case $P(h)=-h^2\Delta_g+V-E$, one thinks of eigenfunctions as tunnelling from $V\leq E$ into the forbidden region $V>E$. Since the interior eigenfunctions are harmonic, this decay is somewhat more subtle in the case of Steklov eigenfunctions. The boundary data $\varphi_h$ concentrate microlocally on the cosphere bundle of the boundary, $S^*\partial\Omega$, while the fact that $-h^2\Delta_gu_h=0$ implies that $u_h$ must concentrate at the zero section {  $\{ (x,\xi) \in T^*\Omega;  \xi = 0 \}.$ } Thus, we think of $u_h$ as tunnelling from boundary values with $|\xi|_g=1$ to the interior where $\xi=0$. For this reason, it is reasonable to view the decay estimates in Theorem \ref{mainthm} as a natural analogue of the well-known Agmon-Lithner estimates (see for example \cite{HeSj}) for Schr\"{o}dinger eigenfunctions in classically forbidden regions. { We note that the assumption that $(\Omega,g)$ is $C^{\omega}$ is necessary in Theorem \ref{mainthm} below because real-analyticity allows for  accuracy up to exponential errors in $h$ in the pseudodifferential calculus, whereas in the $C^\infty$ case, one can only work to $O(h^{\infty})$-error. In particular, one can only microlocalize modulo such errors. Since Steklov eigenfunctions decay exponentially in $h$ in the interior of $\Omega$, the usual $C^{\infty}$ semiclassical calculus of operators is not accurate enough to deal with these functions in a rigorous fashion.  Similarily, our subsequent more general results in Theorems \ref{mainthm2} and \ref{mainthm3} hinge on the microlocal exponential weighted estimate in Proposition \ref{weightedl2} which also requires real-analyticity to effectively control error terms.}

\begin{theo} \label{mainthm}
Let { $(\Omega^{n+1},g)$ } be a compact, {  real-analytic ($C^{\omega}$) } Riemannian manifold with $C^{\omega}$ boundary  $\partial \Omega$ and ${\mathcal D}: C^{\infty}(\partial \Omega) \to C^{\infty}(\partial \Omega)$ be the associated DtN operator. Then for all $\delta>0$ there exist $0<\e_0=\e_0(\Omega^n,g,\delta)$  such that for $\phi_h\in C^\omega(\partial\Omega)$ with
$$(h\mc{D}-1)\phi_h=0,\quad \|\phi_h\|_{L^2(\partial\Omega)}=1,$$
 the harmonic extension $u_{h}(x)$  satisfies the exponential decay estimate
\begin{equation} \label{basic bound}
| \partial_{x}^{\alpha} u_{h}(x) | \leq C_{\alpha,\delta} h^{  -\frac{\red{n}}{2}  + \frac{1}{4} - |\alpha|} \, \exp \Big(- d(x) / h\Big),\quad d_{\partial\Omega}(x)<\e_0, 
\end{equation}
{ where $C_{\alpha,\delta}>0$ is a constant independent of $h.$}
In \eqref{basic bound},
$$ d(x) = d_{\partial \Omega}(x) +(C_{\Omega,g}-\delta)d_{\partial \Omega}^2(x),$$
where $d_{\partial \Omega}(x)$ is the Riemannian distance  to the boundary and 
\begin{equation}
\label{secondterm}
C_{\Omega,g}=-\frac{3}{2}+\frac{1}{2}\inf_{(x',\xi')\in S^*\partial\Omega}Q(x',\xi').
\end{equation}
Here $Q(x',\xi')$ is the symbol of the second fundamental form of the boundary $\partial \Omega.$
\end{theo}

\begin{rem}
The estimate \eqref{basic bound} is only valid in a small collar neighborhood of $\partial\Omega$ and indeed, since $C_{\Omega,g}$ may be negative, the rate function $d(x)$ may cease to give exponential decay outside a collar neighborhood of $\partial\Omega$. 
However, the maximum principle for the Laplace equation on $\Omega_\e=\{x\in \Omega\mid d_{\partial\Omega}(x)>\e\}$ together with Theorem \ref{mainthm} also implies that for any $\e>0$, there exists $C,c>0$ so that 
$$|u_h(x)|\leq Ce^{-c/h},\quad d_{\partial\Omega}(x)>\e.$$
Unfortunately, we lose control of the rate function outside a small collar neighborhood of the boundary.
\end{rem}

We will see by examining the case of $\Omega=B(0,R)\subset\re^2$ that the rate of decay in \eqref{basic bound} is optimal to first order at the boundary. We do not expect to be able to obtain the optimal second order estimate since we are forced to throw away some of the oscillations in $u$ when we apply the Cauchy-Schwarz inequality in (\ref{IDEA}), however, the behavior with respect to $Q$ is optimal as we will see from several examples.

We point out that the estimate in Theorem \ref{mainthm} holds without change if one takes the $\phi_h$ to be $L^2$-normalized Laplace eigenfunctions on the boundary. Indeed,  the bound in Theorem \ref{mainthm} can be adapted to the case of harmonic extensions of eigenfunctions of general elliptic, analytic, self-adjoint $h$-pseudodifferential operators on $M$. However, the bounds are somewhat cumbersome to state and we do not pursue this here.

It is worth noting that the proof of Theorem \ref{mainthm} is microlocal and thus the constant $C_{\Omega,g}$ can be made to depend on the nearest point in $\partial\Omega$. In particular, let $(x',x_{\red{n+1}})$ be Fermi normal coordinates in a collar neighborhood of $\partial\Omega$ so that $x_{\red{n+1}}=d_{\partial\Omega}(x)$. Then the estimate \eqref{basic bound} holds with $d(x)$ replaced by
$$\tilde{d}(x)=x_{\red{n+1}}+(a_{\Omega,g}(x')-\delta)x_{\red{n+1}}^2$$
where 
$$a_{\Omega,g}(x')=-\frac{3}{2}+\frac{1}{2}\inf_{p\in S^*_{x'}\partial\Omega}Q(p).$$

\subsection{Examples of Steklov Eigenfunctions -- sharpness of $d(x)$ to first order}$ $

We now examine a few examples to illustrate  the results of Theorem \ref{mainthm}.
\subsubsection{The Disk}
\label{s:e1}
Let $\Omega=B(0,R)\subset \re^2$. Then the Steklov eigenvalues are precisely $\sigma=0,\frac{1}{R},\frac{2}{R}\dots $ with corresponding Steklov eigenfunctions given by 
\begin{equation}
\label{e:e1}u^{\pm}_{\red{k}}=\frac{1}{\sqrt{2\pi R}R^n}r^{\red{k}}e^{\pm i{\red{k}}\theta},\quad \sigma =\frac{{\red{k}}}{R}.
\end{equation}
In particular, letting $h=\sigma^{-1}={\red{k}}^{-1}R$, 
$$u^{\pm}_{\red{k}}=\frac{1}{\sqrt{2\pi R}}e^{[R\log (1-(R-r)/R))]/h}e^{i\theta/h}.$$
Therefore, in this case 
$$d(x)=-R\log (1-d_{\partial\Omega}(x)/R)=\sum_{{\red{j}}=1}^\infty \frac{[d_{\partial\Omega}(x)]^{\red{j}}}{{\red{j}}R^{{\red{j}}-1}}.$$
This shows that to first order, the results of Theorem \ref{mainthm} are sharp. Moreover, notice that the second fundamental form of $\partial B(0,R)$ is given by $R^{-1}$ and thus, for the disk, the optimal quadratic term is
\begin{equation}
\label{e:trueQuadratic}\frac{1}{2}\inf_{p\in S^*\partial\Omega}Q(p)[d_{\partial\Omega}(x)]^2
\end{equation}
hence, modulo the $\frac{3}{2}$ in \eqref{secondterm}, the constant $C_{\Omega,g}$ sharp. In particular, as the curvature of the boundary increases, the decay into the interior becomes more rapid.

The case of spheres in higher dimensions is nearly identical if we replace $e^{\pm in\theta}$ by a spherical harmonic.

\subsubsection{Cylinders}
Let $(M,g)$ be a real analytic manifold of dimension $n$ without boundary and $\Omega=(-1,1)_t\times M_x$ with metric $dt^2+g(x).$ Then 
$$\Delta_\Omega=\partial_t^2+\Delta_M.$$
Let $\varphi_{\red{k}}$ be an orthonormal basis for $L^2(M)$ with 
$$(-\Delta_M-\lambda_{\red{k}}^2)\varphi_{\red{k}}=0.$$
Then the Steklov eigenfunctions are given by 
$$u_h(t,x)=\frac{\cosh(\lambda_{\red{k}}t)}{\cosh(\lambda_{\red{k}})}\varphi_{\red{k}}(x),\qquad v_h(x,t)=\frac{\sinh(\lambda_{\red{k}}t)}{\sinh(\lambda_{\red{k}})}\varphi_{\red{k}}(x)$$
with Steklov eigenvalues $\sigma_{\red{k}}=\lambda_{\red{k}}\tanh(\lambda_{\red{k}})$ and $\sigma'_{\red{k}}=\lambda_{\red{k}}\coth(\lambda_{\red{k}})$ respectively. Notice that for $\lambda_{\red{k}}\gg 1$, 
$$\cosh(x)=\frac{1}{2}e^{|x|}+O(e^{-|x|}),\qquad \sinh(x)=\frac{\sgn(x)}{2}e^{|x|}+O(e^{-|x|}).$$

In particular, near $|t|=1$,
$$|u_h(t,x)|=(e^{-\lambda_{\red{k}}(1-|t|)}+O(e^{-\lambda_{\red{k}}}))|\varphi_{\red{k}}(x)|.$$
Then, notice that 
$$\sigma_{\red{k}}=\lambda_{\red{k}}(1+O(e^{-\lambda_{\red{k}}})).$$
So, using H\"ormander's $L^\infty$ bounds (see e.g. \cite[Chapter 7]{Zw}) we have the estimate
$$|u_h(t,x)|=(e^{-\sigma_{\red{k}}(1-|t|)}+O(e^{-\sigma_{\red{k}}}))|\varphi_{\red{k}}(x)|\leq C\sigma_{\red{k}}^{\frac{n-{\red{1}}}{2}}e^{-\sigma_{\red{k}}(1-|t|)}.$$
Indeed, it is not hard to construct examples (e.g. where $M$ is the sphere) so that this estimate is sharp.
The analysis is similar for $v_h$.

In particular, the best possible decay rate is given by
$$d(x)=d_{\partial\Omega}(x)$$
and we again see that the first order term in Theorem \ref{mainthm} is sharp and that the quadratic term is given by \eqref{e:trueQuadratic} since in this case the second fundamental form is $0$. 

\subsubsection{The Annulus}
\label{s:e2}
Now, consider $B(0,1)\setminus B(0,r_0)\subset \re^2$. Then a simple computation shows that the Steklov eigenvalues are the roots, of 
$$p_{\red{k}}(\sigma)=\sigma^2-\sigma {\red{k}}\left(\frac{1+r_0}{r_0}\right)\left(\frac{1+r_0^{2{\red{k}}}}{1-r_0^{2{\red{k}}}}\right)+\frac{{\red{k}}^2}{r_0},\quad {\red{k}}=0,1,\dots$$
with corresponding eigenfunctions
\begin{equation}
\label{e:e2}
u^{\pm}_\sigma(r,\theta)=C_{{\red{k}},\sigma}e^{\pm in\theta}\left(r^{\red{k}}+\frac{{\red{k}}-\sigma}{{\red{k}}+\sigma}r^{-{\red{k}}}\right).
\end{equation}
It is easy to show that the roots of $p_n(\sigma)$ have 
$$\sigma_{{\red{k}},1}={\red{k}}+O({\red{k}}r_0^{2{\red{k}}}),\quad \sigma_{{\red{k}},2}=\frac{{\red{k}}}{r_0}+O({\red{k}}r_0^{2{\red{k}}}).$$
Then, 
\begin{gather*} 
u^{\pm}_{\sigma_{{\red{k}},1}}=\frac{1}{\sqrt{2\pi}}e^{\pm i{\red{k}}\theta}r^{\red{k}}+O(e^{-cv}),\quad u^{\pm}_{\sigma_{{\red{k}},2}}=\frac{1}{\sqrt{2\pi r_0}}r_0^{\red{k}}r^{-{\red{k}}}e^{\pm i{\red{k}}\theta}+O(e^{-c{\red{k}}}).
\end{gather*}
The case of $u_{\sigma_{{\red{k}},1}}$ is identical to that for the disk, so we focus on $u_{\sigma_{{\red{k}},2}}$. Let $h=\sigma_{{\red{k}},2}^{-1}=r_0{\red{k}}^{-1}+O(e^{-c{\red{k}}}).$ Then,
$$|u^{\pm}_{\sigma_{{\red{k}},2}}(r,\theta)|=\frac{1}{\sqrt{2\pi r_0}}e^{-r_0\log[1+ (r-r_0)/r_0]/h}(1+O(e^{-c/h})).$$
Therefore, in this case 
$$d(x)=-r_0\log (1+d_{\partial\Omega}(x)/r_0)=-\sum_{{\red{j}}=1}^\infty \frac{[-d_{\partial\Omega}(x)]^{\red{j}}}{{\red{j}}r_0^{{\red{j}}-1}}.$$
This shows again that to first order, the results of Theorem \ref{mainthm} are sharp. Moreover, notice that the second fundamental form of $\partial\Omega$ near $\partial B(0,r_0)$ is given by $-r_0^{-1}$ and thus, near this boundary component, the optimal quadratic term is again given by \eqref{e:trueQuadratic}.

\subsection{Microlocal Estimates}

It is clear from the approximate Poisson formula for $u_h(x)$ (see \eqref{poisson} below) that the exponential eigenfunction decay in the interior $\Omega$ is closely related to the precise rate of $h$-microlocal exponential decay of the boundary DtN eigenfunctions $\phi_h$ off the cosphere bundle $S^*\partial \Omega = \{(y,\eta) \in T^*\partial \Omega; |\eta|_g =1 \}.$ To derive the requisite bounds, we prove weighted exponential $h$-microlocal estimates for the associated wave packets $T(h) \phi_h$ where $T(h): C^{\infty}(\partial \Omega) \to C^{\infty}(T^*\partial \Omega)$ is a globally-defined FBI transform in the sense of \cite{Sj}.  Since these estimates seem of independent interest, we prove them for a rather general class of analytic $h$-pseudodifferential operators. An important consequence is the following exponential decay estimate for $T(h) \phi_h$ off the characteristic variety when $\phi_h$ solves $P(h)\phi_h=O(e^{-c/h}).$

We recall that an operator $P(h) \in Op_{h}(S^{0,k})$ with principal symbol $p(y,\eta)$ is said to have  simple characteristics provided $dp \neq 0$ on the set $\{p=0 \}.$ Moreover, $p(y,\eta)$ is classically elliptic if $|p(y,\eta)| \geq C' \langle \eta \rangle^{k}$ for $|\eta| \geq C$ with constants $C,C' >0.$ Let $S^{m,k}_{cla}$ denote the class of classical analytic symbols (see section \ref{s:FBIDescription}).

\begin{theo} \label{mainthm2} Let $(M^{\red{n}},g)$ be a compact, closed, real-analytic manifold and  $P(h) \in Op_h (S^{0,k}_{cla})$ be an analytic, $h$-pseudodifferential operator with real, classically elliptic principal symbol $p(x,\xi)$ having simple characterisitics.  Suppose that
$$P(h)\phi_h=O_{L^2}(e^{-c/h}) ,\quad \|\phi_h\|_{L^2}=1.$$

\noindent Let $T(h): C^{\infty}(M) \to C^{\infty}(T^*M)$ be a globally-defined FBI transform as in \eqref{FBIphase} associated with an $h$-ellptic symbol $a \in S^{3{\red{n}}/4,{\red{n}}/4}_{cla}$  and consider the weight function
$$\psi(x,\xi) =  \frac{ \delta  \cdot p^2(x,\xi)}{ \langle \xi \rangle^{2k}}.$$ \  Then, provided $\delta >0$ is a sufficiently small positive constant depending on $(M,g),$ it follows that   for $h \in (0,h_0(\delta)]$ with $h_0(\delta)>0$ sufficiently small,

\begin{equation}\label{e:global} \| e^{\psi/h} T(h) \phi_h \|_{L^2(T^*M)} = O(1).\end{equation}\
\end{theo}\

The main technical ingredient needed for the proof of Theorem \ref{mainthm2} is given in Proposition \ref{weightedl2};  it is essentially the manifold analogue of the microlocal exponential weighted estimates in $\R^n$ proved by Martinez \cite{Mar, Mar2}  and Nakamura \cite{Na} (see also \cite{T}). 
As a direct application of Theorem \ref{mainthm2}, we prove the  first-order exponential decay estimate in Theorem \ref{mainthm} \eqref{basic bound}.

\begin{rem}Although  weaker than  Theorem \ref{mainthm2} (since no rate of decay is specified), we point out that the following exponential decay estimate is an immediate consequence of  Theorem \ref{mainthm2}.

 \begin{cor} \label{thm2}For fixed $\ep_0 >0$ consider the cutoff $\chi_{\ep_0}(x,\xi) := \chi \Big( \frac{ p(x,\xi)}{\ep_0} \Big).$
Then, for any {$\phi_h$ satisfying the assumptions in Theorem \ref{mainthm2},} there exists constant $C(\ep_0)>0$ such that  for $h \leq h(\ep_0),$ 
$$ \|  (1-\chi_{\ep_0}) T(h) \phi_h \|_{L^2(\partial \Omega)} = O(e^{-C(\ep_0)/h}).$$
\end{cor}
\noindent We also note that Corollary \ref{thm2} also follows from a rather standard parametrix construction for analytic $h$-pseudodifferential operators (see subsection \ref{longrange}).  \end{rem}

To prove Theorem \ref{mainthm} \eqref{secondterm}, one must  estimate the constant $C_{\Omega,g}$ in the rate function $d(x) = d_{\partial \Omega}(x) - C_{\Omega,g} d^2_{\partial \Omega}(x).$ Unfortunately,   the weighted $L^2$ bound in Theorem \ref{mainthm2} for $T \phi_h$ does not quite suffice for this since one needs an additional geometric bound for the constant $\delta >0.$ To  achieve this, we must $h$-microlocally refine the bound in Theorem \ref{mainthm2}. 
Global refinement of the  decay estimate for  $T(h)\varphi_h$ by specifying sharp constant $\delta>0$ in the Gaussian seems a rather intractible problem. However, with Corollary \ref{thm2} in hand, we see that away from the characteristic variety, $T(h)\varphi_h$ is exponentially small. Consequently, instead of attempting to refine the global estimate in Theorem \ref{mainthm2}, we $h$-microlocalize to a small neighbourhood of the characteristic variety $p^{-1}(0).$  In order to exploit the real-analyticity of $(M,g)$ and $P(h)$, we choose the neighbourhood of $M$ in $T^*M$ to be the Grauert tube complexification $M_{\tau}^{\C} \supset M,$ which we assume throughout contains the characteristic variety, $p^{-1}(0).$  In addition, for the subsequent estimates, it will be important to choose a specific $h$-microlocal FBI transform $T_{hol}(h): C^{\infty}(M) \to C^{\infty}(M_{\tau}^{\C})$ that is compatible with the complex structure on $M_{\tau}^{\C}$. In view of \cite[Theorem 0.1]{GLS}, it is natural to choose $T_{hol}(h)$ to be the holomorphic continuation of the heat kernel on $(M,g)$ at time $t = \frac{h}{2}.$ 

\smallskip
\noindent{\bf{{{\red{Motivating Example of FBI transforms}}}}}\par\smallskip

\red{Recall that the standard FBI transform on $\R^n$ is given by 
\begin{equation}
\label{e:stdFBI}
T_{\R^n}u(\alpha_x,\alpha_\xi)=2^{-\frac{n}{2}}(\pi h)^{-\frac{3n}{4}}\int e^{-\frac{i}{h}\langle \alpha_x-y,\alpha_\xi\rangle-\frac{1}{2h}|\alpha_x-y|^2}u(y)dy.
\end{equation}
We will introduce two globally defined FBI transforms below, $T_{hol}$ and $T_{geo}$. In the case of $\R^n$, these two FBI transforms agree and are given by~\eqref{e:stdFBI}.
That is, on $\R^n$,
$$
T_{hol}(\alpha_x-i\alpha_\xi)=T_{geo}(\alpha_x,\alpha_\xi)=T_{\R^n}(\alpha_x,\alpha_\xi).
$$
}
Before formally stating our next result, we give some motivation. Consider the simple example of the circle $\R/ 2\pi \Z$ with a flat metric $g = dx^2.$ \red{That is, consider the boundary for the examples in Sections~\ref{s:e1} and~\ref{s:e2}.} The \red{functions} $\phi_h(x) = e^{ix/h}$ \red{appearing in~\eqref{e:e1} and~\eqref{e:e2}} satisfy $(hD_{x}-1)\phi_h=0.$ 
In this case, the complexification is 
$$\C / 2\pi \Z = \{ \alpha_x + i \alpha_\xi; \alpha_x + 2\pi \equiv \alpha_x \} \cong T^* (\R/2\pi \Z).$$ 
\red{We compute $T_{hol}\phi_h$ by extending $\phi_h$ smoothly to $\R$ as a solution of $(hD_{x}-1)\phi_h=0$ or using the definition of $T_{hol}$ in~\eqref{e:thol}. Then,}
$$T_{hol}\phi_h(\alpha_x-i\alpha_\xi)=h^{-1/4}e^{-\frac{1}{2h}\alpha_\xi^2-\frac{1}{2h}}e^{i(\alpha_x-i\alpha_\xi)/h}=h^{-1/4}e^{i\alpha_x/h}e^{-\frac{1}{2h}(\alpha_\xi-1)^2}.$$ 
In particular,  with $\psi_{hol}(\alpha) = \frac{1}{2} (\alpha_\xi - 1)^2,$
$$e^{\psi_{hol}/h}T_{hol}\phi_h=h^{-1/4}e^{i\alpha_x/h}.$$
Thus, for any $\gamma<1$,
\begin{equation} \label{torus1}
c\leq \|e^{\frac{\gamma\psi_{hol}}{h}}T_{hol}\phi_h\|_{L^2(T^*S^1)}\leq C. \end{equation}
We note that \eqref{torus1} is consistent with Theorem \ref{mainthm2}. 
However, it is useful to observe that with the precise weight function $\psi_{hol} = \frac{1}{2}(\alpha_{\xi}-1)^2$ (with $\gamma =1$),
\begin{equation} \label{torus2}
 \| e^{\frac{\psi_{hol}}{h}}T_{hol}\phi_h \|_{L^\infty(T^*S^1)}= h^{-1/4}. \end{equation}
Thus, with the optimal weight function $\psi_{hol}$, one has a polynomial gain of $h^{-1/4}$ in the weighted $L^2$ mass. The computations for eigenfunctions on higher-dimensional flat tori $\R^n/\Z^n$ are very similar.

To state the $h$-microlocal refinement of Theorem \ref{mainthm2}, we let $M$ be a compact, closed, real-analytic manifold of dimension $m$  and $\widetilde{M}$ denote a Grauert tube complex thickening of $M$ with $M$ a totally real submanifold. By Bruhat-Whitney, $\widetilde{M}$ can be identified with $M^{\C}_{\tau} := \{ (\alpha_x, \alpha_\xi) \in T^*M; \sqrt{2\rho}(\alpha_x,\alpha_\xi) \leq \tau \}$ where $\sqrt{2\rho} = |\alpha_{\xi}|_g$ is the exhaustion function using the complex geodesic exponential map $ \kappa : M_{\tau}^{\C} \rightarrow \tilde{M}$ with $\kappa(\alpha) = \exp_{\alpha_x}( i \alpha_{\xi}).$ 

\red{\begin{rem}
We use the notation $\alpha$ rather than $(x,\xi)$ because it is useful to think of $\alpha \in \widetilde{M}$ at some times and $\alpha\in T^*M$ at other times where we identify $\widetilde{M}$ as a subset of $T^*M$.
\end{rem}}

By possibly rescaling the semiclassical parameter $h$ we assume without loss of generality that the characteristic manifold
$$ p^{-1}(0) \subset M^{\C}_{\tau}.$$
Now, let $e^{h\Delta_g/2}$ have Schwartz kernel $E(x,y,h)$ and $E^{\mathbb{C}}(\alpha,y,h)$ denote the holomorphic continuation of $E(x,y,h)$ to $M^{\mathbb{C}}_\tau$ in the outgoing $x$-variables. It is proved in \cite{GLS} Theorem 0.1 (see also Section \ref{s:FBIDescription}) that the operator given by
\begin{equation} 
\label{e:thol}T_{hol}(h)u(\alpha):=h^{-{\red{n}}/4}e^{-\rho(\alpha)}\int_M E^{\mathbb{C}}(\alpha,y,h)u(y)dy
\end{equation}
is an $L^2$-normalized, $h$-microlocal FBI transform defined for $\alpha\in \red{\widetilde{M}}.$ 

\begin{theo}
\label{mainthm3}
Under the same assumptions as in Theorem \ref{mainthm2} and with $T=T_{hol}$ is as in \eqref{e:thol}, there exists $\e>0$ small enough  so that with $\gamma<1/2$,
\begin{equation}\label{holbound1}\|e^{\gamma a_0p^2/h}T_{hol}(h)\phi_h\|_{L^2(\{|p|\leq \e\})}=O(1), \quad  a_0=\frac{1}{2|dp|_{g_s}^2} \end{equation}\
where $g_s$ is the Sasaki metric on $TM$ (see for example \cite[Chapter 9]{BlairSasaki}).

Moreover, there exists $\psi_{hol} \in C^{\infty}(M_{\tau}^{\C})$ with
$$\psi_{hol}(\alpha) = p^2(\alpha)\big( \, a_0(\alpha)+O(p(\alpha)) \,\big),$$
where $a_0$ is given in \eqref{holbound1} such that  for $\e>0$ sufficiently small, 

\begin{equation}\label{holbound2}\|e^{\psi_{hol}/h}T_{hol}(h)\phi_h\|_{L^2(\{|p|\leq \e\})}=O(h^{-1/4}).\end{equation}\

\end{theo}

We note that the example of the Laplace eigenfunctions on the circle (see \eqref{torus1} and \eqref{torus2} above), shows that the upper bound in Theorem \ref{mainthm3} \eqref{holbound2} is sharp. 

%\begin{rem} Using a slightly more complicated construction to find the weight function, we can improve $\psi_{hol}$ so that it is sharp to any order $p^N$ for Laplace eigenfunctions on the circle as well as for the function $e^{ix^3/3h}$ with $P(h)=hD_x+x^2$. See Remark \ref{r1}.
%\end{rem}

\begin{rem}We use \eqref{holbound1} instead of \eqref{holbound2} in the proof of Theorem \ref{mainthm}. Using \eqref{holbound2} results in better estimates as soon as $d(x)\gg ch\log h^{-1}$, but for simplicity and since the function $d(x)$ that we obtain is still not sharp, we do not state these estimates here.
\end{rem}

\subsection{Sketch of the proof of Theorem \ref{mainthm}}

Let ${\mathcal P}: C^{\infty}(\partial \Omega) \to C^{\infty}(\Omega)$ be the Poisson operator for the boundary value problem in \eqref{dirichlet}, so that  $ u_h(x) = {\mathcal P} \phi_h(x).$ The first step in the proof of Theorem \ref{mainthm} amounts to understanding the microlocal structure of the Poisson operator, ${\mathcal P}$ following the analysis in \cite{SU}.

\subsubsection{Microlocal analysis of the Poisson operator}
In view of \cite[Section 3]{SU} and the fact that $\phi_h$ is microlocally supported away from the zero section one can write
\begin{align} \label{idea1}
u_h(x) &= {\mathcal P} \phi_h (x) =  U(h) \phi_h(x) + O(e^{-C/h}) \end{align}
where $U(h): C^{\infty}(\partial \Omega) \to C^{\infty}(\Omega)$ is a semiclassical, complex-phase $h$-Fourier integral operator supported near diagonal. In terms of Fermi coordinates $(x_{\red{n+1}},x')$ in a collar neighbourhood  $U= \{ (x',x_{\red{n+1}}); x_{\red{n+1}} \geq 0 \}$ of $\partial \Omega = \{ x_{\red{n+1}} =0 \}$, $U(h)$ has Schwartz kernel 

\begin{equation} \label{poissonkernel}
K(x,y',h) = (2\pi h)^{-n} \int_{\R^{n}} e^{ i \langle x' -y', \red{\xi'} \rangle /h} \, e^{- {\red{\Psi}}(x_{\red{n+1}},x',\red{\xi'}) /h} \, a(x,y',\red{\xi'},h) \, \chi(x'-y') \, d\red{\xi'}. \end{equation}

Here, 
$${\red{\Psi}}(x_{\red{n+1}},x',\xi')=x_{\red{n+1}}|\xi'|_{x'}+\frac{x_{\red{n+1}}^2(Q(x',\xi')+i\lan \partial_{x'}|\xi'|_{x'},\xi'\ran_{x'})}{2|\xi'|_{x'}}+O(x_{\red{n+1}}^3|\xi'|_{x'})$$
 satisfies a complex eikonal equation \eqref{eikonal} and $a(x,y',\red{\xi'},h)$ is a semiclassical analytic symbol as in \eqref{scsymbol}. \red{Note that the error is exponentially decreasing since we work in the analytic setting.}

In the Euclidean case, we note that one can derive the semiclassical Poisson formula \eqref{idea1} in an elementary fashion directly from the potential layer formulas using residue computations. For the benefit of the reader, we outline the argument here.
Let $\Omega \subset \R^{n}$ be a bounded Euclidean domain with real-analytic boundary $\partial \Omega.$
Let $G(z,z') \in {\mathcal D}'(\R^{n} \times \R^{n}) $ be the free Green's functions with $ \Delta_{z'} G(z,z') = \delta(z-z').$  From Green's formula and the DtN eigenfunction condition, one gets
\begin{align} \label{layer}
u_{h}(z) & = h^{-1} \int_{\partial \Omega} G(z,z') \phi_{h}(z') d\sigma(z')   - \int_{\partial \Omega}  N(z,z') \phi_{h}(z') d\sigma(z'). \end{align}
with $N(z,z') = \partial_{\nu(z')} G(z,z'),\,\,\, (z,z') \in \R^{{\red{n+1}}} \times \partial \Omega.$ Writing $G(z,z')$ and $N(z,z')$ as Fourier integrals and rescaling the frequency variables $\xi \to h^{-1} \xi$ one rewrites \eqref{layer} in the form
\begin{align} \label{layer}
u_h(z) & = (2\pi h)^{-{(\red{n+1})}}h \int_{\R^{{\red{n+1}}}} \int_{\partial \Omega}  e^{i \langle z -z', \xi \rangle /h } ( |\xi|^2 + i 0)^{-1} \, \phi_h(z') d\sigma(z') d\xi \nonumber \\
&   + i (2\pi h)^{-{(\red{n+1})}}h \int_{\R^{{\red{n+1}}}} \int_{\partial \Omega} e^{i \langle z-z',\xi \rangle /h} \langle \nu(z'), \xi \rangle (|\xi|^2 + i0)^{-1} \phi_h(z') d\sigma(z') d\xi. \end{align}

Let $\chi \in C^{\infty}_0(\R)$ with $\chi =1$ near the origin and supp $\chi \subset [-\epsilon_0,\epsilon_0].$
We note that by making a change a change of contour $\xi \mapsto \xi + i \delta \xi$ in \eqref{layer} with $0 < \delta <1$ one can insert a spatial cutoff $\chi(|z-z'|)$ in both integrals modulo an $O(e^{-C/h})$ error. 
Next, we introduce convenient coordinates in a tubular neighbourhood $U_{\partial \Omega}$ of the  boundary. Given a local $C^{\omega}$ parametrization of the boundary $q: U \to \partial \Omega$ with $U \subset \R^n$ open, we write locally
$$ z = q(x') + x_{{\red{n+1}}}  \nu(x')\,\,\, \text{and} \,\, z' = q(y')$$

By choosing $\epsilon_0>0$ sufficiently small, we can assume that $z$ and $z'$ lie in the same local coordinate chart.
In terms of these new coordinates, one can rewrite the phase function
$$ \langle z - z',\xi \rangle = \langle q(x') - q(y'), \xi \rangle  + x_{{\red{n+1}}} \langle \nu(x'), \xi \rangle= \langle x' - y', dq^t(x',y') \xi \rangle + x_{{\red{n+1}}} \langle \nu(x'), \xi \rangle.$$
We make the affine change of variables in \eqref{layer} given by $\xi \mapsto (\eta',\eta_{{\red{n+1}}})$ where
$$\eta_{{\red{n+1}}} = \langle \nu(x'), \xi \rangle, \,\,\, \eta' = dq^t(x',y') \xi$$

Then, for $x= (x',x_{\red{n+1}}) \in U_{\partial \Omega},$ using the fact that the DtN eigenfunctions are $h$-microlocally $O(e^{-C/h})$ near the zero section $\red{\eta}' = 0$ (see Proposition \ref{zsection}), one can write
\begin{equation} \label{layer 2}
\begin{aligned}
&u_h(x)  + O(e^{-C(\ep)/h}) = \\
&\quad(2\pi h)^{-{(\red{n+1})}}h \int_{|\eta'| > \epsilon} \int_{\partial \Omega}  e^{i  \langle x'-y', \eta' \rangle /h }  e^{i  x_{\red{n+1}} \eta_{\red{n+1}} /h} \,b(x,y',\eta) \chi(x'-y') \phi_h(y') dy' d\eta_{\red{n+1}} d\eta' 
\end{aligned}
\end{equation}
where, 
$ b(x,y',\eta) =  (1 + i \eta_{\red{n+1}}) ( \eta_{\red{n+1}}^2 + |\eta'|_x^2 + i 0)^{-1}  a(x',y').\ $ \\

For $\eta' \neq 0,$ a residue computation gives  $ \int_{\R} e^{i x_{\red{n+1}} \eta_{\red{n+1}}/h} ( \eta_{\red{n+1}}^2 + |\eta'|_{x} + i0)^{-1} d\eta_{\red{n+1}} =  \frac{  \pi e^{- x_{\red{n+1}} |\eta'|_x/h}}{ |\eta'|_x}$
and similarily,
$  \int_{\R} e^{i x_{\red{n+1}} \eta_{\red{n+1}}/h} \eta_{\red{n+1}} ( \eta_{\red{n+1}}^2 + |\eta'|_{x} + i0)^{-1} d\eta_{\red{n+1}} = \pi  e^{- x_{\red{n+1}} |\eta'|_x/h}.$ Substitution of these integral formulas in \eqref{layer 2} gives modulo $O(e^{-C(\ep)/h})$ error,
\begin{align*} 
u_{h}(x) & = (2\pi h)^{-{\red{n}}} \int_{|\eta'| >\ep} \int_{\partial \Omega}  e^{i  \langle x'-y',\eta' \rangle /h}  e^{-  x_{\red{n+1}} |\eta'|_x /h} |\eta'|_x^{-1} a(x',y')  ( 1 + i |\eta'|_{x}^{-1}) \phi_{\lambda}(y')  dy' d\eta'.
\end{align*}
This is  consistent with the general formula in \eqref{poissonkernel}.

 \subsubsection{Microlocal lift of the Poisson representation \eqref{idea1}}Given the representation of $u_h$ in \eqref{idea1} in terms of semiclassical, complex phase $h$-Fourier integral operator  $U(h): C^{\infty}(\partial \Omega) \to C^{\infty}(\Omega),$  the key  idea in the proof of Theorem \ref{mainthm} is to  lift \eqref{idea1} to the cotangent bundle of the boundary $T^*\partial \Omega$ and then apply the weighted estimate in Theorem \ref{mainthm2} to give the first order approximation for the Steklov decay rate function $d(x)$ in Theorem \ref{mainthm} \eqref{basic bound}. The quadratic term in $d(x)$ is then bounded from above to prove Theorem \ref{mainthm} \eqref{secondterm} using the refined $h$-microlocal weighted estimates in Theorem \ref{mainthm3}.

Roughly speaking, we do this as follows: Viewing $x \in \Omega$ as parameters, we consider the family of functions $K_{x,h} \in C^{\infty}(\partial \Omega)$ with 
$$K_{x,h}(y') := K(x,y',h).$$
Then, \eqref{idea1} can be written in the form
\begin{align} \label{idea2}
u_{h}(x) = \langle K_{x,h},  \overline{\phi_h} \rangle_{L^2(\partial \Omega)} + O(e^{-C/h}). \end{align}

To lift \eqref{idea2} we let $T(h): C^{\infty}(\partial \Omega) \to C^{\infty}(T^*\partial \Omega)$ be an FBI transform in the sense of Sj\"{o}strand \cite{Sj} and $S(h): C^{\infty}(T^*\partial \Omega) \to C^{\infty}(\partial \Omega)$ be a left-parametrix with
$$ S(h) T(h) = I + R(h),$$
and $R(h)$ exponentially small in the sense that
$$ | \partial_{x}^{\alpha} \partial_{y}^{\beta} R(x,y)|  = O_{\alpha,\beta} (e^{-C/h}).$$
Given the weight function $\psi \in C^{\infty}(T^*\partial \Omega)$ in Theorem \ref{mainthm2}, one  can write
\begin{align} \label{idea3}
u_{h}(x) = \langle e^{-\psi/h} S(h)^t K_{x,h}, \overline{e^{\psi/h} T(h) \phi_h} \rangle_{L^2(T^*\partial \Omega)} + O(e^{-C/h}). \end{align}\
Using the bound 

$$\| e^{\psi/h} T(h) \phi_h \|_{L^2(T^*M)} = O(1) $$ \

in Theorem \ref{mainthm2} and applying Cauchy-Schwarz in \eqref{idea3} one gets

\begin{equation} \label{IDEA}
|u_{h}(x) |  = O(1) \, \| e^{-\psi/h} S(h)^t K_{x,h} \|_{L^2(T^*\partial \Omega)} + O(e^{-C/h}). \end{equation}\

Finally, a direct analysis of the first term on the RHS of \eqref{IDEA} using the method of analytic stationary phase yields  the bounds \eqref{basic bound} and \eqref{secondterm}  in Theorem \ref{mainthm}. We refer to section \ref{steklovbound} for a detailed proof of Theorem \ref{mainthm} using the weighted bounds in Theorems \ref{mainthm2} and \ref{mainthm3}.

\begin{rem} Both Theorems \ref{mainthm2} and \ref{mainthm3} have other applications to eigenfunction bounds, including the problem of obtaining geometric rates of decay for eigenfunctions in subdomains of configuration space $M$ that correspond to classically forbidden regions that are geometrically more refined than the classical Agmon-Lithner estimates. Specific examples include (but are not limited to) joint eigenfunctions for quantum completely integrable (QCI) eigenfunctions. We hope to return to this elsewhere.
\end{rem}

\subsection{Outline of the paper}
In section \ref{s:FBIDescription} we discuss the eigenfunction mass microlocalization results for the eigenfunctions $\phi_h$. The long range exponential decay estimates are proved in Proposition \ref{p:longRange} and the short range exponential weighted estimates near the characteristic variety (and inside the Grauert tube $M_{\tau}^{\C}$) are proved in Proposition  \ref{weightedl2}. These estimates are combined to prove Theorem \ref{mainthm2} in section \ref{mainthm2proof}. The $h$-microlocally refined weighted estimates for $T_{hol}(h) \phi_h$  along with the proof of Theorem \ref{mainthm3} are taken up in section \ref{refine}. In section \ref{steklovbound}, the exponential weighted estimates in sections \ref{mainthm2proof} and \ref{refine} are used to prove the decay estimates in Theorem \ref{mainthm} for the Steklov eigenfunctions. In section \ref{zerosection}, we prove the necessary $h$-microlocal exponential decay estimates for the Steklov eigenfunctions  near the zero section of $T^*\partial \Omega.$ This is necessary since the semiclassical DtN operator $ h {\mathcal D}: C^{\infty}(\partial \Omega) \to C^{\infty}(\partial \Omega)$ fails to be an $h$-analytic pseudodifferential operator microlocally near the zero section.\\

\noindent {\sc Acknowledgemnts.} The authors would like to thank Iosif Polterovich and Steve Zelditch for their comments on an earlier version of this paper. Thanks also to Andras Vasy and Maciej Zworski for valuable suggestions. Finally, thanks to the anonymous referee for many helpful suggestions. J.G. is grateful to the National Science Foundation for support under the Mathematical Sciences Postdoctoral Research Fellowship  DMS-1502661.  The research of J.T. was partially supported by NSERC Discovery Grant \# OGP0170280 and an FRQNT Team Grant. J.T. was also supported by the French National Research Agency project Gerasic-ANR-
13-BS01-0007-0.

\section{eigenfunction mass microlocalization}
\label{s:FBIDescription}

Let $M$ be a compact, closed, real-analytic manifold of dimension $m$  and $\widetilde{M}$ denote a Grauert tube complex thickening of $M$ with $M$ a totally real submanifold. By Bruhat-Whitney, $\widetilde{M}$ can be identified with $M^{\C}_{\tau} := \{ (\alpha_x, \alpha_\xi) \in T^*M; \sqrt{\rho}(\alpha_x,\alpha_\xi) \leq \tau \}$ where $\sqrt{2\rho} = |\alpha_{\xi}|_g$ is the exhaustion function \red{$M^{\C}_{\tau}$, where we identify $\widetilde M$ with $M_{\tau}^{\C}$} using the complexified geodesic exponential map $ \kappa : M_{\tau}^{\C} \rightarrow \tilde{M}$ with $\kappa(\alpha) = \exp_{\alpha_x,\C}( i \alpha_{\xi})$
\red{ Viewed on $\widetilde{M}$, the function $\sqrt{\rho}(\alpha) = \frac{-i}{2\sqrt{2}} r_{\C}(\alpha,\bar{\alpha}),$ which satisfies homogeneous Monge-Ampere and its level sets exhaust the complex thickening $\widetilde{M}$ (see Remark \ref{kahler} and \cite{GS1} for further details). }

\smallskip

\red{ 
{\bf The example of the round sphere} To illustrate these basic complex analytic entities, we consider the case of   the $n$-dimensional round sphere 
$$M = \{(x_1,...,x_{n+1}) \in \R^{n+1};  x_1^2 + \cdots + x_{n+1}^2 = 1 \} \subset \R^{n+1}.$$
 The complexification of $M$ is the quadric
\begin{eqnarray}
\widetilde{M} &= \{(z_1,...,z_{n+1}) \in \C^{n+1};  z_1^2 + \cdots z_{n+1}^2 = 1 \} \nonumber \\
&= \{ (x,\xi) \in \R^{2(n+1)} ; |x|^2 - |\xi|^2 = 1, \,\, \langle x, \xi \rangle = 0 \}. \end{eqnarray}

The Riemannian exponential map written in terms of affine ambient coordinates  on $\R^{n+1}$ is 
$$\exp_{x}(\xi) = \cos(|\xi|) x  + (\sin |\xi|) \frac{\xi}{|\xi|},$$
where $\xi\in T_xM$, so $\xi\in \R^{n+1}$ is orthogonal to $x$. The complexification of $\exp_x(\xi)$ is then given by
$$\exp_{x,\C}( i \xi) =  ( \cosh |\xi| ) x + i (\sinh |\xi|) \frac{\xi}{|\xi|}.$$ 
The distance function $r(x,y) = 2 \sin^{-1} \Big( \frac{|x-y|}{2} \Big)$ complexifies to 
$$ r_{\C}(z,w) = 2 \sin^{-1} \Big( \frac{ \sqrt{ (z-w)^2} }{2} \Big)$$
and the associated exhaustion function on the complexification is
$$ \frac{-i}{2\sqrt{2}} r_{\C}(z,\bar{z})  =  \frac{1}{\sqrt{2}} \, \sinh^{-1}( |\Im z|), \quad z \in \tilde{M}. $$
%The associated intrinsic Kahler form on $\tilde{M}$ (see also Remark \ref{kahler}) is $\Omega_g =  \partial \bar{\partial} \sinh^{-1} (|\Im z|).$
Pulling $r_{\C}(z,\bar{z})$ back to $T^*{\mathbb S}^2$ via the complexified exponential map gives
$$ \sqrt{\rho}(z) =  \exp_{x,\C}^{-1} \Big( - \frac{i}{2\sqrt{2}} r_{\C}(z,\bar{z}) \Big) = \frac{ |\xi|}{\sqrt{2}}.$$
In terms of local coordinates $\alpha =(\alpha_x,\alpha_{\xi}) \in T^*{\mathbb S}^2,$ this just gives
$\sqrt{\rho}(\alpha) = \frac{1}{\sqrt{2}} |\alpha_{\xi}|_g.$
Of course, the multiplcative factor of $\frac{-i}{2\sqrt{2}}$ above is just a computationally convenient normalization that we choose to adopt here.}

\smallskip

By possibly rescaling the semiclassical parameter $h$ we assume without loss of generality that the characteristic manifold
$$ p^{-1}(0) \subset M^{\C}_{\tau}.$$
 We will also have to  consider a complexification of $T^*M$ of the form
 \begin{equation}
 \label{e:snake}
 \widetilde{T^*M}:= \{ \alpha; |\Im \alpha_x | < \tau, \,\, |\Im \alpha_{\xi}| \leq \frac{1}{C} \langle \alpha_\xi \rangle \}
 \end{equation}
 where $C \gg 1$ is a sufficiently large constant and  $T^*M \subset \widetilde{T^*M}$ is then a totally-real submanifold.

  We recall that a complex $m$-dimensional submanifold, $\Lambda,$ of $\widetilde{T^*M}$ is said to be {\em I-Lagrangian} if it is Lagrangian with respect to
$$ \Im \omega = \Im ( d\alpha_x \wedge d \alpha_{\xi})$$
where $\omega = d\alpha_{x} \wedge d \alpha_{\xi}$ is the complex symplectic form on $\widetilde{T^*M}$. 

Let $U\subset T^*M$ be open. Following \cite{Sj}, we say that $a \in S^{m,k}_{cla}(U)$  provided $a \sim h^{-m} (a_0 + h a_1 + \dots)$ in the sense that
\begin{equation}
\label{scsymbol} 
\begin{gathered} 
\partial_{x}^{\red{l_1}} \partial_{\xi}^{\red{l_2}} \overline{\partial}_{(x,\xi)} a = O_{\red{l_1},\red{l_2}}(1) e^{- \langle \xi \rangle/Ch}, \quad (x,\xi)\in U, \\ 
  \Big| a - h^{-m} \sum_{0 \leq j \leq \langle \xi \rangle/C_0 h} h^{j} a_j \Big| = O(1) e^{- \langle \xi \rangle/C_1 h},\quad
 |a_j| \leq C_0 C^{j} \, j ! \, \langle \xi \rangle^{k-j},\qquad (x,\xi)\in U.
 \end{gathered}
 \end{equation}
We sometimes write $S^{m,k}_{cla}=S^{m,k}_{cla}(T^*M)$. 

Following \cite{SU}, we also define the notion of a \emph{homogeneous analytic symbol of order $k$} and write $a\in S^k_{ha}$ provided that there exist holomorphic functions $a_k$ on a fixed complex conic neighborhood of $T^*M\setminus\{0\}$ homogeneous of degree $k$ in $\xi$ so that \red{there exists $C_0>0$ so that} 
\begin{equation}
\label{homsymbol1}
\left|a_{k-j}\left(x,\frac{\xi}{|\xi|}\right)\right|\leq \red{C_0}^{j+1}j^j,\qquad j\geq 0  
\end{equation}
\red{and for every $C_1>0$ large enough, there exists $C_2>0$ so that}
\begin{equation}
\label{homsymbol2}
\left|a(x,\xi)-\sum_{0\leq j\leq |\xi|/C_1}a_{k-j}(x,\xi)\right|\leq C_2e^{-|\xi|/C_2},\quad |\xi|\geq 1.
\end{equation}

We say that an operator $\red{A}(h)$ is a \emph{semiclassical analytic pseudodifferential operator of order $m,k$} if its kernel can be written as $\red{A}(x,y;h)=K_{\red{1}}(x,y;h)+R_{\red{1}}(x,y;h)$ where for all $\alpha,\beta$,
$$|\partial_x^\alpha \partial_y^\beta R_{\red{1}}(x,y\red{;h})|\leq C_{\alpha\beta} e^{-c_{\alpha\beta}/h}, \,\,\, c_{\alpha \beta} >0,$$
and
$$K_{\red{1}}(x,y;h)=\frac{1}{(2\pi h)^{\red{n}}}\int e^{\frac{i}{h}\lan x-y,\xi\ran}a(x,\xi,h)\chi(|x-y|)d\xi$$
where $\chi\in \Cc(\re)$ is 1 near 0 and $a\in S^{m,k}_{cla}$. \red{We say $A$ is $h$-elliptic if $|a_0(x,\xi)|>ch^{-m}\langle \xi\rangle^k$ where $a_0$ is from \eqref{scsymbol}. Recall also that $A$ is classically elliptic if there is $C>0$ so that if $|\xi|>C$, $|a_0(x,\xi)|>C^{-1}h^{-m}|\xi|^k$.}

We say that an operator, $\red{B}$ is a \emph{homogeneous analytic pseudodifferential operatior of order $k$} if its kernel can be written as $\red{B}(x,y)=K_{\red{2}}(x,y)+R_{\red{2}}(x,y)$ where $R_{\red{2}}(x,y)$ is real analytic and 
$$K_{\red{2}}(x,y)=\frac{1}{(2\pi)^{\red{n}}}\int e^{i\lan x-y,\xi\ran}\red{b}(x,\xi)\chi(|x-y|)d\xi$$
for some $\red{b}\in S^k_{ha}$.and $\chi\in \Cc(\re)$ is 1 near 0. \red{We say $B$ is elliptic if there exists $c>0$ so that $b_k>c|\xi|^k$ on $|\xi|\geq 1$ where $b_k$ is from~\eqref{homsymbol1} and~\eqref{homsymbol2}.} For more details on the calculus of analytic pseudodifferential operators, we refer the reader to \cite{SjA}.

 As in \cite{Sj},  given an $h$-elliptic, semiclassical analytic symbol $a \in S^{3{\red{n}}/4,{\red{n}}/4}_{cla}(M \times (0,h_0]),$  we consider an intrinsic FBI transform $T(h):C^{\infty}(M) \to C^{\infty}(T^*M)$ of the form
\begin{equation} \label{FBI}
T u(\alpha;h) = \int_{M} e^{i\phi(\alpha,y)/h}  a(\alpha,y,h)\chi( \alpha_x, y) u(y) \, dy \end{equation}
with $\alpha = (\alpha_x,\alpha_{\xi}) \in T^*M$  in the notation of \cite{Sj}. 

\red{
\begin{rem}
The normalization $a\in S^{3n/4,n/4}_{cla}$ appears so that $T$ is $L^2$ bounded with uniform bounds as $h\to 0$~\cite{Sj}. 
\end{rem}
}

 The phase function is required to satisfy $\phi(\alpha,\alpha_x) = 0, \, \partial_y \phi(\alpha,\alpha_x) = - \alpha_{\xi}$ and
$$ \Im (\partial_y^2 \phi)(\alpha,\alpha_x) \sim |\langle \alpha_{\xi} \rangle| \, \Id.$$

Given $T(h) :C^{\infty}(M) \to C^{\infty}(T^*M)$ it follows by an analytic stationary phase argument \cite{Sj} that one can construct an operator $S(h): C^{\infty}(T^*M) \to C^{\infty}(M)$ of the form
\begin{equation} \label{left}
 S v(x;h) = \int_{T^*M} e^{-i  \, \overline{\phi(x,\alpha)}  /h} b(x,\alpha,h) v(\alpha) \, d\alpha \end{equation}
with $b \in S^{3{\red{n}}/4,{\red{n}}/4}_{cla}$ such $S(h)$  is a left-parametrix for $T(h)$ in the sense that
$$S(h) T(h) = \Id + R(h),\qquad\partial_{x}^{\alpha} \partial_{y}^{\beta} R(x,y,h) = O_{\alpha, \beta}(e^{-C/h}).$$

We use  two invariantly-defined FBI transforms. The first transform 
$T_{geo}(h): C^{\infty}(M) \to C^{\infty}(T^*M)$ is defined using only the Riemannian structure of $(M,g)$ and has phase function
\begin{equation} \label{FBIphase}
\phi(\alpha,y) = i \exp_{y}^{-1}(\alpha_x) \cdot \alpha_{\xi} - \frac{1}{2} r^{2}(\alpha_{x},y) \langle \alpha_{\xi} \rangle. \end{equation}
Here, $r(\cdot,\cdot)$ is geodesic distance and 
$\chi(\alpha_x,y) = \chi_0(r(\alpha_x,y))$ where $\chi_0: \R \to [0,1]$ is an even cutoff with supp $\chi_0 \subset [-inj(M,g), inj(M,g)]$ and $\chi_0(r) =1$ when $|r| < \frac{1}{2} inj(M,g).$

The transform $T_{geo}(h)$ will be used to derive $h$-microlocal exponential decay outside the Grauert tube $M_{\tau}^{\C}$  and far from the characteristic variety $p^{-1}(0).$  We will refer to the corresponding estimates as {\em long-range}.

To estimate $h$-microlocal eigenfunction mass inside the Grauert tube $M_{\tau}^{\C}$ containing $p^{-1}(0)$ we use instead another $FBI$-transform $T_{hol}(h)$ which is defined in terms of the holomorphic continuation of the heat operator $e^{t \Delta_g}$ at time $t = h/2.$ We refer to the corresponding estimates as  {\em short-range}.

Before continuing,  we briefly recall here some background on the operator $T_{hol}(h): C^{\infty}(M) \to C^{\infty}(M_{\tau}^{\C})$ and refer the reader to \cite{GLS} for further details.

\subsubsection{Complexified heat operator on closed, compact manifolds} \label{heat}
Consider the heat operator of $(M,g)$ defined at time $h/2$ by $$E_{h}=e^{\frac{h}{2}\Delta_g}:C^{\infty}(M) \to C^{\infty}(M).$$
%We remark that although the convention is to write $t$  instead of $\h$ to refer to the time,
% we make this change of notation for we are only interested in small time asymptotics $\h\to 0^+$,
% and all our results come from semiclassical analysis arguments.

 By a result of Zelditch \cite[Section 11.1]{Z}, the maximal geometric tube radius $\tau_{\max}$ agrees with the maximal analytic tube radius in the sense that for all $ 0<\tau < \tau_{\max}$, all the eigenfunctions $\varphi_j$  extend holomorphically to $M_\tau^\C$ (see also \cite[Prop. 2.1]{GLS}). In particular, the kernel $E(\cdot,\cdot;h)$ admits a holomorphic extension to $M_\tau^\C \times M_\tau^\C$
for all $0<\tau < \tau_{\max}$  and $h \in (0,1)$, \cite[Prop. 2.4]{GLS}. We denote the complexification by $E_h^\C( \cdot, \cdot)$.
To recall asymptotics for $E^{\C}_h$ we note that
  the squared geodesic distance on $M$
$$r^2(\cdot, \cdot): M \times M \to \R$$
holomorphically continues in both variables to $M_{\tau} \times M_{\tau}$ in a straightforward fashion. 
More precisely, $0<\tau<\tau_{\max}$, there exists a connected open neighbourhood $\tilde \Delta \subset M_\tau^\C \times M_\tau^\C$ of the diagonal $\Delta \subset M \times M$ to which $r^2(\cdot, \cdot)$ can be holomorphically extended \cite[Corollary 1.24]{GLS}. We denote the \red{holomorphic} extension
by $r_\C^2(\cdot,\cdot) .$ Moreover, one can easily recover the exhaustion function $\sqrt{\rho_g}(\alpha_z)$ from $r_{\C}$; indeed, 
$\rho_g(\alpha_z)=-r^2_\C(\alpha_z, \bar{\alpha_z})$
for all $\alpha_z \in M_\tau^\C$.

To analyze the asymptotic behaviour of $E_h^{\C}(\alpha_z,y)$ with $(\alpha_z,y) \in M^\C_\tau \times M$, we split the kernel into two pieces where\\
 {\indent (i) the point $(\pi_{_M} \alpha_z,y) \in M \times M$ is close to the diagonal in terms of $inj$ and the Grauert tube radius $\tau$,} \\
 \indent (ii) the point $(\pi_{_M}\alpha_z ,y) \in M \times M$ is relatively far from the diagonal in terms of $inj$ and $\tau.$\\
  
 To control the behaviour of the complexified heat kernel for a pair of points $(\pi_{_M}  \alpha_z,y) \in M \times M$ that are relatively close or far from the diagonal, we  need  the  following result \cite{GLS}

\begin{prop}\label{kernel} There exist $0<\tau_0 \leq \tau_{\max}$  and positive constants $\beta$, $\delta_0$, $h_0$ and $C$,
 depending only on $\tau_0>0$,  such that for $0<\tau \leq \tau_0$, $0<\delta \leq \delta_0$ and
 $(\alpha_z,y)\in M_{\tau}^{\C}\times M$, the following is true: \\

(i) \, When $r(\pi_{M}  \alpha_z, y)<\delta$ and  $h \in (0,h_0],$
\begin{equation}\label{close to diag}
E_h^\C(\alpha_z,y)=e^{-\frac{r^2_\C(\alpha_z,y)}{2h}} a^\C(\alpha_z,y; h) + O (e^{-\beta/h}).
\end{equation}
Here, $a^\C(\alpha_z,y;h)$ is the polyhomogeneous sum 
\begin{equation} \label{phg}
a^{\C}(\alpha_z,y; h):=(2\pi h)^{-{\red{n}}/2} \sum_{0\leq k\leq D/h} a^\C_k(\alpha_z,y)h^k,
\end{equation}
where the $a_k^\C$'s denote the analytic continuation of the coefficients appearing in the formal solution
of the heat equation on $(M,g)$
 \medskip

(ii) \, \red{There exists $C>0$ so that} when $r(\pi_{M}  \alpha_z, y)>\frac{\delta}{2}$ and $h \in (0,1),$
\begin{equation}\label{away from diag}
\left| E_h^\C(\alpha_z,y) \right| \leq C \;e^{- \frac{\delta^2}{\red{C} h}},
\end{equation}
where $C$ is a positive constant depending only on $(M,g)$.
\end{prop}

From now on, we always carry out our analysis in the complex Grauert tubes $M_\tau^\C$ with $0<\tau \leq \tau_{\red{0}},$ where in view of Proposition \ref{kernel}, we have good control of the complexified heat kernel, $E^\C_h(\cdot,y)$ for $y \in M$.

For $(\alpha_z,y) \in M_{\tau}^{\C} \times M$ with $r(\Re z,y) <\epsilon$ with $\epsilon>0$ small, one can show that the function $y \mapsto -  \Re r^2_{\C}(\alpha_z,y)$ attains a non-degenerate maximum at $y = \Re z$. The corresponding strictly plurisubharmonic weight  is the square of the exhaustion function given by
$$ 2\rho(\alpha) =  \red{-} \Re r^2_{\C}(\alpha_z, \Re z) = \red{-} \frac{1}{4} r_{\C}^2(\alpha_z, \bar{\alpha_z}) =  |\alpha_{\xi}|_{\alpha_x}^2$$
where,
$$\alpha_z = \exp_{\alpha_x} ( -i \alpha_{\xi}).$$

Using this observation and the expansion in Proposition \ref{kernel} it is proved in \cite [Theorem 0.1]{GLS}  that the operator $T_{hol}(h): C^{\infty}(M) \to C^{\infty}(M_{\tau}^{\C})$ given by
\begin{equation} \label{holFBI}
T_{hol} \phi_h (\alpha) = h^{-{\red{n}}/4} \int_{M} e^{ [ - r_{\C}^2(\alpha_z,y)/2 - \rho(\alpha_z)]/h} a^{\C}(\alpha_z,y,h) \chi(\alpha_x,y) \phi_h(y) dy, \quad \alpha \in M_{\tau}^{\C} \end{equation}\
is also an FBI transform in the sense of \eqref{FBI} with amplitude $a \in S^{m/2,0}_{cla}$ and phase function
\begin{align} \label{hol phase}
\phi(\alpha,y) &= i  \Big( \, \frac{ r_{\C}^2(\alpha_z,y)}{2} + \rho(\alpha) \, \Big). 
\end{align}\
In \eqref{holFBI} the multiplicative factor $h^{-{\red{n}}/4}$ is added to ensure $L^2$-normalization so that  \newline $\| T_{hol} \phi_h \|_{L^2(M_{\tau}^{\C})} \approx  1.$
The fact that the transform $T_{hol}$ is compatible with the complex structure of the Grauert tube $M_{\tau}^{\C}$ will be used a crucial way in the proof of the $h$-microlocal, shortrange weighted $L^2$ bounds in Proposition \ref{weightedl2}.

%In terms of the complex manifold, $\tilde{X}$ there is a natural strictly plurisubharmonic exhaustion function given by 
%$$ \Phi(z) = \frac{ |\alpha_\xi|^2_{g}}{2}, \,\, \, z = \exp_{\alpha_x}(i \alpha_{\xi}).$$
%Fix $\delta >0$ a small constant. Consider the manifold 

%\begin{equation} \label{ILag}
%\Lambda(\delta):= \{ \beta \in \widetilde{T^*X}; \,\, \beta = \exp (i \delta) (\alpha), \,\, \alpha \in T^*X, \,\, |\alpha_{\xi}|_x < \delta \}. \end{equation}

%\bigskip

%Since the complex symplectic form $\omega$ is invariant under $\exp (i \delta H_p)$ it follows that for $\delta >0$ small, $\Lambda(\delta)$ is both I-Lagrangian and R-symplectic. It is a small complex perturbation of the totally-real R-symplectic submanifold $T^*X \subset \widetilde{T^*X}.$

%In the folllowing, we consider an analytic, elliptic, self-adjoint $h$-pseudodifferential operator $\Lambda(h): C^{\infty}(M) \to C^{\infty}(M)$ and let
%$$p(\alpha_x,\alpha_{\xi}) := ( \gamma(\alpha)  - 1 )^2, \,\,\, \gamma(\alpha) = \sigma( \Lambda(h)) (\alpha).$$

%We assume throughout that $\gamma$ has simple-characteristics with

%$$ d \gamma(\alpha) |_{\gamma^{-1}(0)}  \neq 0.$$\ 

Since $P(h)$ has simple characteristics and is classically elliptic  $ p^{-1}(0)$  is a compact, real-analytic hypersurface and by assumption, $p^{-1}(0) \subset M_{\tau}^{\C}.$

%With some abuse of notation, denote the holomorphic continuation to $ \widetilde{T^*M}  |_{\tilde{M}}$ also by $p.$
%The FBI transform extend to an operator $T_{\Lambda}(h): C^{\infty}(X) \rightarrow C^{\infty}(\Lambda)$

\subsubsection{Long-range estimates} \label{longrange}

Let $p(\alpha, h)\sim \sum_{j=0}^\infty p_j(\alpha)h^j\in S^{m}$ be the full symbol of $P(h)$ and assume that it lies in $S^{0,m}_{cla}(W)$ where $W$ is a neighborhood of $(x_0,\xi_0)$. Here, $\xi_0$ is allowed to be a point at infinity in which case a neighborhood means a conic neighborhood of $\xi_0$ near infinity. We say $p$ is elliptic at $(x_0,\xi_0)$ if $|p_0|\geq c\lan \xi\ran^m>0$ in a neighborhood of $(x_0,\xi_0)$. 
\begin{prop}
\label{p:longRange}
Suppose that $P(h)$ is a \red{semiclassical pseudodifferential} operator analytic in a neighborhood of $(x_0,\xi_0)$ and elliptic at $(x_0,\xi_0)$. Suppose that 
$$P(h)\varphi_h=O{}(e^{-c/h}),\quad \|\varphi_h\|_{L^2}=1.$$
Then, for any FBI transform $T(h)$, there exists $c>0$ and $W$ a neighborhood of $(x_0,\xi_0)$ so that 
$$\|T(h)\varphi_h\|_{L^2(W)}=O(e^{-c/h}).$$
\end{prop}
\begin{proof}
Let $\chi_1\in C^{\infty}_0(T^*M)$ so that $\chi_1\equiv 1$ near $(x_0,\xi_0)$ and $p$ is elliptic and analytic on $\supp \chi_1$. \red{Let $T(h)$ be an FBI transform with symbol $r\in S_{cla}^{3n/4,n/4}$ and phase function $\varphi$.} An application of analytic stationary phase \cite{Sj}  gives
 \begin{equation} \label{sp1}
 \chi_1(\alpha)T(h) P(h)  \phi_h = \chi_1(\alpha)T_{b}(h) \phi_h + O(e^{-C \langle \alpha_{\xi} \rangle / h}) \end{equation}
 where
 $$ T_{b}(h) \phi_h (\alpha) = \int_M e^{i\phi(\alpha,y)/h} b(\alpha,y,h) \chi (|\alpha_x-y|) \phi_h(y) dy,$$
 with 
 \begin{equation}
 \label{e:persimmon}
 b(\alpha,y,h) = \sum_{j=0}^{\red{C}h^{-1}} \tilde{p}_{j}(y, \red{\alpha},-d_y \phi(\alpha,y) ) h^j, \hspace{5mm} \tilde{p}_{j} \in S^{3{\red{n}}/4-j, {\red{n}}/4+\red{m}-j}_{cla}
 \end{equation}
 $$ \tilde{p}_0\red{(y,\alpha,\eta)} = \red{r_0(\alpha,y)p_0(y,\eta)}.$$
 \red{Here, $r_0$ is the principal symbol of $r$ from~\eqref{scsymbol}.}
Since
$$ - d_{y} \phi(\alpha,y) = \alpha_{\xi} \, ( 1 + O(|\alpha_x-y|) )$$
and $|\alpha_x -y| < \delta \ll 1$ on supp $\chi(\alpha_x-y)$ it follows that
$p_0(y,-d_y \phi)$ is $h$-elliptic near $\alpha=(x_0,\xi_0)$. In particular, 
\begin{equation}\label{comparable}
 \frac{1}{C} \langle \alpha_{\xi} \rangle^m \leq  | p_0(y,-d_y \phi(\alpha,y)) | \leq C \langle \alpha_{\xi} \rangle^m,\qquad \alpha\text{ near }(x_0,\xi_0). \end{equation}\
%
%Let $\psi \in S^{0}(1)$ be a weight function with $\| \psi \|_{C^0(T^*M)}$ sufficiently small but independent of $h$. 
%
%Let $\chi_{in} \in C^{\infty}_{0}(M_{\tau}^{\C})$  with $\chi_{in}(\alpha) =1$ for $\alpha \in M_{\tau/2}^{\C}$ and $\chi_{out} = 1- \chi_{in}.$
%

Then, from \eqref{sp1}  and the eigenfunction equation $P(h) \phi_h =O(e^{-c/h}),$ it follows that
\begin{equation} \label{lr1}
 \| \chi_1 T_{b}(h) \phi_h (\alpha)\|_{L^2} = O(e^{-C \langle \alpha_{\xi} \rangle / h}). \end{equation} \
 We claim that \eqref{lr1} is independent of FBI transform; in particular,
 \begin{equation} \label{lr2}
  \| \chi_1 T_{geo}(h) \phi_h \|_{L^2} = O(e^{-C/h}). \end{equation} \

%Since $p(\alpha)$ is $h$-elliptic on supp $\chi_{out},$ one can construct an analytic $h$-psdo $R(h) \in Op_{h}(S^{0}_{cla}(T^*M))$ with $r \sim \sum_j r_j,  r_{j} \in S^{-j}_{cla}$ so that 
Since $\chi_1(\alpha)b(\alpha,y,h)$ is $h$-elliptic near $(x_0,\xi_0,x_0)$, \cite[Proposition 6.2]{Sj} proves the estimate. We review the proof here for the reader's convenience. The operator given by 
$$Au(x)=h^{-3n/2}\iint\exp\left[\frac{i}{h}\left(\varphi(\alpha,x)-\overline{\varphi({\alpha},y)}\right)\right]b(\alpha,y)\chi_1(\alpha)\chi(|\alpha_x-y|)\chi(|\alpha_x-x|)u(y)dyd\alpha$$
is an $h$-pseudodifferential operator with elliptic symbol near $(x_0,\xi_0)$. So, for any $a(\alpha,x)$, supported near $(x_0,\xi_0,x_0)$, we can find a classical analytic symbol, $\tilde{b}$ defined near $(x_0,\xi_0,x_0)$ so that 
$$A(\tilde{b}(\alpha,\cdot)e^{i\varphi/h})=a(\alpha,x)e^{i\varphi/h}$$
modulo exponential errors.
In particular, for $W$ a small enough neighborhood of $(x_0,\xi_0)$,
\begin{align*}
a(\alpha,x)e^{i\varphi/h}=\int_{\beta\in W} b(\beta, x)e^{i\varphi(\beta,x)/h} K(\alpha,\beta)d\beta
\end{align*}
where 
$$K(\alpha,\beta)=h^{-3n/2}\iint\exp\left[\frac{i}{h}\left(\varphi(\alpha,y)-\overline{\varphi(\beta,y)}\right)\right]\tilde{b}(\alpha,y)dy.$$
By an application of analytic stationary phase,
$$ K(\alpha,\beta) = e^{i \Phi(\alpha,\beta)/h} c(\alpha,\beta;h) \chi(\alpha_x- \beta_x)  \chi( |\alpha_{\xi}|^{-1} |\alpha_{\xi} - \beta_\xi|)  + O_N( e^{-C/h} ) \max \{ |\alpha_{\xi}|,|\beta_{\xi}| \}^{-N},$$ 
where, $ \Im \Phi(\alpha, \beta) \geq  |\alpha - \beta|^2,$ and $c \in S^{\red{n},\red{\infty}}_{cla}.$

In particular, we can write 
$$\chi_2(\alpha)T_{geo}= K \chi_1(\alpha)T_b+R(h)$$\
where $K$ is tempered in $h$,
$$ |\partial_{x}^{\alpha} \partial_{y}^{\beta} R(x,y,h)| = O_{\alpha,\beta}(e^{-C/h}), \,\, C>0,$$\
and \red{$\chi_1\equiv 1$ on $\supp \chi_2$. Henceforward, we write $\chi_2 \Subset \chi_1$ to denote this.}
Thus,
$$ \chi_2T_{geo} \phi_h =    K\chi_1 T_b  \phi_h + O(e^{-C/h}) =O_{}(e^{-C/h}).$$\

\end{proof}

\subsubsection{Short-range estimates}

%In the following we let $U_j; j=1,...,N$ be a open covering of $X$ by coordinate charts and let $\chi_{j}; j=1,,,.N$ be a partition of unity subordinate to this covering.

Let $\chi_{in} \in C^{\infty}_{0}(M_{\tau}^{\C};[0,1])$ and  $\tilde{\chi}_{in} \in C^{\infty}_{0} (M_{\tau}^{\C}; [0,1])$ be a cutoff with $\tilde{\chi}_{in} \Supset \chi_{in}$. 

To deal with the shortrange case, using analytic stationary phase one constructs an $h$-pseudodifferential  intertwining operator $Q(h) \in Op_h(S^{0,\infty}(T^*M))$ that is $h$-microlocally analytic on the Grauert tube $M_{\tau}^{\C} \subset T^*M$ and satisfies
 \begin{equation} \label{intertwining}
 \chi_{in} T_{hol}(h) P(h) \phi_h  = \chi_{in} Q(h) T_{hol}(h) \phi_h + O(h^{\infty}) \| \tilde{\chi}_{in} T_{hol}(h) \phi_h \|_{L^2} + O(e^{-C/h}).\end{equation}\
 %
% Moreover, in the second error term on the RHS of \eqref{intertwining}) the $O(h^{\infty})$ can be improved $O(e^{-C'/h})$ with some $C'>0$, but we will not need this in what follows.
 %
 
 To construct $Q(h)$ in \eqref{intertwining}, using  \eqref{sp1} we write 
 $$T_{hol}(h) P(h) \phi_h = T_{b}(h) \phi_h + O(e^{-C/h})$$ 
 where $b \in S^{3{\red{n}}/4,{\red{n}}/4+\red{m}}_{cla}(T^*M)$ \red{is given by~\eqref{e:persimmon}.} Then, the symbol $q(\alpha,\alpha^*;h) \sim_{h \to 0^+} \sum_{j=0}^{\infty} q_j(\alpha, \alpha^*) h^j$ of $Q(h)$ is determined by solving the equations
 \begin{align} \label{intertwine}
 q_j(\alpha, d_{\alpha} \phi(\alpha,y)) &= \tilde{p}_j (y,\red{\alpha}, -d_y \phi(\alpha,y));        \,\, j=0,1,2,3,... \nonumber \\
 \tilde{p}_0(y,\red{\alpha},\eta) &= \red{r_0(y,\alpha)}p_0(y,\eta) \end{align}\
 \red{where $\alpha^*$ is the dual coordinate to $\alpha$ and $\tilde{p}_j$ are determined as in~\eqref{e:persimmon}.}
To solve for the $q_j$  in \eqref{intertwine}, we first consider the complexified equations
\begin{equation} \label{intertwine complex}
 q_j^{\C}(\alpha, d_{\alpha} \phi^{\C}(\alpha,y)) = \tilde{p}_j^{\C}(y, -d_y \phi^{\C}(\alpha,y));        \,\, j=0,1,2,3,... \end{equation}
\red{Let $\widetilde{M_{\tau}^{\C}}$ be a complex extension of $M_{\tau}^{\C}$.} Here, $^\C$ denotes holomorphic continuation to $(\alpha,y) \in \widetilde{ M_{\tau}^{\C} } \times M_{\tau}^{\C}$ with  $|\alpha_x - y| < \ep_0$ 
\red{ and $M_{\tau}^{\C} = \{ \alpha; \sqrt{\rho}(\alpha) \leq \tau \}$ is identified with the complex thickening $\tilde{M}.$}
 Since
 \begin{align} \label{phase asymptotics} 
 \partial_{\alpha_x} \phi^{\C} &= - \partial_y  \phi^{\C} + O(|\alpha_x -y|) = \alpha_{\xi} + O(|\alpha_x-y|), \nonumber \\
 \partial_{\alpha_{\xi}} \phi^{\C} &= \alpha_{x}-y + O(|\alpha_x-y|^2), \end{align}
  it follows that near $\alpha_x=y$,
$\det \partial^2_{y \, \alpha_\xi}\varphi\neq 0$ and so by the holomorphic implicit function theorem, $d_{\alpha_{\xi}} \varphi(\alpha,y)=w$ defines $y=\beta_x^{\C}(\alpha,w)$ with $\beta_x^{\C}$ holomorphic in a neighborhood of $\alpha_x=y.$  Hence, restricting to  real points $(\alpha,y) \in M_{\tau}^{\C} \times M$ with $|\alpha_x -y| < \ep_0,$ we can write 
\begin{equation} \label{intertwinesymbols}
q_j(\alpha,\alpha^*)=\tilde{p}_j(\beta_x(\alpha,\alpha_{\xi}^*),-d_y\varphi(\alpha,\beta_x(\alpha,\alpha_\xi^*)))=\tilde{p}_j(\beta_x(\alpha,\alpha_\xi^*),\beta_\xi(\alpha,\alpha_{\xi}^*)); j=0,1,2,..., \end{equation} \
where $\beta_x$ and $\beta_{\xi}$ are locally $C^{\omega}.$ 

%The fact that the $q_j$'s are analytic symbols satisfying Cauchy estimates $|q_j| \leq C_0 C^j  j\!$ follows from \eqref{intertwinesymbols} since the $\tilde{p}_j$'s are analytic.

Since $\tilde{p}_0 = p_0,$ for the principal symbols one gets
\begin{equation} \label{toeplitzmultiplier}
q_0(\alpha,\alpha^*) = p_0(\beta_x(\alpha,\alpha_\xi^*),\beta_\xi(\alpha,\alpha_{\xi}^*)). \end{equation} \
It will be useful to introduce the principal symbol of conjugated operator $e^{\psi/h} Q(h) e^{-\psi/h}$ given by
\begin{equation} \label{conjugmultiplier}
q_0^{\psi}(\alpha,\alpha^*) := q_0(\alpha, \alpha^* + i d_{\alpha} \psi(\alpha)). \end{equation}\

Now, $\partial^2_{y, \alpha_\xi} \varphi|_{y=\alpha_x}=-\Id$, and since $\partial_y^2r(\alpha,y)|_{y=\alpha_x}=2 \Id$, $\partial^2_y\varphi|_{y=\alpha_x}=i\Id$. Therefore, using also that $$d_{\alpha_\xi}\varphi(\alpha,y)=\alpha_x-y+O_{}(|\alpha_x-y|^2),$$
It follows from \eqref{phase asymptotics} that,
\begin{equation} \label{twist}
\beta_{x}(\alpha,\alpha_{\xi}^*) = \alpha_x-\alpha_\xi^*+O(|\alpha_\xi^*|^2), \quad \beta_{\xi}(\alpha, \alpha_{\xi}^*) = \alpha_{\xi}+i\alpha_{\xi}^*+O(|\alpha_\xi^*|^2). \end{equation} 
for $(\alpha,y) \in M_{\tau}^{\C} \times M$ with $r(\alpha_x,y) < \ep_0.$
%Consequently, by the holomorphic implicit function theorem, for $\ep_0>0$ small,
%
%
%
%
%Thus, the  holomorphic mapping $\Omega_{1}^{\C} \to \Omega_{2}^{\C}$ induced by the projection $(\alpha,y) \mapsto (\alpha_\xi,y)$ is a complex
%submersion for $\ep_0>0$ small. So by the holomorphic implicit function theorem one can solve the successive equations in \eqref{intertwine complex}) and consequently, the real equations in \eqref{intertwine}) by taking real parts. Moreover, by a simliar implicit function theorem argument  one observes that $(\alpha,y) \mapsto (\alpha, d_{\alpha_\xi}\phi)$ is a local biholomorphism. Taking real parts, it follows that there exist locally-defined real-analytic functions $\beta_x, \beta_{\xi} \in C^{\omega}$ such that
%
%\begin{align} \label{IFT}
% q_0(\alpha, \alpha^*) &= p_0( \beta_x(\alpha,\alpha_{\xi}^*), \beta_{\xi}(\alpha,\alpha_{\xi}^*)) \\  \nonumber \\  \nonumber
 %\beta_{x}(\alpha,\alpha_{\xi}^*=0) &= \alpha_x, \,\, \beta_{\xi}(\alpha, \alpha_{\xi}^*=0) = \alpha_{\xi}. \end{align}
%
As an example, we note that in the $\R^n$ case with standard phase function $\phi(\alpha,y) = (\alpha_x-y) \alpha_{\xi} + \frac{i}{2} |\alpha_x-y|^2,$ one has
$\beta_x = \alpha_x - \alpha_{\xi}^*$ and $\beta_{\xi} = \alpha_{\xi}+i\alpha_{\xi}^*.$

Given the intrinsic complex structure on the Grauert tube $M_{\tau}^{\C},$  we denote the associated Cauchy-Riemann operators  by $\partial: C^{\infty}(M_{\tau}^{\C}) \to \Omega^{1,0}(M_{\tau}^{\C})$ and $\overline{\partial}: C^{\infty}(M_{\tau}^{\C}) \to \Omega^{0,1}(M_{\tau}^{\C}).$  Moreover, given local coordinates $\alpha_x$ in a chart $U \subset M,$  the corresponding complex coordinates in $U^{\C} \subset M_{\tau}^{\C}$  will be denoted by $\alpha_{z}:= \exp_{\alpha_x}(-i \alpha_{\xi}), \,\, \bar{\alpha_z}  = \exp_{\alpha_x}(  i \alpha_{\xi}).$

Given a smooth one-form $\theta \in \Omega^{1}(M_{\tau}^{\C})$ one can write it in local $(\alpha_x,\alpha_{\xi})$-coordinates in the form
$$\theta = \alpha_x^* d\alpha_x + \alpha_{\xi}^* d \alpha_{\xi}$$
and in terms of complex coordinates $(\alpha_z, \bar{\alpha_z})$ as
$$\theta = \alpha_z^* d\alpha_z + \bar{\alpha_z}^* d  \bar{\alpha_z}.$$
Consequently, in terms of the Cauchy-Riemann operators, $\alpha_z^* = \sigma (\partial)(\alpha)$ and $ \bar{\alpha_{z}}^* =  \sigma(\bar{\partial})(\alpha).$ \red{Here, $\sigma$ denotes the principal symbol of a pseudodifferential operator.}

Given a weight function $\psi \in C^{\infty}(M_{\tau}^{\C})$ and the strictly plurisubharmonic weight $\rho(\alpha_{z},\bar{\alpha_z}) =\frac{ |\alpha_{\xi}|_g^2}{2},$ we consider the associated submanifold $\Lambda \subset \widetilde{M_{\tau}^{\C}} \,$  given by  \

\begin{equation} \label{ILAG}
\Lambda = \{  \,(\alpha, 2i \partial\psi(\alpha)+i(\dbar -\partial)\rho(\alpha) ), \,\, \alpha \in M_{\tau}^{\C} \}. \end{equation}\

As we shall see below, the manifold $\Lambda$ will play an important role in our main exponential weighted estimate in Proposition \ref{weightedl2}.

For future reference, we note that in terms of the local complex coordinates $(\alpha_z, \bar{\alpha}_z)$ in a geodesic normal coordinate chart $U,$ 
\begin{equation} \label{ILAGlocal}
\Lambda |_{U} := \{ ( \, \alpha_z, \bar{\alpha_z}; \alpha_z^*=(2i\partial_{\alpha_z}\psi- i \partial_{\alpha_z}\rho)(\alpha), \, \bar{\alpha_z}^* = i \overline{\partial}_{\alpha_z}\rho(\alpha)   ), \,\, \alpha \in M_{\tau}^{\C} \} \end{equation}\

\subsubsection{ Complex geometry of $\Lambda$}

We first recall some basic complex symplectic geometry: Let $X$ be a complex $n$-dimensional manifold with complex cotangent bundle $T^*X.$ Viewing $X$ as a real-analytic manifold, we let $T^*X_{\red{\R}}$ denote the real $\red{4}n$-dimensional  cotangent bundle. There is a natural identification \cite{Leb} of
$ T^*X_{\R}$ with  $T^*X$  given as follows. Let $v \in TX$ (a complex tangent vector) and $(z, \zeta) \in T^*X$ (a complex covector). Then, the identification $\iota: T^*X \to T^*X_{\R}$ is given by
$$ \iota (z, \zeta(v)) = (z, \xi(v)); \quad \xi(v) = \Re \zeta(v).$$
In terms of local  coordinates $(\Re z, \Im z): X \to \R^{2n}$  and the corresponding dual coordinates $(\xi,\eta) \in T_{(\Re z, \Im z)}^*X_{\R},$
$$ \iota (\Re z, \Im z;\xi,\eta) = (z, \zeta), \quad \zeta = \xi + i \eta.$$
Let $\Gamma \subset T^*X$ be $I$-Lagrangian with respect to the complex symplectic form $\Omega = dz \wedge d \zeta.$ Then, for any contractible coordinate chart $U,$ we recall  that \cite[Lemma 3.1]{Leb} using the identification $\iota,$ one can locally characterize $\Gamma$ as the graph of complex differential; that is,
\begin{equation} \label{ilag}
 \Gamma |_{U} = \{ (z, 2i \partial_{z} f), \,\, z \in U \}, \end{equation}
with $f  \in C^{\infty}(U;\R)$ and $ 2 \partial_z f = ( \partial_{\Re z} + i \partial_{\Im z} ) f.$

%In view of the local charaterization in \eqref{ILAGlocal}), 

%$$ \iota ( \Lambda |_{U}) = \{ (\, \alpha_z, 2i\partial_{\alpha_z} ( ? )  \,), \,\, \alpha_z \in U \},$$ t\

We claim that $\Lambda \subset  T^* ( M_{\tau}^{\C})$ in \eqref{ILAG} can be naturally identified with an  $I$-Lagrangian with respect to the canonical complex symplectic form.  More precisely, consider
$$ \tilde{\Lambda}: = \Lambda - i  \, \text{graph}( d\rho) \subset \widetilde{M_{\tau}^{\C}}.$$
To see that $\tilde{\Lambda}$ is indeed $I$-Lagrangian, we note that since $d = \partial + \dbar,$ one can write
$$ 2 i \partial \psi  + i(\dbar - \partial) \rho = 2i \partial (\psi - \rho)  + i d \rho$$ and so,
$$ \tilde{\Lambda} = \{ (\alpha,2i \partial (\psi - \rho)(\alpha)), \,\, \alpha \in M_{\tau}^{\C} \}.$$
Consequently, in view of \eqref{ilag}, $\tilde{\Lambda}$ is indeed $I$-Lagrangian. Morever, since $\rho$ is strictly plurisubharmonic with $ \partial \dbar \rho >0$, it follows that with $\| \psi \|_{C^2}$ sufficiently small, $\tilde{\Lambda}$ is also $\R$-symplectic.

We note that clearly one can write the Toeplitz multiplier $q_0 |_{\Lambda}$ on the RHS of Proposition \ref{weightedl2} as $\tilde{q_0} |_{\ \tilde{\Lambda}}$ where $\tilde{q_0}(\alpha,\alpha^*) = q_0(\alpha, \alpha^* + i d_{\alpha} \rho(\alpha))$ and $\tilde{\Lambda}$ is the $I$-Lagrangian above. However,  we find working with $q_0$ \red{(from~\eqref{toeplitzmultiplier}) }and $\Lambda$ \red{(from~\eqref{ILAG})} computationally simpler and so we continue to work throughout with these instead.

\begin{rem} \red{Here we call $q_0$ a Toeplitz multiplier in reference to the corresponding Toeplitz operator $S_{hol}q_0T_{hol}$  (see e.g. \cite[Chapter 13.4]{Zw}).} 
\end{rem}

%We note that in the case of $\R^n$, 
%$$ \Lambda = \{ (\alpha_x, \alpha_{\xi};  \partial_{\alpha_{\xi}}\psi - \alpha_{\xi},   \partial_{\alpha_x} \psi ); \,\, (\alpha_{x},\alpha_{\xi}) \in \R^{2n} \} .$$
%Consequently, one directly verifies that $\Lambda$ is $I$-Lagrangian since
%$$ \Im \Omega |_{\Lambda} = d\alpha_{x} \wedge d \partial_{\alpha_x} \psi + d \alpha_{\xi} \wedge d (\partial_{\alpha_{\xi} } \psi - \alpha_{\xi} ) = 0$$

The following $h$-microlocal manifold version of the  microlocal Agmon estimates in $\R^n$  \cite{Mar,Na}  is central to the proofs of Theorems \ref{mainthm}, \ref{mainthm2}.\\ 
\begin{prop} \label{weightedl2}
%Let $\Lambda(\delta)\subset \widetilde{T^*X}$ be the I-Lagrangian in \eqref{ILag}. Then,
 Let $\Lambda \subset \widetilde{M_{\tau}^{\C}}$ be as in \eqref{ILAG}.
Then, for any  $P(h) \in Op_h (S^{0,\infty}_{cla})$ \red{there exists $\delta>0$ so that for} $\psi \in S^0(1)$ with $\| \psi  \|_{C^1}\red{<\delta}$,
\begin{multline*}
 \langle \chi_{in} e^{\psi/h} T_{hol}(h) P(h)   \phi_h,  e^{\psi/h} T_{hol}(h) \phi_h \rangle_{L^2} =  \langle \chi_{in} e^{\psi/h} \, q_0 |_{\Lambda} \, T_{hol}(h)  \phi_h, e^{\psi/h}  T_{hol}(h) \phi_h \rangle_{L^2}\\
 +  O(h)  \| \tilde{\chi}_{in} e^{\psi/h} T_{hol}(h) \phi_h \|_{L^2}^2 + O(e^{-C/h}).\end{multline*}

%Here,  $\partial: C^{\infty}(M_{\tau}^{\C}) \to C^{\infty}(M_{\tau}^{\C})$ is the intrinsic Cauchy-Riemann operator on the Grauert tube $M_{\tau}^{\C}.$
\end{prop}

\begin{proof}

The operator $Q_{\psi}(h):= \chi_{in} e^{\psi/h} Q(h) e^{-\psi/h}$ has Schwartz kernel
\begin{align} \label{SK}
Q_{\psi}(\alpha,\beta,h) = (2\pi h)^{-2n}  \lim_{\ep \to 0^+} \int e^{i \langle \alpha - \beta, \alpha^* \rangle /h} e^{- \epsilon \langle \alpha^* \rangle/h} e^{i [ \psi(\alpha) - \psi(\beta)]/h}  \chi_{in}(\alpha) \, q(\alpha,\alpha^*;h) \, d\alpha^*. \end{align}
By Taylor expansion,
$$  \psi(\alpha) - \psi(\beta) = \langle \Psi(\alpha,\beta), \alpha - \beta \rangle$$
with $| \Psi| < \| \psi \|_{C^1} < \delta.$ Since $q(\alpha,\alpha^*,h)$ is analytic, for $\delta >0$ small it follows by Stokes formula one can make the contour deformation 
$$ \alpha^* \mapsto \alpha^* + i \Psi(\alpha,\beta)$$ in \eqref{SK}. Boundary terms as $|\alpha^*| \to \infty$ vanish and one gets that
$Q^{\psi}(h) \in Op_{h}(S^{0}(1))$ with symbol
$$ q^{\psi}(\alpha,\alpha^*,h) \sim \sum_{j=0}^{\infty} \chi_{in}(\alpha) q_j( \alpha, \alpha^* + i d_{\alpha} \psi) h^j.$$
and principal symbol
$$ q^{\psi}_0(\alpha,\alpha^*) = q_{0}(\alpha, \alpha^* + i d_{\alpha} \psi(\alpha)),$$
where $q_0$ is defined in \eqref{toeplitzmultiplier}.

In view of \eqref{intertwining} it follows that
\begin{align} \label{weighted1}
\langle \chi_{in} e^{\psi/h} T_{hol}(h) P(h)   \phi_h, e^{\psi/h} T_{hol}(h) \phi_h \rangle_{L^2}  &=  \langle \chi_{in} Q^{\psi}(h)  \, [e^{\psi/h} T_{hol}(h)    \phi_h ], \, e^{\psi/h} T_{hol}(h) \phi_h \rangle_{L^2} \nonumber \\
& + O(h^{\infty}) \| \tilde{\chi}_{in} e^{\psi/h} T_{hol}(h) \phi_h \|_{L^2}^2 + O(e^{-C/h}). \end{align} \
Next, in analogy with \cite{Mar}, we observe that with $\bar{D}_{\alpha_z} = \frac{1}{i} \bar{\partial}_{\alpha_z},$
\begin{align} \label{invariance}
h \overline{D}_{\alpha_z} \Big( e^{\psi/h} T_{hol}(h) \phi_h \Big) &= \int_{M} h \overline{D}_{\alpha_z} \Big( e^{  [ \, - r^{2}_{\C}(\alpha_z,y)/2 + \psi(\alpha_z,\bar{\alpha_z})  - \rho(\alpha) \, ] /h } a(\alpha_z,y,h) \chi(\alpha_x,y) \Big) \phi_h(y) dy \\  \nonumber
&= \Big(-i\overline{\partial}_{\alpha_z} \psi(\alpha)  + i \overline{\partial}_{\alpha_z} \rho(\alpha) \Big) \, e^{\psi/h} T_{hol}(h) \phi_h  + O(e^{-C(\ep_0)/h}) \\ \nonumber 
&= -  i  \overline{\partial}_{\alpha_z} \tilde{\psi}(\alpha)  \, e^{\psi/h} T_{hol}(h) \phi_h  + O(e^{-C(\ep_0)/h}),
\end{align}
where we have written 
$$\tilde{\psi}=\psi-\rho.$$
The last line in \eqref{invariance} follows since $r^{2}_{\C}(\cdot,y)$ and $a(\cdot,y,h)$ are holomorphic and the exponential error arises from differentiation of the cutoff $\chi(\alpha_x,y).$ \red{In particular, when $r(\alpha_x,y)\geq \e_0>0$, there exists $C(\e_0)>0$ so that}
$$ \red{\Re}  [ \, - r^{2}_{\C}(\alpha_z,y)/2 + \psi(\alpha_z,\bar{\alpha_z})  - \rho(\alpha) ] \red{\leq -} C(\e_0)\red{<}0.$$

Similarly, taking complex conjugates, one gets that

\begin{align} \label{invariance2}
hD_{\alpha_z} \Big( \overline{ e^{\psi/h} T_{hol}(h) \phi_h } \Big) &=  i \partial_{\alpha_z} \tilde{\psi}(\alpha)  \, \overline{ e^{\psi/h} T_{hol}(h) \phi_h}  + O(e^{-C(\ep_0)/h}).
\end{align}

  Taylor expansion of the principal symbol $q^{\psi}_0$ of $Q_{\psi}(h)$ around $\bar{\alpha_z}^* = - i \overline{\partial}_{\alpha_z} \tilde{\psi}$  and $\alpha_z^* = i \partial_{\alpha_z} \tilde{\psi}$ gives 
  \begin{align} \label{taylor}
q_0^{\psi}(\alpha,\alpha^*) &= q_0 (\alpha, \alpha^* + i d_{\alpha} \psi(\alpha)) \nonumber \\ \nonumber \\
&= q_0(\alpha_z,\bar{\alpha_z}; \alpha_z^* = i \partial_{\alpha_z} (\psi-\rho) + i \partial_{\alpha_z} \psi,  \, \bar{\alpha_z}^* = - i \overline{\partial}_{\alpha_z}(\psi -\rho) + i \overline{\partial}_{\alpha_z}\psi )  \nonumber \\ \nonumber
&\qquad+ r_1(\alpha,\alpha^*) ( \bar{\alpha}_{z}^* + i \overline{\partial}_{\alpha_z} \tilde{\psi} ) + r_2(\alpha,\alpha^*) ( \alpha_{z}^* - i \partial_{\alpha_z} \tilde{\psi} ) \\ \nonumber \\ \nonumber
&= q_0 (\alpha_z, \bar{\alpha_z}; \alpha_z^* = 2 i \partial_{\alpha_z} \psi - i \partial_{\alpha_z} \rho, \, \bar{\alpha_z}^* =  i \overline{\partial}_{\alpha_z}\rho)  \nonumber \\
&\qquad+ r_1(\alpha,\alpha^*) ( \bar{\alpha}_{z}^* + i \overline{\partial}_{\alpha_z} \tilde{\psi} ) + r_2(\alpha,\alpha^*) ( \alpha_{z}^* - i \partial_{\alpha_z} \tilde{\psi} ) \\ \nonumber \\ \nonumber
 &=  q_{0}(\alpha) |_{\alpha \in \Lambda} + r_1(\alpha,\alpha^*) ( \bar{\alpha}_{z}^* + i \overline{\partial}_{\alpha_z} \tilde{\psi} ) + r_2(\alpha,\alpha^*) ( \alpha_{z}^* - i \partial_{\alpha_z} \tilde{\psi} )  \end{align}
Since $r_1,r_2 \in S^{0}(1),$ it then follows from \eqref{weighted1}-\eqref{taylor}, $L^2$-boundedness of pseudodifferential operators, and that $[Op_h(S^{m_1}),Op_h(S^{m_2})]\in hOp_h(S^{m_1+m_2})$ (see for example \cite[Chapters 4,9]{Zw}) that

\begin{align} \label{toeplitz}
\langle \chi_{in} Q_{\psi}(h)  \, [e^{\psi/h} T_{hol}(h)    \phi_h ] \,, e^{\psi/h} T_{hol}(h) \phi_h \rangle_{L^2}  = 
\langle \chi_{in} e^{\psi/h} q_{0}(\alpha) |_{\alpha \in \Lambda} T_{hol}(h) \phi_h,   \, e^{\psi/h} T_{hol}(h) \phi_h \rangle_{L^2} &\nonumber \\ \nonumber \\
 + \langle \chi_{in} r_1(\alpha,hD_\alpha) [ h\overline{D}_{\alpha_z}  + i \overline{\partial}_{\alpha_z} \tilde{\psi} ] e^{\psi/h} T_{hol}(h) \phi_h, \, e^{\psi/h} T_{hol}(h) \phi_h \rangle_{L^2}& \nonumber \\ \nonumber \\
  + \langle \chi_{in}   e^{\psi/h} T_{hol}(h) \phi_h, \, r_2( \alpha ,h D_{\alpha}) [ h\overline{D}_{\alpha_z}  + i \overline{\partial}_{\alpha_z} \tilde{\psi} ]  e^{\psi/h} T_{hol}(h) \phi_h \rangle_{L^2}  
+ O(h) \| \tilde{\chi}_{in} e^{\psi/h} T_{hol}(h) \phi_h \|_{L^2}^2.& \nonumber \\ \nonumber \\
=\langle \chi_{in} e^{\psi/h} q_{0}(\alpha) |_{\alpha \in \Lambda} T_{hol}(h) \phi_h,   \, e^{\psi/h} T_{hol}(h) \phi_h \rangle_{L^2} 
 + O(h) \| \tilde{\chi}_{in} e^{\psi/h} T_{hol}(h) \phi_h \|_{L^2}^2. & \end{align}
 This finishes the proof of the Proposition. \end{proof}

\subsection{Microlocal eigenfunction decay estimates}

\subsubsection{Estimation of the multiplier $q_0 |_{\Lambda}$}\ \
In order to prove Theorem \ref{mainthm2} we give an invariant characterization of the Toeplitz multiplier $q_0 |_{\Lambda}.$ 

In terms of the local coordinates $(\alpha_z, \bar{\alpha}_z)$ the Toeplitz multiplier in Proposition \ref{weightedl2} is 
$$q(\alpha_z,\bar{\alpha_z},i\partial_{\alpha_z}(2\psi-\rho),i\partial_{\bar{\alpha}_z}\rho).$$
To give this in invariant meaning, we note that  function $q_{0} \in C^{\omega}(T^*M_{\tau}^{\C})$ and 
$$(\alpha_z,\bar{\alpha_z},i\partial_{\alpha_z}(2\psi-\rho),i\partial_{\bar{\alpha}_z}\rho)$$
 are local coordinates for the point 
\begin{align*} 
\zeta &=    i\partial_{\alpha_z}(2\psi-\rho)(\alpha)  \, d\alpha_z + i\partial_{\bar{\alpha}_z}\rho(\alpha) \,  d\bar{\alpha}_z   \in T^*M_{\tau}^{\C}.\\
&=2i\partial\psi(\alpha) +i(\dbar-\partial)\rho(\alpha) \in T^*M_{\tau}^{\C}.
\end{align*}
Consequently, the Toeplitz  multiplier  equals
 \begin{align} \label{multiplier}
 & q_0^{\psi}( \alpha, i \partial \tilde{\psi}(\alpha) - i \dbar \tilde{\psi}(\alpha)) 
  = q_0 (\alpha, i (\partial - \dbar) \tilde{\psi}(\alpha)  + i d \psi (\alpha) )  \nonumber \\
& = q_0 (\alpha, 2 i \partial \psi(\alpha) +  i(\dbar - \partial) \rho(\alpha) ) = q_0 |_{\Lambda}(\alpha), \quad \alpha \in M_{\tau}^{\C}, \end{align}\
in view of the definition of $\Lambda$ in \eqref{ILAG}, where
\begin{equation} \label{intrinsicmult}
 q_0 |_{\Lambda}(\alpha) = q_0( \alpha,2i\partial\psi(\alpha) - i (\partial-\dbar)\rho(\alpha) ).\end{equation}
Here, $\partial, \dbar: C^{\infty}(M_{\tau}^{\C}) \to \Omega^{1}(M_{\tau}^{\C})$ are the intrinsic Cauchy-Riemann operators. 

The $\rho$ portion, $i\dbar \rho-i\partial \rho$, of the argument on the RHS of \eqref{intrinsicmult} can be readily computed. Given the strictly-plurisubharmonic weight function $\rho(\alpha) = \frac{1}{2} |\alpha_{\xi}|^2_{\alpha_x}$, we recall that \cite[p. 568]{GS1},
\begin{equation} \label{invariant}
i (\bar{\partial} - \partial) \rho = \omega,
\end{equation}\
\noindent where $\omega = \sum_j \alpha_{\xi_j} d\alpha_{x_j}$ is the canonical one-form. Since both sides  are invariant, one can easily verify the identity \eqref{invariant} by computing in geodesic normal coordinates at the center of the coordinate chart.

 It follows from \eqref{invariant} that 
 \begin{equation}
 \label{restricted values}
 \begin{aligned} 
 \alpha_x^*(2i\partial\psi -i(\partial-\dbar)\rho)&=\alpha_x^*(2i\partial\psi)+\alpha_\xi \\
 \alpha_\xi^*(2i\partial\psi-i(\partial-\dbar)\rho)&=\alpha_\xi^*(2i\partial\psi).
 \end{aligned}
 \end{equation}

%implies that
%\begin{equation} \label{key}
%\alpha_{\xi}^* = O(|\partial \psi|) + O( \partial_{\alpha_x} \rho). \end{equation}

%In particular, in view of \eqref{IFT},

From \eqref{restricted values} one gets
\begin{align} \label{symbolupshot}
q |_{\Lambda}(\alpha) &= p_0 \Big( \, \beta_x( \alpha, \, \alpha_{\xi}^* ( 2i\partial\psi) ) \, , \beta_{\xi}( \alpha,\alpha_{\xi}^* ( 2i\partial\psi ) \, \Big) \nonumber \\ 
&= p_0 ( \, \alpha_x + O(|\partial \psi|), \, \alpha_{\xi} + O(|\partial \psi|)  \, ), \end{align}\
where in the last equality we have used \eqref{twist}. 

\begin{rem} \label{kahler}
We recall that $\rho$ is a Kahler potential for the intrinsic K\red{a}hler form $\Omega_g$ on the Grauert tube corresponding to the complex structure $J_g$ induced by complexified Riemannian exponential map of the metric $g$; that is, 
$$  \partial \bar{\partial} \rho = \Omega_g.$$

The corresponding exhaustion function $\sqrt{2\rho}(\alpha) = |\alpha_{\xi}|_g$ satisfies homogeneous Monge-Ampere,
$$ \det \Big(  \frac{\partial^2 \sqrt{\rho} }{\partial \alpha_{z_i} \partial \bar{\alpha_{z_j}} } \Big) = 0.$$
We refer to \cite{GS1, LS} for further details.
\end{rem}\\

From Proposition \ref{weightedl2} and \eqref{symbolupshot} we get the following useful estimate.

\begin{prop}\label{useful}
Under the same assumptions as in Proposition \ref{weightedl2},
\begin{align*}
&\langle \chi_{in} e^{\psi/h} T_{hol}(h) P(h)   \phi_h, e^{\psi/h} T_{hol}(h) \phi_h \rangle_{L^2}  = \\
&\qquad\qquad\qquad\qquad\qquad\qquad\qquad\qquad \langle  \chi_{in} e^{\psi/h} \, p(\alpha + O(|\partial \psi |) )  \, T_{hol}(h)  \phi_h, e^{\psi/h} T_{hol}(h) \phi_h  \rangle_{L^2}  \\ 
\\
&\qquad \qquad\qquad\qquad\qquad\qquad\qquad\qquad\qquad +  O(h)  \| \tilde{\chi}_{in} e^{\psi/h} T_{hol}(h) \phi_h \|_{L^2}^2 + O(e^{-C/h}). \end{align*}\
\end{prop}\

\subsection{Microlocal concentration of the eigenfunctions: Proof of Theorem \ref{mainthm2}} \label{mainthm2proof}
In this section, we prove the global  weighted decay estimate in Theorem \ref{mainthm2}.
%Since the construction of the full weight function $\psi_{hol}$ is fairly delicate, we start by proving the first part of the theorem The detailed construction is then carried in subsection \ref{refine} by modifying the approach below.  

\begin{proof}
We first prove the weighted estimate in \eqref{e:global} $h$-microlocally on support of $\chi_{in} \in C^{\infty}_{0}(M_{\tau}^{\C})$ and for the $h$-microlocal FBI transform $T_{hol}(h).$ 
\begin{equation} \label{hol bound}
 \| \chi_{in} e^{\psi/h} T_{hol}(h) \phi_h \|_{L^2(T^*M)} = O(1).\ \end{equation} \

The transform $T_{hol}(h)$ is only defined on the Grauert tube $M_{\tau}^{\C},$ which is generally a proper, bounded subset of $T^*M.$  We then need to prove that the weighted bound in \eqref{hol bound} still holds for $T(h) = T_{geo}(h)$, after possibly shrinking $\delta >0$ somewhat (independently of $h$). 

In the following, we let $P(h) \in Op_h(S^{0,k}_{cla})$ be as in the statement of Theorem \ref{mainthm2} and $\phi_h$ be an exponential quasimode with
$$ P(h) \phi_h = O(e^{-c/h}).$$

\subsubsection{$h$-microlocal bounds for $T_{hol}(h) \phi_h$}

{\red{Recall $\psi=\frac{\delta p^2}{2\langle \xi\rangle^{2k}}.$ For}} $\delta >0$ small,

\begin{equation} \label{comparable}
\frac{1}{C(\delta)} p^2(\alpha_x,\alpha_{\xi}) \leq \Re p^2(\alpha + O(|\partial_{\alpha} \psi|) ) \leq C(\delta) p^2(\alpha_x,\alpha_{\xi}). \end{equation}

\noindent To see that \eqref{comparable} holds, we split into two cases:
 
 \begin{enumerate}
\item[Case (i)] $\{ \alpha; |p(\alpha)| \ll 1\}$ (near the characteristic variety). Here,  we make a Taylor expansion to get
 \begin{equation} \label{taylor1}
 p^2(\alpha + O(|\partial_{\alpha}\psi|) = p^2(\alpha) +  O(\delta) |2p\partial_{\alpha}p|^2 = p^2(\alpha) + O(\delta \, p^2(\alpha)).
 \end{equation}
 
 % &p( \alpha + O(\delta ( |\alpha_\xi|_g^2 -1) ) ) \nonumber \\
% &= | \alpha_{\xi}_g^2 . + O( \delta  ( \alpha_{\xi}^2_{g} -1) ) - 1|^2 
 Consequently, near $p=0$ it follows from \eqref{taylor1} that for $\delta >0$ small,
$$  \Re p^2(\alpha + O(|\partial_{\alpha}\psi|) ) \geq c p^2(\alpha).$$\

\item[Case (ii)] $\{ \alpha; |p(\alpha)| \gtrapprox 1 \}$ (far field). Here we use the fact that $\psi \in S^{0}(1)$ and $ |\partial p^2| \lessapprox \langle \alpha_{\xi} \rangle^{2k-1}$ and just make the first-order Taylor expansion
$$p^2(\alpha +   O(|\partial_{\alpha}\psi|) ) = p^2(\alpha) + O( \delta  \langle \alpha_{\xi} \rangle^{2k-1} ).$$
Since $ p^2(\alpha) \gtrapprox \langle \alpha_{\xi} \rangle^{2k}$ in this range, $\eqref{comparable}$ is also satisfied in this case, provided one chooses $\delta >0$ small.
\end{enumerate}

Since $P(h) \in Op_h(S^{0,\infty}_{cla})$ implies that also $P^2(h) \in Op_h(S^{0,\infty}_{cla}),$ and so Proposition \ref{useful} applies just as well with the latter. 
We note that by \eqref{comparable},  there is a constant $C>0$ such that
$$ \frac{1}{C} p^2(\alpha) \leq \Re p^2|_{\Lambda}(\alpha) \leq C p^2(\alpha), \quad \alpha \in \text{supp}\, \chi_{in}$$ and so, by an application of Proposition  \ref{useful}  with the globally-defined weight function $\psi$ in the statement of Theorem \ref{mainthm2},  it follows that

\begin{align} \label{holomorphic estimate}
\langle \chi_{in} e^{\psi/h} T_{hol}(h) P^2(h)  \phi_h, e^{\psi/h} T_{hol}(h) \phi_h \rangle _{L^2} & = \langle \chi_{in} e^{\psi/h}  p^2|_{\Lambda} T_{hol}(h)  \phi_h, e^{\psi/h} T_{hol}(h) \phi_h \rangle_{L^2}  \nonumber \\  
 &+  O(h)  \| \tilde\chi_{in} e^{\psi/h} T_{hol}(h)  \phi_h \|_{L^2}^2  + O(e^{-C/h})  \nonumber \\ \nonumber \\
  \geq \langle \chi_{in} e^{\psi/h}  p^2 T_{hol}(h)  \phi_h, e^{\psi/h} T_{hol}(h) \phi_h \rangle_{L^2}   
 &+  O(h)  \| \chi_{in} e^{\psi/h} T_{hol}(h)  \phi_h \|_{L^2}^2 + O(e^{-C/h}).  
 \end{align}\
 
 In the last line of \eqref{holomorphic estimate} we have used the long-range estimate \eqref{p:longRange} yet again to write 
$$\| \tilde{\chi}_{in} e^{\psi/h} T_{hol}(h)  \phi_h \|_{L^2}^2 =  \| \chi_{in} e^{\psi/h} T_{hol}(h)  \phi_h \|_{L^2}^2 + O(e^{-C/h}).$$

 %where $\tilde\chi_{in}$ has $\tilde{\chi}_{in}\equiv 1$ on $\supp \chi_{in}$  and we have simply replaced $P(h)$ by $P(h)^2$ in Proposition \ref{useful}.
To bound the LHS in \eqref{holomorphic estimate} note that $A_k(h):=S_{hol}(h) \lan \alpha_\xi \ran^{k}T_{hol}(h)\in  Op_h(S^{0,k}_{cla})$. Hence, $A_{-k}(h)P(h)\in \Op_h(S^{0,0}_{cla})$ and since $P(h)$ is classically elliptic with $\|P(h) \red{\phi_h}\|_{L^2}+\|\red{\phi_h}\|_{L^2}=O(1)$, $A_{k}(h)\red{\phi_h}\in L^2$. In particular,
\begin{align*} 
&\lan \chi_{in} e^{\psi/h}T_{hol}(h)P^2(h) \red{\phi_h}, e^{\psi/h}T_{hol}(h)\red{\phi_h}\ran\\
&=\lan  \chi_{in} e^{\psi/h}\lan \alpha_\xi \ran^{-k}T_{hol}(h)P^2(h) \red{\phi_h}, e^{\psi/h}\lan \red{\alpha_\xi}\ran^kT_{hol}(h)\red{\phi_h}\ran\\
&=\lan \chi_{in} e^{\psi/h}T_{hol}(h)A_{-k}(h)P^2(h) \red{\phi_h},e^{\psi/h}T_{hol}(h)A_k(h) \red{\phi_h}\ran+O(e^{-C/h})\\
&=O(e^{-C/h}),
\end{align*}
since by assumption $P(h) \phi_h = O_{L^2}(e^{-C/h}).$
%Note also that we have apriori from our long range estimates in \eqref{p:longRange}) that 
%$$\|T_{hol}(h)\phi_h\|_{L^2(\{|p|\geq \delta\})}=O(e^{-C/h}).$$

Thus, from \eqref{holomorphic estimate} one gets that for $\delta$ small enough \red{and $C=C(\delta)>0$},
\begin{equation} \label{weight2}
 \langle \chi_{in} e^{\psi/h}  p^2 T_{hol}(h)  \phi_h, e^{\psi/h} T_{hol}(h) \phi_h \rangle_{L^2} = O_{\delta}(h)  \| \chi_{in} e^{\psi/h} T_{hol}(h) \phi_h \|_{L^2}^2+O(e^{-C/h}). \end{equation}\

 We note that for any ${\red{N}}>0,$
 $$ \sup_{  p^2(\alpha) \leq {\red{N}} h } \psi(\alpha) = O(h), $$
 so that $e^{\psi(\alpha)/h} = O(1)$ when $p^2(\alpha) \leq {\red{N}} h.$
 It then follows from \eqref{weight2} that
 \begin{multline} \label{weight3}
 \langle \chi_{in} e^{\psi/h} [  p^2 + O(h)] T_{hol}(h)  \phi_h, e^{\psi/h} T_{hol}(h) \phi_h \rangle_{L^2 (\{ p^2 \geq {\red{N}}h\}) }  \\
 = O_{\delta}(h)  \| \chi_{in} T_{hol}(h)\phi_h \|_{L^2 (\{ p^2 \leq {\red{N}} h \}) }^2  + O(e^{-C/h}) = O(h). \end{multline} 
Choosing ${\red{N}=\red{N(\delta)}}>0$ large enough to absorb the $O(h)$ term on the LHS of \eqref{weight3} (which is independent of ${\red{N}}>0$), it follows that
\begin{equation} \label{weight4}
\| \chi_{in} e^{\psi/h} T_{hol}(h) \phi_h \|_{L^2 (\{ p^2 \geq {\red{N}}h\}) }  = O(1). \end{equation}
Clearly, since $\psi \approx p^2$ near $p=0$ it also follows that
$$ \| \chi_{in} e^{\psi/h} T_{hol}(h) \phi_h \|_{L^2 (\{ p^2 \leq {\red{N}}h\})} = O(1)$$
which finishes the proof of \eqref{hol bound}.

\subsubsection{Weighted bounds in terms of $T_{geo}(h)$}

Since $T_{hol}(h)$ is only $h$-microlocally defined, we need to show that essentially the same weighted bound holds for the globally-defined FBI transform $T_{geo}(h).$
More precisely, in this section we show that with $\ep_0 <1$ sufficiently small (but independent of $h$) the analogue of \eqref{hol bound} holds for the globally-defined FBI transform. That is 
\begin{equation} \label{geobound}
\| \chi_{in} e^{ \ep_0 \psi / h} T_{geo}(h) \phi_h \|_{L^2} = O(1). \end{equation}\
Given \eqref{geobound}, the global result in Theorem \ref{mainthm2} follows (after possibly shrinking $\delta>0$ further) since we have already established the long range bound in \eqref{lr2}.

To prove \eqref{geobound}, we  write
\begin{align} \label{comp}
\chi_{in} e^{ \ep_0 \psi/h} T_{geo}(h) \phi_h &= \red{\chi_{in}e^{\ep_0\psi/h}T_{geo}(h)S_{hol}(h)T_{hol}(h) \phi_h+O(e^{-C/h}) }\nonumber\\
&=\chi_{in} e^{\ep_0 \psi/h} T_{geo}(h) S_{hol}(h) \tilde{\chi}_{in}  T_{hol}(h) \phi_h + O(e^{-C/h}) \\ \ \nonumber \\ \nonumber
&= \chi_{in} e^{\ep_0 \psi/h} T_{geo}(h) S_{hol}(h) e^{-\psi/h}  \tilde{\chi}_{in} e^{\psi/h} T_{hol}(h) \phi_h + O(e^{-C/h}). \end{align}\
\red{Note in the second line we us that $T_{geo}S_{hol}$ is pseudolocal modulo exponential errors.} Then, by an application of analytic stationary phase, after possibly shrinking $\delta >0$,  the Schwartz kernel of the operator $\chi_{in} e^{\ep_0 \psi/h} T_{geo}(h) S_{hol}(h) e^{-\psi/h}$ can be written in  the form

$$ e^{i \Phi(\alpha,\beta)/h} c(\alpha,\beta;h) \chi(\alpha_x- \beta_x)  \chi( |\alpha_{\xi}|^{-1} |\alpha_{\xi} - \beta_\xi|)  + O_N( e^{-C/h} ) \max \{ |\alpha_{\xi}|,|\beta_{\xi}| \}^{-N},$$ 

\noindent where,  
$$ \Im \Phi(\alpha, \beta) \geq - \ep_0 p^2(\alpha) + p^2(\beta) +  |\alpha - \beta|^2,$$\
$c \in S^{{\red{n}},0}_{cla}$.
To estimate $ \Phi(\alpha, \beta)$ when $\alpha$ and $\beta$ are near $p^{-1}(0),$ we note that by Taylor expansion of $p^2(\alpha)$ around $\alpha = \beta$ one gets that
 $$  \Im \Phi(\alpha, \beta) \geq - \ep_0 p^2(\alpha) +  p^2(\beta) + C_0 ( p(\alpha) - p(\beta) )^2,$$ 
 with some $C_0 >0.$ Writing $x = p(\alpha)$ and $y = p(\beta)$ it therefore suffices to consider the function
$$f(x,y)= - \ep_0 x^2 + y^2 + C_0 (x-y)^2; \,\,\,\, (x,y) \in \R^2.$$ \

An application of max/min shows that $f(x,y) \geq 0$  provided $\ep_0 (C_0^2)>0$ is chosen sufficienty small and consequently, it follows that for small $\ep_0>0,$
$$ \Im \Phi(\alpha,\beta) \geq 0.$$
Then, from \eqref{comp} and an application of the Schur lemma, it follows that for $\ep_0>0$ sufficiently small,
$$ \| e^{\ep_0 \psi/h} T_{geo}(h) S_{hol}(h) e^{-\psi/h} \|_{L^2 \to L^2} = O(1)$$ 
and so, 
$$  \| \chi_{in} e^{ \ep_0 \psi/h} T_{geo}(h) \phi_h \|_{L^2} = O(1)  \| \tilde{\chi}_{in} e^{ \psi / h} T_{hol}(h) \phi_h \|_{L^2} + O(e^{-C/h}) = O(1).$$ 
Consequently, the $h$-microlocal bound \eqref{geobound} follows.

After possibly shrinking $\ep_0>0$ further, we know that  by the long range bound in \eqref{lr2}, 
$$ \| e^{\ep_0 \psi/h} (1-\chi_{in}) T_{geo}(h) \phi_h \|_{L^2 \to L^2} = O(1),$$ 
and so, Theorem \ref{mainthm2} follows.
\end{proof}
\bigskip

\subsection{Refinement of the weight function: Proof of Theorem \ref{mainthm3}} \label{refine}
%To bound $C_{\Omega,g}$ in \eqref{secondterm}, we need to have the more precise estimates on $T_{hol}$ in Theorem \ref{mainthm3}. 

\begin{proof}
We first turn to the proof of \eqref{holbound1}. Since the results of Theorem \ref{mainthm3} are $h$-microlocal and, moreover, away from a neighborhood of $\{p=0\}$ we know that any FBI transform applied to $\varphi_h$ is exponentially small we work in a small neighborhood of $\{p=0\}$. In particular, this implies that for $\chi \in \Cc(\re)$ with $\chi \equiv 1$ near 0, and  for fixed arbitrarily small $\ep_0>0$, there exists $c = c(\ep_0)>0$ so that 
$$S_{hol}\chi(\ep_0^{-1}p^2/2)T_{hol}(h)\varphi_h=\varphi_h +O_{L^2}(e^{-c/h}).$$

Thus, we work exclusively  with $T_{hol}(h)$ here and start by reexamining the proof of \eqref{weight4}. The key estimate is \eqref{holomorphic estimate} where we use that $(P(h))^2\varphi_h=O_{H_h^{-k}}(e^{-c/h})$ (here $H_h^{-k}$ is the semiclassical Sobolev space or order $-k$; see for example \cite[Chapter 14]{Zw}). Notice that in order to conclude that $e^{\psi/h}T_{hol}(h)\varphi_h$ is well controlled, we must have good control of the $O(h)$ error term appearing in the right hand side of \eqref{holomorphic estimate}. In particular, we must have control of $e^{\psi/h}T_{hol}(h)\varphi_h$ strictly away from $\{p=0\}$. The long range estimates tell us that we have some exponential decay, however, we do not have any useful control over the constant. Therefore, in order to complete the arguments leading to \eqref{weight4}, we must choose $\psi$ so that 
\begin{equation}
\label{e:constraint}
\begin{gathered}
\psi = 0\text{ on }p^2> 2\ep,\qquad
p^2(\beta(\alpha,\alpha_\xi^*(2i\partial\psi)))\geq cp^2  \text{ with } c>0 \text{ when } p^2 < \ep.
\end{gathered}
\end{equation}
\begin{lem}
\label{l:p2psi}
For all $\delta_0>0$, $0<\gamma<1/2$, there exists $\psi_0\in C^\infty(M_\tau)$ so that $\|\psi_0\|_{C^1}\leq \delta_0$, $\supp \psi_0\subset \{|p|\leq \delta_0\}$,
$$\Re p^2(\beta(\alpha,\alpha_\xi^*(2i\partial\psi_0)))\geq \frac{1-2\gamma}{2}p^2,$$
and in a neighborhood of $p=0$, 
$$\psi_0=\frac{\gamma}{2|dp|_{g_s}^2}$$
where $g_s$ is the Sasaki metric on $TM$ \red{(see e.g.~\cite[Chapter 9]{BlairSasaki})}.
\end{lem}
\begin{proof}

Motivated by the construction of the weight in Theorem \ref{mainthm2}, we let $\chi\in C_c^\infty(\re)$ with $\chi \equiv 1$ on $[-1,1]$, $\supp \chi \subset [-2,2]$, $0\leq \chi\leq 1$, $\chi'(x)\leq 0$ on $x\geq 0$, and make the ansatz
\begin{equation} \label{ansatz1}
\psi_0=\chi(\delta^{-1}p^2)a p^2. \end{equation}
Note that throughout this proof, all $O(\cdot)$ statements are uniform in $\delta$. 

By \eqref{twist},
$$\beta_x(\alpha,\alpha_\xi^*)=\alpha_x-\alpha_\xi^*+O_{}(|\alpha_\xi^*|^2),\quad \beta_\xi(\alpha,\alpha_\xi^*)=\alpha_\xi+i\alpha_\xi^*+O_{}(|\alpha_\xi^*|^2).$$
We work in geodesic normal coordinates centered at $\alpha_x$ (i.e. $\alpha_z=\exp_{\alpha_x}(-i\alpha_\xi)$) so that $\partial\psi_0=\frac{1}{2}(\partial_{\alpha_x}\psi_0+i\partial_{\alpha_\xi}\psi_0)(d\alpha_x-id\alpha_\xi)$. Therefore, 
$$\alpha_\xi^*(2i\partial\psi_0)=\partial_{\alpha_x}\psi_0+i\partial_{\alpha_\xi}\psi_0.$$
Now, 
$$\partial \psi_0 =2 p  \, \partial p \, a(\chi(\delta^{-1}p^2)+\delta^{-1}p^2 \chi'(\delta^{-1}p^2)) + p^2 \chi(\delta^{-1}p^2) \, \partial a.$$
In particular, then 
\begin{equation*}
\begin{aligned} p^2(\beta(\alpha,\alpha_\xi^*(2i\partial\psi_0)))&=p^2 - 2p(\partial_{\alpha_x} p-i\partial_{\alpha_\xi} p)(\partial_{\alpha_x}\psi_0 +i\partial_{\alpha_\xi} \psi_0)+O(|\partial\psi_0|^2)\\
&=p^2 \cdot \Big( 1-4|dp|_{g_s}^2a(\chi(\delta^{-1}p^2)+\delta^{-1}p^2\chi'(\delta^{-1}p^2))+O(p^{-1}|\partial\psi_0|^2) \Big)
\end{aligned}
\end{equation*}
Therefore, choosing 
$$a=\frac{\gamma}{2|dp|_{g_s}^2},$$
proves the lemma since $\chi'(x)\leq 0$ on $x\geq 0$ implies
$$1-4|dp|_{g_s}^2a(\chi(\delta^{-1}p^2)+\delta^{-1}p^2\chi'(\delta^{-1}p^2))\geq 1-2\gamma(\chi(\delta^{-1}p^2)+\delta^{-1}p^2\chi' (\delta^{-1}p^2))\geq 1-2\gamma$$
and 
$$p^{-1}|\partial \psi_0|^2=p^{-1}(O(\delta^{-2}p^6+p^2))=O(p).$$
\end{proof}
\noindent Using $\psi_0$ from Lemma \ref{l:p2psi} in the analysis leading to \eqref{weight4} proves \eqref{holbound1}.

Next, we prove the bound in Theorem \ref{mainthm3} \eqref{holbound2}.
The fact that $(P(h))^2\varphi_h=O_{H_h^{-k}}(e^{-c/h})$ is a weaker condition than $P(h)\varphi_h=O_{L^2}(e^{-c/h})$ and indeed, the example of $e^{ix/h}$ on $S^1$ shows that the weight $\psi_0$ is not quite optimal. To remedy this, and obtain \eqref{hol bound} we need to work directly with $P(h)\varphi_h=O_{L^2}(e^{-c/h})$.  Unlike $p^2$, $p$ does not have a fixed sign, so we need to work separately on $p>0$ and $p<0$. Therefore, we construct slightly different weights on $p>0$ and $p<0$. Since the weighted estimate naturally localizes to each region, we consider first the case $p>0$ and then easily adapt the argument to the case $p<0.$

\begin{lem}
\label{l:weight}
For all $\delta_0>0$, there exists $\psi_+\in C^\infty(M_\tau)$ with $\|\psi_+\|_{C^1}<\delta_0$, $\supp \psi_+\subset \{|p|\leq \delta_0\}$ so that 
$$\Re p(\beta(\alpha,\alpha_{\xi}^*(2i\partial \psi_+)))\geq p^2\quad \text{ on }p\geq 0$$
and in a neighborhood of $p=0$, 
$$\psi_+=\frac{p^2}{2|dp|_{g_s}^2}+O(p^3)$$
where $g_s$ is the Sasaki metric on $TM$.
\end{lem}
\begin{proof}
First, let $\chi\in \Cc(\re)$ with $\chi \equiv 1$ on $[-1,1]$, $\supp \chi\subset [-2,2]$, and $\chi'(x)\leq 0$ on $x\geq 0$. Then fix $\delta>0$ to be chosen small enough later and let
$$\psi_+=\chi(\delta^{-1}p)(a_0p^2+a_1p^3),\quad a_i\in C^\infty(M_\tau).$$
Note that throughout this proof, all $O(\cdot)$ statements are uniform in $\delta$.

Computing as in the proof of Lemma \ref{l:p2psi}, 
\begin{gather*}
\beta_x(\alpha,\alpha_\xi^*(2i\partial\psi))=\alpha_x-(\partial_{\alpha_x}\psi_++i\partial_{\alpha_\xi}\psi_0)+O(|d\psi|^2),\\
\beta_\xi(\alpha,\alpha_\xi^*(2i\partial\psi))=\alpha_\xi+i(\partial_{\alpha_x}\psi_++i\partial_{\alpha_\xi}\psi_0)+O(|d\psi|^2).
\end{gather*}
For convenience, let 
$$\partial_{\bar{\alpha}_z}=\partial_{\alpha_x}+i\partial_{\alpha_\xi},\quad \partial_{\alpha_z}=\partial_{\alpha_x}-i\partial_{\alpha_\xi}.$$
Then, 
$$\partial_{\bar{\alpha}_z}\psi_+=\chi(\delta^{-1}p)(2a_0p\partial_{\bar{\alpha}_z}p+a_13p^2\partial_{\bar{\alpha}_z}p+p^2\partial_{\bar{\alpha}_z}a_0+O(p^3))+\delta^{-1}\chi'(\delta^{-1}p)(\partial _{\bar{\alpha}_z}p)(a_0p^2+O(p^3))$$
Then, Taylor expansion of $p\circ\beta(\alpha,\alpha_\xi^*(2i\partial\psi_+))$ around $\alpha$ gives 
\begin{align*}
p(\beta(\alpha,\alpha_\xi^*(2i\partial\psi_+)))&=p\left[\begin{aligned}&1-2|dp|_g^2\big(\delta^{-1}\chi'(\delta^{-1}p)a_0p+\chi(\delta^{-1}p)(a_0+3pa_1)\big)\\
&\qquad\qquad\qquad -\lan \partial_{\alpha_z}p,\chi(\delta^{-1}p)p\partial_{\bar{\alpha}_z}a_0\ran+\tilde{E}\end{aligned}\right]
\end{align*}
where 
$$|\tilde{E}|\leq \delta^{-2}p^3|\chi'(\delta^{-1}p)|^2+|p|\chi^2(\delta^{-1}p)+p^2|\chi(\delta^{-1}p)|+\delta^{-1}p^2|\chi'(\delta^{-1}p)| =O(p).$$
Grouping terms by homogeneity in $p$, we have then 
\begin{align*}
&p(\beta(\alpha,\alpha_\xi^*(2i\partial\psi_+)))\\
&=p\left[\begin{aligned}&1-2|dp|_{g_s}^2\chi(\delta^{-1}p)a_0 \\
&\qquad\qquad\qquad-p\big(2|dp|_{g_s}^2[\chi(\delta^{-1}p)a_1+\delta^{-1}\chi'(\delta^{-1}p)a_0]-\lan \partial_{\alpha_z}p,\chi(\delta^{-1}p)\partial_{\bar{\alpha}_z}a_0\ran\big)+\tilde{E}\end{aligned}\right]
\end{align*}

Now, let ${\red{N}}>0$ to be chosen large enough independently of $\delta$ later
$$a_0:=\frac{1}{2|dp|_{g_s}^2},\qquad a_1=-{\red{N}}.$$
Then, notice that on $p\geq 0$, $\chi'(\delta^{-1}p)\leq 0$, so, since $\chi\geq 0$,
\begin{align*}
&\Re p(\beta(\alpha,\alpha_\xi^*(2i\partial\psi_+)))\\
&\quad\geq p\left[1-\chi(\delta^{-1}p)+p\big(2|dp|_g^2\chi(\delta^{-1}p)\red{N}-\lan \partial_{\alpha_z}p,\chi(\delta^{-1}p)\partial_{\bar{\alpha}_z}a_0\ran\big)+O(p)\right]
\end{align*}
Therefore, on $\{\chi(\delta^{-1}p)\geq \frac{1}{2}\}\cap\{p\geq 0\}$, 
\begin{align*}
&\Re p(\beta(\alpha,\alpha_\xi^*(2i\partial\psi_+)))\geq p\left[p\big(|dp|_g^2\red{N}-\lan \partial_{\alpha_z}p,\chi(\delta^{-1}p)\partial_{\bar{\alpha}_z}a_0\ran\big)+O(p)\right]\geq \frac{{\red{N}}}{2}p^2
\end{align*}
for ${\red{N}}$ large enough.
On the other hand, on $\supp \chi(\delta^{-1}p)\cap \{\chi(\delta^{-1}p)\leq \frac{1}{2}\}\cap\{p\geq0\}$, 
\begin{align*}
\Re p(\beta(\alpha,\alpha_\xi^*(2i\partial\psi_+)))
\geq p\left[\frac{1}{2}+O(p)\right]\geq \frac{1}{4}p
\end{align*}
if $\delta>0$ is chosen small enough.

This completes the proof of the lemma since 
$$|\psi_+|\leq C\delta^2,\qquad |\partial\psi_+|\leq C\delta.$$
\end{proof}

\begin{rem}
\label{r1}
It is possible to construct $\psi_+$ satisfying the following:
For every ${\red{N}}>0$, ${\red{M_0}}>0$, and $\delta_0>0$, there exists $\psi_+\in C^\infty(M_\tau)$ with $\|\psi_+\|_{C^1}\leq \delta_0$, and $\supp\psi_+\subset \{|p|\leq \delta_0\}$ so that 
$$\Re p(\beta(\alpha,\alpha_\xi^*(2i\partial \psi_+)))\geq \frac{1}{2}p^{\red{N}} +h{\red{M_0}},\quad \text{ on }p\geq 0,$$
and in a neighborhood of $p=0$, 
$$\psi_+=\sum_{j=0}^{{\red{N}}-1}a_jp^2-Ch{\red{M_0}}p,\quad a_0=\frac{1}{2|dp|_{g_s}^2}.$$
Moreover, for any $j<{\red{N}}-1$ and $\delta>0$, there exist $b\in C^\infty(M_\tau)$ with $\|b\|_{C^1}<\delta$ so that if $a_j$ is replaced by $a_j+b,$ then
$$\inf_{0\leq p<\delta_0}\Re p(\beta(\alpha,\alpha_\xi^*(2i\partial \psi_+)))<0.$$
This allows one to improve the higher order terms in $\psi_{hol}$ in Theorem \ref{mainthm3} so that they are sharp for $e^{ix/h}$ modulo $p^{\red{N}}$ for any ${\red{N}}$.
\end{rem}

\begin{figure}
\begin{center}
\includegraphics{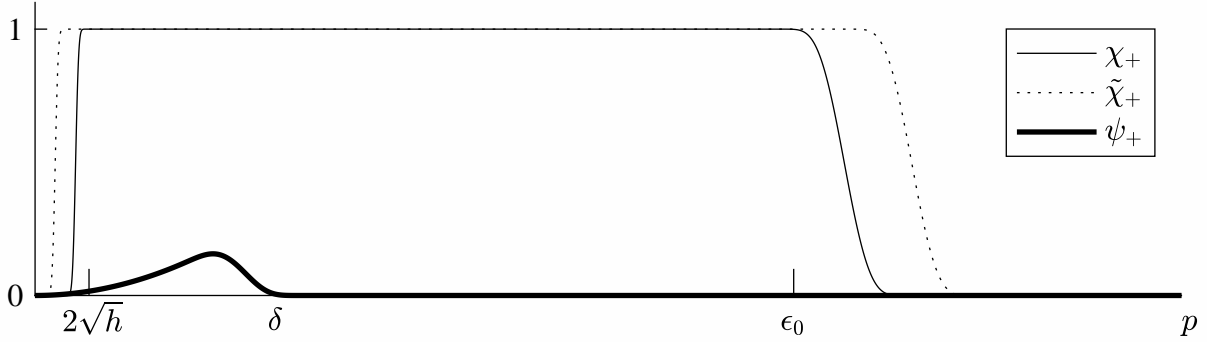}
\end{center}
\caption{We show the various cutoffs used in the proof of Theorem \ref{mainthm3}. }
\end{figure}

Since  we want to localize to $p>0$, we insert a smooth cutoff $\chi_+$ that approximates the indicator function ${\bf 1}_{[0, \e_0]}(p)$. However, in order to estimate error terms in the weighted $L^2$ bounds corresponding to Proposition \ref{weightedl2}, we must ensure that $\psi_+ = O(h)$ on supp $\partial \chi_+$ so that, in particular,
$$ \| e^{\psi_+/h} T_{hol}(h) \phi_h \|_{L^2( \text{supp }  \partial \chi_+ )} = O(1),$$ where  $\supp \partial \chi_{+} \subset \{0 \leq p \leq h^{1/2} \}\cup \{p \geq \ep_0 \}.$

To construct $\chi_+$, we let $\chi_1\in C^\infty(\re)$ with $\supp \chi_1\subset (1,\infty)$ with $ \chi_1\equiv 1$ on $(2,\infty)$ and let 
$$\chi_{+}:=\chi(\ep_0^{-1}p)\chi_1( h^{-1/2}p),$$  
and 
$\tilde{\chi}_{+}=\tilde{\chi}(\ep_0^{-1}p^2/2)\tilde{\chi}_1(h^{-1/2}p)$ where $\tilde{\chi}\in \Cc(\re)$ has $\tilde{\chi}\equiv 1$ on $\supp \chi$ and $\tilde{\chi}_1\in C^\infty(\re)$ with $\tilde{\chi}_1\equiv 1$ on $\supp \chi_1$ and $\supp \tilde{\chi}_1\subset (1,\infty)$. Then $\supp \chi_+\subset \{0\leq p\leq 2\e_0\}$ so that $\psi_+=O(h)$  on $\supp \partial \chi_+$.

We will need the following 
\begin{lem} \label{refinedestimate}
Under the same assumptions as in Proposition \ref{weightedl2},
\begin{multline} \label{holomorphic estimate-}
 \langle \chi_{+} e^{\psi_{+}/h} T_{hol}(h) P(h) \phi_h, e^{\psi_{\pm}/h} T_{hol}(h) \phi_h \rangle _{L^2}  = \langle \chi_{+}e^{\psi_{\pm}/h}  (p\circ \beta)|_{\Lambda} T_{hol}(h)  \phi_h, e^{\psi_{+}/h} T_{hol}(h) \phi_h \rangle_{L^2}  \\
 +  O(h)  \| \tilde{\chi}_{+} e^{\psi_{+}/h} T_{hol}(h)  \phi_h \|_{L^2}^2+O(h^{1/2})\| \tilde{\chi}_{+} e^{\psi_{+}/h} T_{hol}(h)  \phi_h \|_{L^2(\{|p|\leq Ch^{1/2}\})}^2 + O(e^{-C/h}) \end{multline}
 \end{lem}
 
 \begin{proof} We follow  very closely the argument in the proof of Proposition \ref{weightedl2}. The only difference here occurs in the integration by parts arguments  with respect to $\partial_{\alpha_z}$ and $\dbar_{\alpha_z}$ in \eqref{invariance} and \eqref{invariance2}, precisely when the cutoff $\chi_+$ is differentiated in the amplitude since it is a singular semiclassical symbol with $\chi_+ \in S^{0}_{1/2}(1)$ \cite[Chapter 4]{Zw} \red{ i.e. 
 $$|\partial_{\alpha_z}^\alpha\chi_+|\leq C_{\alpha \beta}h^{-|\alpha|/2}.$$} Specifically, one needs to  estimate a  term of the form
 \begin{align} \label{difference}
&  \lan h \partial_{\alpha_z} ( \chi_{\pm} ) \, r(\alpha,hD_{\alpha} ) [ e^{\psi/h}T_{hol}(h)\varphi_h],  e^{\psi/h}T_{hol}(h) \varphi \ran  \nonumber\\
  &\quad=  \lan h \chi \partial_{\alpha_z} ( \chi_{1} ) \, r(\alpha,hD_{\alpha} ) [ e^{\psi/h}T_{hol}(h)\varphi_h],  e^{\psi/h}T_{hol}(h) \varphi \ran \nonumber \\
  &\qquad+  \lan h \chi_1 \partial_{\alpha_z} ( \chi ) \, r(\alpha,hD_{\alpha} ) [ e^{\psi/h}T_{hol}(h)\varphi_h],  e^{\psi/h}T_{hol}(h) \varphi \ran \quad r \in S^{0}(1).
  \end{align}

  Since $\chi_+ \in S^{0}_{1/2}(1),$ it is in a singular symbol class. However, because it is a multiplier depending only the spatial coordinates $\alpha \in M_{\tau}^{\C}$, it can be effecitvely composed with the standard $h$-pseudodifferential operators $r(\alpha, hD_{\alpha}) \in Op_h (S^{0}(1)).$ In particular, symbols compose with $h^{-j/2}$-loss in the $j^{\text{th}}$ term of the asymptotic expansion, $L^2$-boundedness and sharp G\aa rding still hold, as does $h$-pseudolocality.  More precisely, for any $R(h) \in Op_h (S^{0}(1))$ and spatial cutoff $\chi_+ = \chi_+(\alpha,h) \in S^{0}_{1/2}(1),$ 
  $$ \chi_+ R(h) = \chi_+ R(h) \tilde{\chi}_+  + O(h^{\infty})_{L^2 \to L^2}.$$
   Since the cutoff 
  $ \chi_1 \partial_{\alpha_z}  \chi \in S^{0}_{1/2}(1),$  and depends only on the spatial $\alpha$-variables, it follows that for second term on the RHS of \eqref{difference},
$$  \lan h \chi_1 \partial_{\alpha_z} ( \chi) \, r(\alpha,hD_{\alpha} ) [ e^{\psi/h}T_{hol}(h)\varphi_h],  e^{\psi/h}T_{hol}(h) \varphi \ran  = O(h) \| \tilde{\chi}_{+} e^{\psi_{+}/h} T_{hol}(h)  \phi_h \|_{L^2}.$$

As for the first term on the RHS of \eqref{difference},
$h \chi \partial_{\alpha_z} (\chi_1) \in h^{-\frac{1}{2}} S_{1/2}^{0}(1)$ since there is a loss of $h^{1/2}$ coming from differentiation of the $\chi_1$-term. Thus,
 \begin{multline*} 
  \lan h \chi \partial_{\alpha_z} ( \chi_{1} ) \, r(\alpha,hD_{\alpha} ) [ e^{\psi/h}T_{hol}(h)\varphi_h],  e^{\psi/h}T_{hol}(h) \varphi \ran = \\
  O(h^{1/2})\| \tilde{\chi}_{+} e^{\psi_{+}/h} T_{hol}(h)  \phi_h \|_{L^2(\{|p|\leq Ch^{1/2}\})}^2. \end{multline*}
  %t
\end{proof}

% \begin{align} \label{pseudolocal}
%  \lan h\partial\chi_{\pm} r(\alpha,hD_{\alpha} e^{\psi/h}T_{hol}(h)\varphi_h,e^{\psi/h}T_{hol}(h) \varphi \ran= O(h)  \| \tilde{\chi}_{\pm} e^{\psi_{\pm}/h} T_{hol}(h)  \phi_h \|_{L^2}^2+O(h^{1/2})\| \tilde{\chi}_{\pm} e^{\psi_{\pm}/h} T_{hol}(h)  \phi_h \|_{L^2(\{|p|\leq Ch^{1/2}\})}^2.
%\end{align}

Since
$ P(h) \phi_h = O_{L^2}(e^{-c/h}),$ it follows from Lemma \ref{refine} that for $\ep$ small enough and ${\red{N}}>0,$ 
\begin{multline} \label{weight2-}
 \langle \chi_{+} e^{\psi_{+}/h}  (p\circ \beta)|_{\Lambda} T_{hol}(h)  \phi_h, e^{\psi_{+}/h} T_{hol}(h) \phi_h \rangle_{L^2} =  O(h)  \| \tilde{\chi}_{+} e^{\psi_{+}/h} T_{hol}(h) \phi_h \|_{L^2}^2\\
+O(h^{1/2})\| \tilde{\chi}_{+} e^{\psi_{+}/h} T_{hol}(h)  \phi_h \|_{L^2(0 \leq p \leq {\red{N}}h^{1/2})}^2+O(e^{-c/h}). \end{multline}

In analogy with the arguments used to prove \eqref{weight4}, we substitute the lower bound \newline $(p\circ \beta)|_{\Lambda} \geq p^2$ on $p \geq 0$ from Lemma \ref{l:weight} in (\ref{weight2-}). Then, since $\psi_+ = O(h)$ when $p^2 = O(h)$ (so that $e^{\psi_+/h} = O(1)$ on the latter set), it follows that
 \begin{multline} \label{weight3-}
 \langle \chi_{+} e^{\psi_{+}/h} (c{\red{N}}h+O_{}(h))T_{hol}(h)  \phi_h, e^{\psi_{+}/h} T_{hol}(h) \phi_h \rangle_{L^2(p \geq {\red{N}}h^{1/2})}   \\
 =O(h^{1/2})  \| \tilde{\chi}_{+} T_{hol}(h)\phi_h \|_{L^2 (0 \leq p\leq {\red{N}}h^{1/2})}^2 + O(h)  \| (\tilde{\chi}_{+} - \chi_+)  T_{hol}(h)\phi_h \|_{L^2 (p \geq  {\red{N}}h^{1/2})}^2 + O(e^{-c/h}) \\
 = O(h^{1/2}). \end{multline}

\noindent In \eqref{weight3-} we have also used that $e^{\psi_+(\alpha)/h} = O(1)$ for all $\alpha \in \text{supp} (\tilde{\chi}_+  - \chi_{+}) \cap \{ p > C h^{1/2} \}$ since by construction $\psi_+ = O(h)$ on the latter set (see also Figure 1).

Choosing ${\red{N}}>0$ large enough to absorb the $O(h)$ term on the LHS of \eqref{weight3-} it follows that
\begin{equation} \label{weight4-}
\| \chi_{+} e^{\psi_{+}/h} T_{hol}(h) \phi_h \|_{L^2 }  = O(h^{-1/4}). \end{equation}

The analysis on the set $p<0$ follows in the same way as above, except one uses the reflected cutoff function $\chi_{-}(p) = \chi_+(-p), $ 
and the ansatz for the corresponding weight function is
 \begin{equation} \label{ansatz3}
\psi_{-}= \chi(-\delta^{-1}p)  p^2(a_0+a_1p), \quad  p < 0, \, a_j \in C^{\infty}(M_\tau); j=0,1.\end{equation}
so that 
$$-\Re p(\beta(\alpha,\alpha_{\xi}^*(2i\partial \psi_-)))\geq p^2\qquad p\leq 0.$$
That finishes the proof of Theorem \ref{mainthm3} \eqref{holbound2}.
 \end{proof}

\section{Exponential decay of the harmonic extensions of boundary eigenfunctions: Proof of Theorem \ref{mainthm}} \label{steklovbound}

\begin{proof} 
 \red{Recall that $M:=\partial\Omega$ with $\Omega$ the domain of the Steklov problem.} Given the FBI transform $T(h): C^{\infty}(M) \to C^{\infty}(T^*M)$ we construct a left-parametrix $S(h): C^{\infty}(T^*M) \to C^{\infty}(M)$ and then given $q\in S^{m}(1)$ we define the anti-Wick $h$-pseudodifferential quantization by 
$$q^{aw}_{h}:= S(h) q T(h).$$

Let ${\mathcal P}: C^{\infty}(M) \to C^{\infty}(\Omega)$ be the Poisson operator and $\chi \in C^{\infty}_{0}(T^*M; [0,1])$ be a cutoff supported near the zero section, with $\chi(\red{x',\xi'}) =1$ for  $\{ 0 \leq |\red{\xi'}| \leq \ep \}$ and $\chi(\red{x',\xi'}) =0$ for $|\red{\xi'}| > 2 \ep.$ Here, $\ep>0$ is some arbitrarily small but fixed constant. 
Then, one can write
\begin{align} \label{poisson}
u_h(x) &= {\mathcal P} \phi_h (x_{\red{n+1}},\red{x'}) = {\mathcal P} (1 - \chi_h^{aw}) \phi_h (x) + {\mathcal P} \chi_{h}^{aw} \phi_h(x). \end{align}
Here, $ x=(x_{\red{n+1}},x')$ denote Fermi coordinates in a neighbourhood of the boundary  with $\Omega = \{ x_{\red{n+1}} \geq 0 \}$ and $M = \{ x_{\red{n+1}} = 0 \}.$ We show in section \ref{s:zeroSteklov} that for $\chi_{h}^{aw} \phi_h$ (the piece supported near the zero section), one has the apriori bound
\begin{equation} \label{apriori}
 \| \chi_{h}^{aw} \phi_h \|_{L^2} = O(e^{-C/h})
 \end{equation}
 for some $ C>0.$ 
 
 It follows from \cite{SU} that  ${\mathcal P} (1 - \chi_h^{aw}) \phi_h (x)$  (the piece $h$-microlocally supported away from the zero section) has Schwartz kernel of the form
 $$ K(x,y',h) = (2\pi h)^{-{\red{n}}} \int_{\R^{{\red{n}}}} e^{ i(\red{\Psi}(x_{\red{n+1}},x',\red{\xi'})- \langle \red{y', \xi '}\rangle) /h}  \, c(x,\red{y'},\red{\xi'},h) \, \chi(x'-y') \, d\red{\xi'} + O(e^{-C/h}),$$
with $c \in S^{0}(1)$ supported away from $|\red{\xi'}|=0$ and in $S^{0,0}_{cla}(|\red{\xi'}|>2\e)$, $\red{\Psi}$ solving
\begin{equation}
\label{eikonal}
(\partial_{x_{\red{n+1}}}\red{\Psi})^2+r(x,\partial_{x'}\red{\Psi})=0,\quad \red{\Psi}(0,x',\xi')=\lan x',\xi'\ran,\quad \Im \red{\Psi}\geq 0
\end{equation}
with $r(a,x',\xi')$ the symbol of the Laplacian induced on $\{x_{\red{n+1}}=a\}.$ In particular, 
$$r(x_{\red{n+1}},x',\xi')=|\xi'|_{x'}^2+2x_{\red{n+1}}Q(x',\xi')+O(x_{\red{n+1}}^2)$$
where $Q(x',\xi')$ is the symbol of the second fundamental form.
It follows by Taylor expansion in $x_{\red{n+1}}$ that 
$$\red{\Psi}=i\red{\Psi}_1(x_{\red{n+1}},x',\xi')+\lan x',\xi'\ran$$
where 
$$\red{\Psi}_1(x_{\red{n+1}},x',\xi')=x_{\red{n+1}}|\xi'|_{x'}+\frac{x_{\red{n+1}}^2(Q(x',\xi')+i\lan \partial_{x'}|\xi'|_{x'},\xi'\ran_{x'})}{2|\xi'|_{x'}}+O(x_{\red{n+1}}^3|\xi'|_{x'})$$
is homogeneous of degree 1 in $\xi'$. 
 The near-diagonal cutoff $\chi(x'-y')$ appears above since $K(x,y',h)$ is the part of the Poisson operator arising $h$-microlocally from the complement of the zero section.

Consequently, from \eqref{poisson} and \eqref{apriori} we have
\begin{align} \label{poisson2}
u_{h}(x) &= (2\pi h)^{-{\red{n}}}  \int_{\R^{\red{n}}} \int_{M} e^{ i(\red{\Psi}(x_{\red{n+1}},x',\red{\xi'})- \langle y', \red{\xi'}\rangle) /h} \, c(x,y',\red{\xi'},h) \, \chi(x'-y') \, \phi_h(y') \, dy' d\red{\xi'}\nonumber \\
& + O(e^{-C/h}). \end{align}

Next we make an analytic resolution of the identity and write
$$ Id = S(h) T(h) + O(e^{-C/h})$$
where $T(h): C^{\infty}(M) \to C^{\infty}(T^*M)$ is an FBI transform \red{with phase function $\varphi$ (not to be confused with the Steklov eigenfunction $\varphi_h$)} and $S(h): C^{\infty}(T^*M \to C^{\infty}(M)$ is a left-parametrix. Notice that by Theorem \ref{mainthm2} or Corollary \ref{thm2} together with the analysis in section \ref{s:zeroSteklov} for $\chi_1\in C_c(\re)$ with $\chi_1 \equiv 1$ on $[-1,1]$, and $\supp \chi _1\subset [-2,2]$, for any $\e>0$,
$$(1-\chi_1(\e^{-1}(|\alpha_{\xi'}|_{\alpha_{x'}}-1))T(h)=O_{L^2}(e^{-c/h}).$$

Let $\chi_{1,\e}(\alpha)=\chi_1(\e^{-1}(|\alpha_{\xi'}|_{\alpha_{x'}}-1)).$ Then,
$$\varphi_h=S(h)\chi_{1,\e}(\alpha)T(h)\varphi_h+O_{L^2}(e^{-c/h})$$
and substitution in the integral formula for $u_h$ gives
\begin{equation} \label{lift}
u_h(x) = (2\pi h)^{-{\red{n}}}  \int_{M} K_{h,x}(y')  \, S(h)\chi_{1,\e}(\alpha) T(h) \phi_h(y') dy' + O(e^{-C/h})
\end{equation} 
with
$$K_{h,x}(y') = \int_{\R^{{\red{n}}}}   e^{ i(\red{\Psi}(x_{\red{n+1}},x',\red{\xi'})- \langle y', \red{\xi '}\rangle) /h}  c(x',y',\red{\xi'},h)  \chi(x'-y') \, d\red{\xi'}.$$
One can then rewrite the formula \eqref{lift} in the form
\begin{align} \label{lift2}
u_h(x) &= (2\pi h)^{-{\red{n}}} \langle \chi_{1,\e}(\alpha)S(h)^t K_{h,x}, \overline{T(h) \phi_h} \rangle_{L^2(T^*M)} + O(e^{-C/h}).
\end{align}

Thus, $\chi_{1,\e}(\alpha)S(h)^t K_{h,x}(\alpha)$ equals
\begin{multline*} 
 (2\pi h)^{-{\red{n}}} \int_{M} \int_{\R^{{\red{n}}}} e^{i  [ \, -\overline{\phi(\alpha,y')}+ \red{\Psi}(x_{\red{n+1}},x',\red{\xi'})- \langle y', \red{\xi'} \rangle \, ] /h} \\
 c(x',y',\red{\xi'},h) \chi(\alpha_x -y')  a(y',\alpha_{x'},h) \chi(x'-y') \chi_{1,\e}(\alpha)\, dy' d\red{\xi'}
 \end{multline*}
with $\alpha = (\alpha_{x'}, \alpha_{\xi'}) \in T^*M$ and $c \in S^{0,0}_{cla}(|\red{\xi'}|>2\e), \, a \in S^{3{\red{n}}/4, {\red{n}}/4}_{cla}.$ \
Next, we apply analytic stationary phase in the $(y',\red{\xi'})$ variables. We can do for $x_{\red{n+1}}$ small since $\psi(x_{\red{n+1}},x',\eta')=\lan x',\eta'\ran +O_{C^\infty}(x_{\red{n+1}})$. 
Writing
$$ \Phi(x,\red{\xi'};\alpha,y'):= - \overline{\phi(\alpha,y')} + \red{\Psi}(x,\red{\xi'})- \langle y', \red{\xi'}\rangle $$ 
for the total phase and computing in geodesic normal coordinates centered at $\alpha_{x'},$ the critical point equations are
\begin{gather*}
 \partial_{\red{\xi'}} \Phi = x'-y' +i\partial_{\red{\xi'}}\red{\Psi}_1(x',\red{\xi'})=0,\\
\partial_{y'}\Phi=-\partial_{y'} \overline{\phi(\alpha,y')} - \red{\xi'}  = \alpha_{\xi'} - \red{\xi'} + O(|\alpha_{x'}-y'|)  =  0.
 \end{gather*}
It follows that, denoting the complex critical points by $y'_c$, $\eta'_c$, we have, using the fact that $\psi_1$ is homogeneous of degree 1
$$\lan x'-y_c',\red{\xi}_c'\ran=-i\lan \partial_{\red{\xi}'} \red{\Psi}_1 (x',\red{\xi}'_c),\red{\xi}'_c\ran=-i\red{\Psi}_1(x,\red{\xi}'_c).$$
Therefore, 
\begin{gather*} 
\Phi(x,\red{\xi}'_c,\alpha,y'_c)=-\overline{\varphi(\alpha,y'_c)},\\
 y'_{c} = x'+i\partial_{\red{\xi}'}\psi_1|_{\red{\xi}'=\red{\xi}'_c},\quad \red{\xi}_c'=\alpha_{\xi'}+O(|\alpha_{x'}-y_c'|).
 \end{gather*}
 Here, we implicitly analytically continue $\overline{\varphi(\alpha,y')}$. In particular, notice that since on the support of the integrand, $\left||\alpha_{\xi}|_g-1\right|\ll 1$, and $\red{\xi}_c'=\alpha_{\xi'}+O(|\alpha_{x'}-y_c'|)$, we have that $||\red{\xi}_c'|_g-1|\ll 1$ and hence the fact that the amplitude is not analytic in $\red{\xi}'$ near $\red{\xi}'=0$, does not cause issues in the analytic stationary phase argument.

Consequently, the $(2 \pi h)^{-{\red{n}}}$ factor in front of \eqref{lift} gets killed upon application of analytic stationary phase in $(y',\eta')$  and we get the bound
\begin{align} \label{sp 1}
|\chi_{1,\e}(\alpha)S(h)^t K_{x,h}(\alpha)| &\lessapprox   h^{-3{\red{n}}/4}e^{\Im \overline{\phi(\alpha,y_c'(x,\alpha))}/h} \chi_{1,\e}(\alpha)\chi(\alpha_{x'}-y_c') .
\end{align}
Substitution of \eqref{sp 1} together with the weighted $L^2$ estimate $\| e^{\psi/h} T(h) \phi_h \|_{L^2} = O(1)$ with $\psi = \gamma p^2, \, \gamma < \frac{1}{2}$ (see Theorem \ref{mainthm3} \eqref{holbound1} )   in \eqref{lift2} gives, by Cauchy-Schwarz, 
\begin{align} \label{preupshot}
| u_h(x) | &\leq  | \, \langle e^{-\psi/h} \chi_{1,\e}(\alpha)S(h)^t K_{h,x},  \overline{e^{\psi/h} T(h) \phi_h} \rangle_{L^2(T^*M)} \, | + O(e^{-C/h}) \nonumber \\ 
& \leq h^{-3{\red{n}}/4} \| e^{-\psi(\alpha)/h} e^{\Im\overline{\varphi(\alpha,y_c'(x,\alpha))}/h} \|_{L^2(T^*M,d\alpha)} + O(e^{-C/h}).
\end{align}

 Now, recall that we may work with $T_{hol}$ and $S_{hol}$ since for some $\e>0$,
 $$S_{hol}\chi(\e^{-1}p^2)T_{hol}\varphi_h=\varphi_h+O_{L^2}(e^{-c/h}).$$
In this case the analytic continuation, of $-\overline{\varphi}$ is given by
 $$i(\rho(\bar{\alpha})+r^2_{\mathbb{C}}(\bar{\alpha},y))/2.$$
 Computing in normal geodesic coordinates centered at $\alpha_x$, observe that 
 $$r^2(\alpha_x,y)=\lan\alpha_x-y,\alpha_x-y\ran$$
 Therefore, 
 \begin{align*} 
2i\overline{\varphi(\alpha,y_c'(x,\alpha))}&=\lan \alpha_{x'}+i\alpha_{\xi'}-y_c',\alpha_{x'}+i\alpha_{\xi'}-y_c'\ran+|\alpha_{\xi'}|^2\\
&=| \alpha_{x'}-\Re y_c'|^2-|\Im y_c'|^2+2i\lan \alpha_{x'}-\Re y_c',\Im y_c'\ran +2i\lan \alpha_{x'}-y_c',\alpha_{\xi'}\ran
 \end{align*} 
Now, by Taylor expansion in $\red{\xi'}$,
\begin{align*}\lan \partial_{\red{\xi'}}\red{\Psi}_1|_{\red{\xi}'=\red{\xi}_c'},\alpha_{\xi'}\ran &=\red{\Psi}_1(x,\alpha_{\xi'})+x_{\red{n+1}}O(|\red{\xi}_c'-\alpha_{\xi'}|^2)\\
&=\red{\Psi}_1(x,\alpha_{\xi'})+x_{\red{n+1}}O(|\alpha_{x'}-y_c'|^2)\\
&=\red{\Psi}_1(x,\alpha_{\xi'})+x_{\red{n+1}}O((|\alpha_{x'}-\Re y_c'|^2+|\Im y_c'|^2))
\end{align*}
so,
 \begin{align*} 
\Im \overline{\varphi(\alpha,y_c'(x,\alpha))}&=-\frac{1}{2}(| \alpha_{x'}-\Re y_c'|^2(1+O_{C^\infty}(x_{\red{n+1}}))-|\Im y_c'|^2(1+O_{C^\infty}(x_{\red{n+1}}))) -\red{\Psi}_1(x,\alpha_{\xi'})
 \end{align*} 

Written out explicitly, the last line in \eqref{preupshot} says that for $\gamma<1/2$,
\begin{align} 
|u_h(x)| & \lessapprox h^{-3{\red{n}}/4}  \Big( \int_{T^*M} e^{2 \Im\overline{\varphi(\alpha,y_c'(x,\alpha))}/h} e^{-2 \psi(\alpha) /h} \, d\alpha \Big)^{1/2}  \nonumber \\ 
& \lessapprox h^{-3{\red{n}}/4}  \Big( \int_{T^*M} e^{-[(|\alpha_{x'}-\Re y_c'|^2-|\Im y_c'|^2)(1+O_{C^\infty}(x_{\red{n+1}}))+2\Re\red{\Psi}_1(x,\alpha_{\xi'})+2 \psi(\alpha)] /h} \, d\alpha \Big)^{1/2}  \nonumber \\
& \lessapprox h^{-3{\red{n}}/4}  \Big( \int_{T^*M} e^{-[(|\alpha_{x'}-\Re y_c'|^2-|\Im y_c'|^2)(1+O_{C^\infty}(x_{\red{n+1}}))+2\Re\red{ \Psi}_1(x,\alpha_{\xi'})+2 \psi(\alpha)] /h} \, d\alpha \Big)^{1/2}  \nonumber \\
& \lessapprox h^{-3{\red{n}}/4}  \Big( \int_{S^{{\red{n-1}}}}\int_{0}^\infty\int_{M} e^{-[(|\alpha_{x'}-\Re y_c'|^2-|\Im y_c'|^2)(1+O_{C^\infty}(x_{\red{n+1}}))+2\Re\red{\Psi}_1(x,r\omega)+ \gamma(r-1)^2] /h} \, d\alpha_{x'}drd\omega \Big)^{1/2}  \nonumber 
\end{align}
Now, recall that $ \Im y_c'=O(x_{\red{n+1}})$ and $\partial_{\alpha}y_c'=O(x_{\red{n+1}})$ and hence applying the method of steepest descent, in the $\alpha_{x'},r$ variables, the critical point of the phase occurs at 
\begin{align}
 2(\alpha_{x'}-\Re y_{c}')(I-\partial_{\alpha_{x'}}\Re y_c')+(\alpha_{x'}-\Re y_c')^2O(x_{\red{n+1}})-2(\Im y_c')\partial_{\alpha_{x'}}\Im y_c'&=O(x_{\red{n+1}}^3)\label{e:xStationary}\\
2(\alpha_{x'}-\Re y_{c'}')(-\partial_r\Re y_c')+(\alpha_{x'}-\Re y_c')^2O(x_{\red{n+1}})-2(\Im y_c')\partial_r \Im y_c'\nonumber\\
\qquad\qquad\qquad\qquad+2\partial_r\Re\red{\Psi}_1(x,r\omega)+2\gamma(r-1)&=O(x_{\red{n+1}}^3)\label{e:rStationary}%
\end{align}
Using \eqref{e:xStationary}, we have that at the stationary point
$$\alpha_{x'}-\Re y_c'=(\Im y_c')\partial_{\alpha_{x'}}\Im y_c'+O(x_{\red{n+1}}^3).$$
Putting this into \eqref{e:rStationary} gives
$$2\partial_r\Re \red{\Psi}_1(x,r\omega)+2\gamma(r-1)=2(\Im y_c')\partial_r\Im y_c'(x,\alpha_x,r\omega)+O(x_{\red{n+1}}^3).$$
Now, 
\begin{gather*}
2\partial_r\Im y_c'=0,\quad 2\partial_r\Re \red{\Psi}_1(x,r\omega)=a(x,\omega)\\
a(x,\omega)=2x_{\red{n+1}}+x_{\red{n+1}}^2\partial_{\xi'}Q(x',\omega)+x_{\red{n+1}}^2Q(x',\omega)+O(x_{\red{n+1}}^3)
\end{gather*}
So, the critical point occurs at 
$$r=1-\frac{a(x,\omega)}{2\gamma}+O(x_{\red{n+1}}^3).$$
Noting that $\Im y_c'=x_{\red{n+1}}+O(x_{\red{n+1}}^2)$ and evaluating the exponential at this point yields
\begin{align*} 
|u_h(x)|&\lessapprox h^{-{\red{n}}/2+1/4}\left(\int_{S^{{\red{n-1}}}}e^{-2(x_{\red{n+1}}+x_{\red{n+1}}^2(Q(x',\omega)-\gamma^{-1}-1)/2+O(x_{\red{n+1}}^3))/h}d\omega\right)^{1/2}\\
&\lessapprox h^{-{\red{n}}/2+1/4}e^{-(x_{\red{n+1}}-x_{\red{n+1}}^2(-\inf_{\omega\in S^{{\red{n}}-1}}Q(x',\omega)+\gamma^{-1}+1)/2+O(x_{\red{n+1}}^3))/h}.
\end{align*}
The same argument works for derivatives with each differentiation creating  a power of $h^{-1}.$ \end{proof}
 \begin{rem}
We note that the first order bound with $d(x) = d_{\partial \Omega}(x) + O( d_{\partial \Omega}^2(x))$ in Theorem \ref{mainthm} \eqref{basic bound} follows from the weighted $L^2$-bound in Theorem \ref{mainthm2} by essentially the same argument as above. The proof in that case is slightly simpler since one need only keep track of the analytic stationary phase computations to $O(x_{\red{n+1}}^2)$ error (rather than $O(x_{\red{n+1}}^3)$).
\end{rem}

\section{Microlocal estimates near the zero section} \label{zerosection}

While the Dirichlet to Neumann map is a homogeneous analytic pseudodifferential operator, it is not well behaved as a semiclassical analytic pseudodifferential operator near the zero section. In particular, when semiclassically rescaled, the full symbol of a homogeneous pseudodifferential operator, $a$, typically has singularities of the form
$$|\partial_\xi^\alpha a(x,\xi/h)|\sim C_\alpha h^{-|\alpha|},\quad |\xi|\ll 1.$$

Apriori, this singular $h$ dependence may result in the transport of semiclassical analytic wavefront sets away from the zero section under the action of a (homogeneous) pseudodifferential operator. The purpose of this section is to show that no such transport occurs and then to use this information to estimate Steklov eigenfunctions near the zero-section.

 %%%%%%%%%%%%%%%%%%%%%%%%%%%%%%%%%%%%%%%%%%%%%%%%%%%%%%%%%%%%%%%%%%%%%%%%%%%%%%%%%%%%
%%%%%%%%%%%%%%%%%%%%%%%%%%%%%%%%%%%%%%%%%%%%%%%%%%%%%%%%%%%%%%%%%%%%%%%%%%%%%%%%%%%%

\subsection{Cauchy estimates and the Euclidean FBI transform}

For $0\leq \tilde{h}\leq h$, let
$$T_{\tilde{h} }u(x,\xi,h;\mu):=\int e^{\frac{i}{\tilde{h}}\lan x-y,\xi\ran-\frac{1}{2\tilde{h}}|x-y|^2}u(y)dy.$$
Let also 
$$S_{\tilde{h}}v(x):=2^{-n}(\pi\tilde  h)^{-3n/2}\int e^{\frac{i}{\tilde{h}}\lan y-x,\eta\ran -\frac{1}{2\tilde{h}}|x-y|^2}v(y,\eta)dyd\eta.$$
Then for all $u\in \mc{S}(\re^{\red{n}})$, $u=S_{\tilde{h}}T_{\tilde{h}}u$. 

Define 
$$T'_{a,\mu,\tilde{h}}u:=\int e^{\frac{i}{\tilde{h}}\lan x-y,\xi\ran-\frac{\mu}{2\tilde{h}}|x-y|^2}a(x,y,\xi)u(y)dy.$$

The next proposition is very similar to \cite[Proposition 2.1]{Ji}. The difference is that we do not require $\mu\geq 1$. Instead, we keep track of the dependence of various estimates on $\mu\geq \mu_0$. 
\begin{prop}
\label{l:changeFBI}
Suppose that $a$ has tempered growth in $(x,y,\xi)$ and $0<\mu_0\leq \mu<\mu_1$. Then
suppose that for $W$ a neighborhood of $(x_0,\xi_0)$, 
$$\sup_{W}|T_{\tilde{h}} u |\leq Ce^{-\delta/\tilde{h}}.$$
Then there is a neighborhood $V$ of \red{$(x_0,\xi_0)$}
$$\sup_{\red{V}}|T'_{a,\mu,\tilde{h}}u|\leq Ce^{-c\min( \mu_0,\mu_1^{-1},\delta)/\tilde{h}}.$$
\end{prop}
\begin{proof}[Proof Sketch]
In order to prove the lemma, one writes $T'_{a,\mu,\tilde{h}}u=T'_{a,\mu,\tilde{h}}S_{\tilde{h}}T_{\tilde{h}}u$. Then, one can easily estimate the kernel of  $T'_{a,\mu,\tilde{h}}S_{\tilde{h}}$ using a simple change of variables.
\end{proof}

The next proposition is similar to \cite[Proposition 2.2]{Ji}, except that we obtain a quantitative estimate on distances to the zero-section.

\begin{prop}
Fix $x_0\in M$, $\e>0$. \red{Let $X$ be a neighborhood of $x_0$ in $M$.} Suppose that $\|u\|_{L^\infty(\red{X})}\leq Ch^{-N}$. Then the following are equivalent:
\begin{itemize}
\item there exist $C,c>0$, \red{$h_0>0$} and an $\e$-independent neighborhood, $W$, of $x_0$ so that for every $0\leq \tilde{h}\leq h\leq h_0$,
\begin{equation}
\label{e:FBI}
|T_{\tilde{h}}u(x,\xi,h)|\leq Ce^{-c /\tilde{h}}\|u\|_{L^\infty(X)},\quad x\in W ,\,|\xi|\geq \e,
\end{equation}
\item there exists $C_1>0$, \red{$h_0>0$} and an $\e$-independent neighborhood, $W$, of $x_0$ and constant $c_1>0$ so that \red{for $0<h<h_0$}
\begin{equation}
\label{e:supNorm}\sup_{\Re x\in W,\,|\Im x|\leq c_1}|u|\leq C_1e^{\frac{\e}{2h}}\red{\|u\|_{L^\infty(X)}}
\end{equation}
\item there exists $C>0$, $\red{h_0>0}$ and an $\e$-independent neighborhood, $W$, of $x_0$ so that \red{for $0<h<h_0$}
\begin{equation}
\label{e:cauchy}|(hD)^\alpha u\red{(x)}|\leq C\|u\|_{L^\infty(X)}C^{|\alpha|}(h|\alpha|+\e)^{|\alpha|},\quad x\in W.
\end{equation}
\end{itemize}
\end{prop}
\begin{proof}[Proof Sketch]
The fact that \eqref{e:supNorm} implies \eqref{e:cauchy} follows from basic Cauchy estimates, while that \eqref{e:cauchy} implies \eqref{e:supNorm} follows from writing out the Taylor formula with remainder.

The equivalence of \eqref{e:supNorm} and \eqref{e:FBI} follows from writing down a resolution of the identity in terms of the FBI transform and deforming the contour into the complex domain. 

\end{proof}

The next estimate proves that the application of an analytic homogeneous pseudodifferential operator preserves estimates of the form \eqref{e:cauchy}. Roughly speaking, we show that the Sobolev mapping properties of such an operator behave like the Sobolev mapping properties of multiplication by an $h$-independent analytic function. Throughout the proof of the next proposition, we will use the following elementary estimates estimates without comment
\begin{gather}
\left(1+\frac{s}{t}\right)^t\leq e^s\leq \left(1+\frac{s}{t}\right)^{t+s},\quad t,s>0\\
(ne)^{-|\alpha|}|\alpha|^{|\alpha|}\leq \alpha !\leq |\alpha|^{|\alpha|},\quad \alpha \in \mathbb{N}^n
\end{gather}
\begin{prop}
\label{l:pseudo}
Let $P$ be a homogeneous analytic pseudodifferential operator of order $k$. Suppose that $u$ satisfies \eqref{e:cauchy} in a neighborhood, $U$ of $x_0$ with some constant $C$.  Then there exists a neighborhood, $W$ of $x_0$ so that $Pu$ satisfies
$$|(hD)^\beta Pu(x)|\leq h^{-k-{\red{n}}}C^{|\beta|+1}(\e+h|\beta|)^{|\beta|},\quad x\in W.$$
\end{prop}
\begin{proof}
Let $\chi \in C_c(\re)$ with $\chi\equiv 1$ on $[-1,1]$ and $\supp \chi \subset [-2,2]$. Then for any $\delta>0$, the kernel of $P$ is given by 
$K_\delta (x,y)+R_\delta(x,y)$ where $R_\delta$ is real analytic and
$$K_\delta (x,y)=(2\pi )^{-{\red{n}}}\int e^{i\lan x-y,\xi\ran} {\red{p}}(x,y,\xi)\chi((3\delta)^{-1}|x-y|)d\xi$$
with ${\red{p}}\in S^k_{ha}$.

The kernel of $\partial_x^\beta P$ is given by $\partial_x^\beta K+\partial_x^\beta R_\delta$.  Since $R_\delta$ is analytic, 
$$\left|\int \partial_x^\beta R_\delta (x,y)u(y)dy\right|\leq \|\partial_x^\beta R(x,y)\|_{L^2_y}\|u\|_{L^2}\leq C_\delta ^\beta \beta!\|u\|_{L^2},$$
so we need only consider $\partial_x^\beta K_\delta$. 

Observe that
$$\partial_x^\beta K_\delta (x,y)=(2\pi )^{-{\red{n}}}\int e^{i\lan x-y,\xi\ran }\sum_{\beta'+\beta''=\beta}\xi^{\beta'}\partial_{x}^{\beta''}({\red{p}}\chi((3\delta)^{-1}|x-y|))d\xi.$$
Deforming the contour in $\xi$ to 
$$\xi\mapsto \xi+iR\frac{x-y}{|x-y|}$$
for some $R>0$,
we can, modulo an analytic error, replace the kernel by
\begin{equation}
\label{e:kernel1}(2\pi )^{-{\red{n}}}\int e^{i\lan x-y,\xi\ran }\sum_{\beta'+\beta''=\beta}\xi^{\beta'}\partial_{x}^{\beta''}({\red{p}})\chi(\delta^{-1}|x-y|)d\xi.\end{equation}
Let $\partial_x^\beta \tilde{P}$ be the operator with kernel as in \eqref{e:kernel1}.
Then, let $\psi\in C_c^\infty(M)$ have support on a neighborhood, $W$ of $x_0$ so that $W\subset U$ and $d(W,\partial U)\geq 2\delta$ \red{ with $U$ as in the statement of the proposition}. \red{By~\cite[Theorem 4.23]{Zw} there exists $N>0$ a constant (independent of $h,\beta,P$) so that}
\begin{align*} \|\psi\partial_x^\beta \tilde{P}u\|_{L^2}&\leq \sum_{\beta'+\beta''=\beta}\|(\xi^{\beta'}\lan \xi\ran^{-|\beta'|-\red{k}}\partial_x^{\beta''}{\red{p}})(x,D)\|_{L^2\to L^2}\|u\|_{H^{k+|\beta'|}(U)}\\&\leq \sum_{\beta'+\beta''=\beta}\sum_{|\alpha|\leq {\red{Nn}}}\|\partial^\alpha(\xi^{\beta'}\lan \xi\ran^{-|\beta'|-k}\partial_x^{\beta''}{\red{p}})\|_{L^\infty}\|u\|_{H^{k+|\beta'|}(U)}\\
&\leq C\sum_{j=1}^{|\beta|}\sum_{|\alpha|\leq {\red{Nn}}}\frac{|\beta|!}{(|\beta|-j)!j!}C^{|\beta|-j+|\alpha|}(|\beta|-j+|\alpha|)!h^{-k-j}C^{k+j}(\e+h(k+j))^{k+j}\\
&\leq C\sum_{j=1}^{|\beta|}\sum_{|\alpha|\leq {\red{Nn}}}\frac{|\beta|!(\red{n}e)^{|\beta|}}{(|\beta|-j)^{|\beta|-j}j^j}C^{|\beta|-j+|\alpha|}\\
&\qquad(|\beta|-j+|\alpha|)^{|\beta|-j+|\alpha|}h^{-k-j}C^{k+j}(\e+h(k+j))^{k+j}
\end{align*}
\begin{align*}
%&\leq C\sum_{k=1}^{|\beta|}\sum_{|\alpha|\leq Mn}|\beta|!(ne)^{|\beta|}k^{m}C^{|\beta|-k+|\alpha|}\left(1+\frac{|\alpha|}{|\beta|-k}\right)^{|\beta|-k}\\
%&\qquad (|\beta|-k+|\alpha|)^{|\alpha|}C^{m+k}\left(1+\frac{\e+hm}{hk}\right)^{m+k}\\
%&\leq C\sum_{k=1}^{|\beta|}\sum_{|\alpha|\leq Mn}|\beta|!(ne)^{|\beta|}k^{m}C^{|\beta|-k+|\alpha|}e^{|\alpha|} (|\beta|-k+|\alpha|)^{|\alpha|}C^{m+k}\left(1+\frac{\e+hm}{h|\beta|}\right)^{|\beta|}h^{-m}\\
\|\psi\partial_x^\beta\tilde{P}u\|_{L^2}&\leq C^{|\beta|+1}|\beta|!(\red{n}e)^{|\beta|}\left(h+\frac{\e}{|\beta|}\right)^{|\beta|}h^{-|\beta|-k}\\
&\qquad\sum_{\substack{0\leq j\leq |\beta|\\|\alpha|\leq \red{Nn}}}j^{k}C^{|\beta|-j+|\alpha|}e^{|\alpha|} (|\beta|-j+|\alpha|)^{|\alpha|}C^{k+j}\\
&\leq C^{|\beta|+1}\left(h|\beta|+\e\right)^{|\beta|}h^{-|\beta|-k}
\end{align*}
Relabeling $|\beta|=|\beta|+{\red{n}}$ gives the desired estimates by using Sobolev embeddings. The result follows from taking $\psi\equiv 1$ on a slightly smaller neighborhood $W'\subset W$.
\end{proof}

Let $\chi\in \Cc(\re)$ with $\supp \chi\subset[-1,1]$, $\e>0$, and define
$$S^{\e}_{\tilde{h}}v(x):=2^{-n}(\pi h)^{-3n/2}\int e^{\frac{i}{\tilde{h}}\lan y-x,\eta\ran -\frac{1}{2\tilde{h}}|x-y|^2}\chi(\e^{-1}|\eta|)v(y,\eta)dyd\eta.$$
\begin{prop}
\label{l:cutoff}
Let $\chi \in \Cc(\re)$ with $\supp \chi\subset [-1,1]$. Then for for all $\e>0$, $|\xi|\geq 2\e$, and $0<\tilde{h}\leq h$, 
$$|T_{\tilde{h}}S_h^{\e}u(x,\xi)|\leq Ce^{-\frac{|\xi|^2}{2\tilde{h}}}\|u\|_{L^2}.$$
\end{prop}
\begin{proof}
The kernel of $T_{\tilde{h}}S_h^{\e}$ is given by 
\begin{align*}
&=2^{-n}(\pi h)^{-3n/2}\iint e^{\frac{i}{\tilde{h}}\left(\lan x-w,\xi\ran +\frac{\tilde{h}}{h}\lan w-y,\eta\ran\right)-\frac{1}{2\tilde{h}}\left((x-w)^2+\frac{\tilde{h}}{h}(w-y)^2\right)}\chi(\e^{-1}|\eta|)dw
\end{align*}
In particular, letting $\mu=\frac{\tilde{h}}{h}$, 
$$T_{\tilde{h}}S_h^{\e}v(x,\xi)=2^{-n}(\pi h)^{-3n/2}h^n\tilde{h}^{-n}\iint b(x,\xi,y,\eta)v(y,\eta)dyd\eta$$
where 
\begin{gather*} 
b(x,\xi,y,\eta)=e^{\frac{i}{\tilde{h}}\lan x-y,\xi+\mu^2 \eta\ran-\frac{\mu}{2(\mu+1)\tilde{h}}(x-y)^2-\frac{1}{2(\mu+1)\tilde{h}}(\mu\eta-\xi)^2}\chi(\e^{-1}|\eta|)\int e^{-\frac{\mu+1}{2\tilde{h}}\left(w-\left(\frac{\mu x+y+i(\mu \eta-\xi)}{\mu+1}\right)\right)^2}dw
\end{gather*}
Shifting contours shows that 
$$|b(x,\xi,y,\eta)|\leq \chi(\e^{-1}|\eta|)e^{-\frac{1}{2(\mu+1)\tilde{h}}(\mu \eta-\xi)^2}.$$
On $\supp\chi(\e^{-1}|\eta|)$, $|\eta|\leq  \e$ and $|\mu\eta|\leq \e$. Therefore, taking $|\xi|\geq 2\e$ proves the lemma.
\end{proof}

Our next proposition is the key proposition for this section and shows that homogeneous analytic pseudodifferential operators do not transport semiclassical analytic wavefront sets away from the zero-section.

We say that $u$ is \emph{compactly microlocalized} if there exists $\chi\in \Cc(\re)$ so that for some global FBI transform $\tilde{T}$, 
$$(1-\chi(|\alpha_{\xi}|_{\alpha_x}))\tilde{T}u=O(e^{-c/h}).$$
\begin{prop}
\label{p:moveWavefront}
Let $M$ be a compact, real analytic manifold, and $P$ be a homogeneous analytic pseudodifferential operator. Let $\tilde{T}$ be a global FBI transform on $M$ with left parametrix $\tilde{S}$. Let $\chi_0,\chi_1\in \Cc(\re)$ have $\chi_0\equiv 1$ on $[-1,1]$ with $\supp \chi_0\subset [-2,2]$ and $\chi_1\equiv 1$ on $[-{\red{N}},{\red{N}}]$ with $\supp \chi_1\subset [-2{\red{N}},2{\red{N}}]$.  Then, for $\e>0$ small enough and ${\red{N}}>0$ large enough, and $u$ compactly microlocalized, there exists $c>0$ so that
$$(1-\chi_1(\e^{-1}|\alpha_{\xi}|_{\alpha_x}))\tilde TQ\tilde S\chi_0(\e^{-1}|\alpha_{\xi}|_{\alpha_x})\tilde{T}u=O_{ C^\infty}(e^{-c/h}).$$
\end{prop}
\begin{proof}
Notice that
$$\tilde S\chi_0(\e^{-1}|\alpha_\xi|_{\alpha_x}))\tilde{T}u$$
is microlocally vanishing for $|\xi|\geq 2\e$. Therefore, for $\psi$ with small enough support so that the Euclidean FBI transforms are well-defined,
\begin{align*} 
\psi \tilde{S}\chi_0(\e^{-1}|\alpha_\xi|_{\alpha_x})\tilde{Tu}&=\psi S_hT_h\tilde{S}\chi_0(\e^{-1}|\alpha_\xi|_{\alpha_x})\tilde{T}u\\
&=\psi S_hT_h \tilde{S}\chi_0(\e^{-1}|\alpha_\xi|_{\alpha_x})\chi(|\alpha_{\xi}|_{\alpha_x})\tilde{T}u+O{}(e^{-c/h})\\
&=\psi S_h\chi_0((2\e)^{-1}|\xi|)T_h \tilde{S}\chi_0(\e^{-1}|\alpha_\xi|_{\alpha_x})\chi(|\alpha_{\xi}|_{\alpha_x})\tilde{T}u+O{}(e^{-c/h})
\end{align*}
Now, by \red{Proposition}~\ref{l:cutoff}, \eqref{e:FBI} holds for the image of $S_h\chi_0((2\e)^{-1}|\xi|)$. In particular, Lemma \ref{l:pseudo} applies and hence taking a partition of unity, 
$Q\tilde S\chi_0(\e^{-1}|\alpha_\xi|_{\alpha_x}))\tilde{T}u$ satisfies \eqref{e:cauchy} and hence, for any $\tilde{\psi}$ supported in a small enough region so that the Euclidean FBI transform is well defined, 
$$\tilde{\psi}Q\tilde S\chi_0(\e^{-1}|\alpha_\xi|_{\alpha_x}))\tilde{T}u$$
is microlocally vanishing on $|\xi|\geq 2\e^{-1}$ and since this is independent of the choice of FBI transform, taking $M$ large enough gives
$$(1-\chi_1(\e^{-1}|\alpha_{\xi}|_{\alpha_x}))\tilde T\tilde{\psi}Q\psi\tilde S\chi_0(\e^{-1}|\alpha_{\xi}|_{\alpha_x})\tilde{T}u=O_{ C^\infty}(e^{-c/h}).$$
Taking a partition of unity then proves the lemma. 
\end{proof}

\subsection{Application to eigenvalue problems}
\label{s:zeroSteklov}
Let $Q$ be an ($h$-independent) homogeneous, elliptic, classical analytic pseudodiffertial operator of order $k>0$.  Let $\varphi_h$ denote a solution to 
$$(h^kQ-1)\varphi_h=0.$$
We are now in a position to analyze the behavior of $\varphi_h$ near $|\xi|=0$. 

Notice that for any $\tilde \chi\in C_c^\infty(\re^d)$ with $\tilde \chi \equiv 1$ near 0, and $\supp \tilde \chi \subset B(0,1)$,
$(1-\tilde \chi(hD))h^kQ$ is a semiclassical pseudodifferential operator whose symbol is analytic in the interior of $\{\tilde \chi= 0\}$. 
Thus, since $Q$ is elliptic, 
$$\sigma(h^k(1-\tilde \chi(hD))Q)|\geq c|\xi|^k,\quad |\xi|\geq 1$$
and hence by Proposition \ref{p:longRange} for $\chi\in \Cc(\re)$ with $\chi \equiv 1$ on $[-2,2]$, for any $\delta>0$, there exists $c>0$ so that
\begin{equation}
\label{e:longRange}
(1-\chi(\delta^{-1}q(\alpha_x,\alpha_\xi))(1-\chi(\delta^{-1}|\alpha_\xi|_{\alpha_x}))T_{geo}\varphi_h=O{}(e^{-c/h}).
\end{equation}

In the next proposition, we estimate the eigenfunctions $\varphi_h$ near the zero section by their norm in a small annulus around the zero section. In particular, this shows that the cutoff $(1-\chi(\delta^{-1}|\alpha_\xi|_{\alpha_x}))$ can be removed from \eqref{e:longRange}.

\begin{prop} \label{zsection}
Let $\varphi_h$ be a solution to 
\begin{equation}
\label{e:efuncEq}(h^kQ-1)\varphi_h=O_{L^2}(e^{-c/h}),\quad \|\varphi_h\|_{L^2}=1.
\end{equation}
For $\e>0$, let $\chi_\e=\chi(\e^{-1}|\alpha_{\xi}|_{\alpha_x}).$
Then for $\e>0$ small enough, there exists $c>0$ so that 
$$\|\chi_\e T_{geo}\varphi_h\|_{L^2}=O{}(e^{-c/h}).$$
\end{prop}
\begin{proof}

We first use the apriori estimate \eqref{e:longRange} together with Proposition \ref{p:moveWavefront} to decompose the eigenfunction equation \eqref{e:efuncEq} into two (essentially) orthogonal pieces - one supported near and one supported away from the zero section. We are then able to use a simple Neumann series argument to obtain estimates near the zero section. 

Fix $N>0$ to be chosen large enough later.
\begin{align}
O(e^{-c/h})&=T_{geo}(h^kQ-1)S_{geo}T_{geo}\varphi_h\nonumber=T_{geo}(h^kQ-1)S_{geo}(\chi_\e+(1-\chi_\e))T_{geo}\varphi_h\nonumber\\
&=T_{geo}(h^kQ-1)S_{geo}(\chi_\e+(1-\chi_{2^{3N}\e}))T_{geo}\varphi_h\label{e:cutMiddle}+O(e^{-c/h})\\
%&=T_{geo}(h^kQ-1)S_{geo}\chi_\e T_{geo}\varphi_h+(1-\chi_{2^{2N}\e})T_{geo}(h^kQ-1)S_{geo}(1-\chi_{2^{3N}\e})T_{geo}\varphi_h\nonumber\\
&\begin{aligned}&=(\chi_{2^{N}\e}T_{geo}h^kQS_{geo}\chi_{2\e}-1)\chi_\e T_{geo}\varphi_h\\
&\qquad+(1-\chi_{2^{2N}\e})T_{geo}(h^kQ-1)S_{geo}(1-\chi_{2^{3N}\e})T_{geo}\varphi_h+O(e^{-c/h})
\end{aligned}\label{e:decomposed}
\end{align}
In \eqref{e:cutMiddle}, we used \eqref{e:longRange} to see that for any $\psi$ with 
\begin{gather*} \supp \psi \subset \{|\sigma(q)(\alpha_x,\alpha_\xi)|<1\}\cap \{|\alpha_\xi|_{\alpha_x}>0\},\\
\psi T_{geo}\varphi_h=O{}(e^{-c/h}).
\end{gather*}
In \eqref{e:decomposed}, we choose $N$ large enough so that Proposition \ref{p:moveWavefront} implies
$$(1-\chi_{2^{N}\e})T_{geo}h^kQS_{geo}\chi_{2\e},\, \chi_{2^{2N}\e}T_{geo}h^kQS_{geo}(1-\chi_{2^{3N}\e})=O(e^{-c/h}).$$

In order to invert $(\chi_{2^{N}\e}T_{geo}h^kQS_{geo}\chi_{2\e}-1)$ by Neumann series, we obtain estimates on $\chi_{2^{N}\e}T_{geo}h^kQS_{geo}\chi_{2\e}$.
Since $S_{geo}\chi_{2\e}T_{geo}$ is the anti-Wick quantization of $\chi_{2\e}$, it is a pseudodifferential operator with symbol $\chi_{2\e}(\alpha),$ 
$$\|S_{geo}\chi_{2\e}T_{geo}\|_{L^2\to L^2}\leq 1+O{}(h),\quad \|(h\nabla)^k S_{geo}\chi_{2\e}T_{geo}\|_{L^2\to L^2}\leq C\epsilon^k +O{}(h).$$
Therefore, $\|Q\|_{H^k\to L^2}\leq C$ implies
$$\|T_{geo}h^kQS_{geo}\chi_{2\e} T_{geo}\|_{L^2\to L^2}\leq (C\e^k+O{}(h)).$$
In particular, 
\begin{equation}
\label{e:inverseNeumann}(\chi_{2^{N}\e}T_{geo}h^kQS_{geo}\chi_{2\e}-1)^{-1}=-\sum_{k=0}^\infty(\chi_{2^{N}\e}T_{geo}h^kQS_{geo}\chi_{2\e})^k.
\end{equation}
Hence, applying \eqref{e:inverseNeumann} on the left of \eqref{e:decomposed}
$$\chi_\e T_{geo}\varphi_h +(1-\chi_{2^{2N}\e})T_{geo}(h^kQ-1)S_{geo}(1-\chi_{2^{3N}\e})T_{geo}\varphi_h=O{}(e^{-c/h})$$
and, multiplying by $\chi_{\e/2}$ on the left, we have
$$\chi_{\e/2}T_{geo}\varphi_h=O{}(e^{-c/h}).$$
\end{proof}
%\end{comment}

%%%%%%%%%%%%%%%%%%%%%%%%%%%%%%%%%%%%%%%%%%%%%%%%%%%%%%%%%%%%%%%%%%%%%%%%%%%%%%%%%%%%
%%%%%%%%%%%%%%%%%%%%%%%%%%%%%%%%%%%%%%%%%%%%%%%%%%%%%%%%%%%%%%%%%%%%%%%%%%%%%%%%%%%%

\bibliography{biblio}

\begin{thebibliography}{GPPS14}

\bibitem[BL15]{BLin}
Katar{\'{\i}}na Bellov{\'a} and Fang-Hua Lin.
\newblock Nodal sets of {S}teklov eigenfunctions.
\newblock {\em Calc. Var. Partial Differential Equations}, 54(2):2239--2268,
  2015.

\bibitem[Bla10]{BlairSasaki}
David~E. Blair.
\newblock {\em Riemannian geometry of contact and symplectic manifolds}, volume
  203 of {\em Progress in Mathematics}.
\newblock Birkh\"auser Boston, Inc., Boston, MA, second edition, 2010.

\bibitem[GP17]{GP}
Alexandre Girouard and Iosif Polterovich.
\newblock Spectral geometry of the {S}teklov problem (survey article).
\newblock {\em J. Spectr. Theory}, 7(2):321--359, 2017.

\bibitem[GPPS14]{GPPS}
Alexandre Girouard, Leonid Parnovski, Iosif Polterovich, and David~A. Sher.
\newblock The {S}teklov spectrum of surfaces: asymptotics and invariants.
\newblock {\em Math. Proc. Cambridge Philos. Soc.}, 157(3):379--389, 2014.

\bibitem[GS91]{GS1}
Victor Guillemin and Matthew Stenzel.
\newblock Grauert tubes and the homogeneous {M}onge-{A}mp\`ere equation.
\newblock {\em J. Differential Geom.}, 34(2):561--570, 1991.

\bibitem[HL01]{HL}
P.~D. Hislop and C.~V. Lutzer.
\newblock Spectral asymptotics of the {D}irichlet-to-{N}eumann map on multiply
  connected domains in {$\Bbb R^d$}.
\newblock {\em Inverse Problems}, 17(6):1717--1741, 2001.

\bibitem[HS86]{HeSj}
B.~Helffer and J.~Sj{\"o}strand.
\newblock R\'esonances en limite semi-classique.
\newblock {\em M\'em. Soc. Math. France (N.S.)}, (24-25):iv+228, 1986.

\bibitem[Jin17]{Ji}
Long Jin.
\newblock Semiclassical {C}auchy estimates and applications.
\newblock {\em Trans. Amer. Math. Soc.}, 369(2):975--995, 2017.

\bibitem[Leb13]{Leb}
Gilles Lebeau.
\newblock The complex poisson kernel on a compact analytic riemannian manifold.
\newblock {\em preprint}, 2013.

\bibitem[LGS96]{GLS}
Eric Leichtnam, Fran{\c{c}}ois Golse, and Matthew Stenzel.
\newblock Intrinsic microlocal analysis and inversion formulae for the heat
  equation on compact real-analytic {R}iemannian manifolds.
\newblock {\em Ann. Sci. \'Ecole Norm. Sup. (4)}, 29(6):669--736, 1996.

\bibitem[LS91]{LS}
L{\'a}szl{\'o} Lempert and R{\'o}bert Sz{\H{o}}ke.
\newblock Global solutions of the homogeneous complex {M}onge-{A}mp\`ere
  equation and complex structures on the tangent bundle of {R}iemannian
  manifolds.
\newblock {\em Math. Ann.}, 290(4):689--712, 1991.

\bibitem[Mar94]{Mar2}
Andr{\'e} Martinez.
\newblock Estimates on complex interactions in phase space.
\newblock {\em Math. Nachr.}, 167:203--254, 1994.

\bibitem[Mar97]{Mar}
A.~Martinez.
\newblock Microlocal exponential estimates and applications to tunneling.
\newblock In {\em Microlocal analysis and spectral theory ({L}ucca, 1996)},
  volume 490 of {\em NATO Adv. Sci. Inst. Ser. C Math. Phys. Sci.}, pages
  349--376. Kluwer Acad. Publ., Dordrecht, 1997.

\bibitem[Nak95]{Na}
Shu Nakamura.
\newblock On {M}artinez' method of phase space tunneling.
\newblock {\em Rev. Math. Phys.}, 7(3):431--441, 1995.

\bibitem[PST15]{PST}
Iosif Polterovich, David~A Sher, and John~A Toth.
\newblock Nodal length of \uppercase{S}teklov eigenfunctions on real-analytic
  \uppercase{R}iemannian surfaces.
\newblock {\em arXiv preprint arXiv:1506.07600}, 2015.

\bibitem[Sha71]{Sh}
S.~E. Shamma.
\newblock Asymptotic behavior of {S}tekloff eigenvalues and eigenfunctions.
\newblock {\em SIAM J. Appl. Math.}, 20:482--490, 1971.

\bibitem[Sj{\"o}82]{SjA}
Johannes Sj{\"o}strand.
\newblock Singularit\'es analytiques microlocales.
\newblock In {\em Ast\'erisque, 95}, volume~95 of {\em Ast\'erisque}, pages
  1--166. Soc. Math. France, Paris, 1982.

\bibitem[Sj{\"o}96]{Sj}
Johannes Sj{\"o}strand.
\newblock Density of resonances for strictly convex analytic obstacles.
\newblock {\em Canad. J. Math.}, 48(2):397--447, 1996.
\newblock With an appendix by M. Zworski.

\bibitem[SU16]{SU}
Johannes Sj{\"o}strand and Gunther Uhlmann.
\newblock Local analytic regularity in the linearized {C}alder\'on problem.
\newblock {\em Anal. PDE}, 9(3):515--544, 2016.

\bibitem[SWZ16]{SWZ}
Christopher~D. Sogge, Xing Wang, and Jiuyi Zhu.
\newblock Lower bounds for interior nodal sets of {S}teklov eigenfunctions.
\newblock {\em Proc. Amer. Math. Soc.}, 144(11):4715--4722, 2016.

\bibitem[Tay11]{Tayl2}
Michael~E. Taylor.
\newblock {\em Partial differential equations {II}. {Q}ualitative studies of
  linear equations}, volume 116 of {\em Applied Mathematical Sciences}.
\newblock Springer, New York, second edition, 2011.

\bibitem[Tot98]{T}
John Toth.
\newblock Eigenfunction decay estimates in the quantum completely integrable
  cas.
\newblock {\em Duke Math. J.}, 93(2):231--255, 1998.

\bibitem[Yau82]{Yau}
Shing~Tung Yau.
\newblock Survey on partial differential equations in differential geometry.
\newblock In {\em Seminar on {D}ifferential {G}eometry}, volume 102 of {\em
  Ann. of Math. Stud.}, pages 3--71. Princeton Univ. Press, Princeton, N.J.,
  1982.

\bibitem[Yau93]{Yau2}
Shing-Tung Yau.
\newblock Open problems in geometry.
\newblock In {\em Differential geometry: partial differential equations on
  manifolds ({L}os {A}ngeles, {CA}, 1990)}, volume~54 of {\em Proc. Sympos.
  Pure Math.}, pages 1--28. Amer. Math. Soc., Providence, RI, 1993.

\bibitem[Zel12]{Z}
Steve Zelditch.
\newblock Pluri-potential theory on {G}rauert tubes of real analytic
  {R}iemannian manifolds, {I}.
\newblock In {\em Spectral geometry}, volume~84 of {\em Proc. Sympos. Pure
  Math.}, pages 299--339. Amer. Math. Soc., Providence, RI, 2012.

\bibitem[Zel15]{Ze}
Steve Zelditch.
\newblock Hausdorff measure of nodal sets of analytic {S}teklov eigenfunctions.
\newblock {\em Math. Res. Lett.}, 22(6):1821--1842, 2015.

\bibitem[Zhu15]{Zh}
Jiuyi Zhu.
\newblock Doubling property and vanishing order of {S}teklov eigenfunctions.
\newblock {\em Comm. Partial Differential Equations}, 40(8):1498--1520, 2015.

\bibitem[Zhu16]{Zhu}
Jiuyi Zhu.
\newblock Interior nodal sets of {S}teklov eigenfunctions on surfaces.
\newblock {\em Anal. PDE}, 9(4):859--880, 2016.

\bibitem[Zwo12]{Zw}
Maciej Zworski.
\newblock {\em Semiclassical analysis}, volume 138 of {\em Graduate Studies in
  Mathematics}.
\newblock American Mathematical Society, Providence, RI, 2012.

\end{thebibliography}
\bibliographystyle{alpha}

\end{document}